\documentclass[12pt]{amsart}

\usepackage[english]{babel}

\usepackage{mathrsfs}
\usepackage{amsmath,amssymb,mathtools}
\usepackage{epigraph}
\usepackage[all]{xy}
\usepackage{bm}
\usepackage{mathptmx}

\newcommand{\N}{\mathbf N}
\newcommand{\Z}{\mathbf Z}
\newcommand{\Qp}{\mathbf Q_p}
\newcommand{\Ph}{\varphi}

\newcommand\G{{\mathrm{Gal}}}
\newcommand{\Dc}{\mathbf{D}_{\mathrm{cris}}}
\newcommand{\La}{\varLambda}
\newcommand{\Dd}{\mathbf{D}_{\mathrm {dR}}}

\newcommand{\Dpst}{\mathbf{D}_{\mathrm{pst}}}
\newcommand{\Dst}{\mathbf{D}_{\mathrm{st}}}

\newcommand{\RG}{\mathbf{R}\Gamma}

\newcommand{\res}{\mathrm{res}}
\newcommand{\R}{\mathbf{R}}
\newcommand{\boB}{\mathbf{B}}

\newcommand{\A}{\mathbf{A}}
\newcommand{\bE}{\mathbf{E}}
\newcommand{\E}{\mathbf E}
\newcommand{\bD}{\mathbf{D}}

\newcommand{\cl}{\mathrm{cl}}

\newcommand{\Tot}{\mathrm{Tot}}
\newcommand{\Iw}{\mathrm{Iw}}

\newcommand{\Bc}{\mathbf{B}_{\mathrm{cris}}}

\newcommand{\Hom}{\mathrm{Hom}}
\newcommand{\Ext}{\mathrm{Ext}}

\newcommand{\F}{\mathrm{Fil}}
\newcommand{\g}{\gamma}
\newcommand{\Ind}{\mathrm{Ind}}

\newcommand{\Hi}{H_{\mathrm{Iw}}}

\newcommand{\Ddagrig}{\mathbf{D}^{\dagger}_{\mathrm{rig}}}

\newcommand{\iso}{\overset{\sim}{\rightarrow}}

\newcommand{\Rep}{\mathbf{Rep}}
\newcommand{\Gal}{\mathrm{Gal}}

\newcommand{\CR}{\mathcal{R}}
\newcommand{\CDcris}{\mathscr{D}_{\mathrm{cris}}}

\newcommand{\ep}{\varepsilon}

\newcommand{\Zp}{\mathbf Z_p}
\newcommand{\Q}{\mathbf Q}
\newcommand{\CH}{\mathscr H}
\newcommand{\BrdA}{\widetilde{\mathbf B}_{\mathrm{rig},A}^{\dagger}}

\newcommand{\CDst}{\mathcal D_{\mathrm{st}}}
\newcommand{\CDpst}{\mathcal D_{\mathrm{pst}}}
\newcommand{\CDdr}{\mathcal D_{\mathrm{dR}}}

\newcommand{\CO}{\mathcal O}
\newcommand{\DdagrigA}{\bD^{\dagger}_{\mathrm{rig},A}}
\newcommand{\DstL}{\bD_{\mathrm{st}/L}}
\newcommand{\DdrL}{\bD_{\mathrm{dR}/L}}

\newcommand\Fr{\textrm{Fr}}

\newcommand\dR{\mathrm{dR}}
\newcommand\st{\mathrm{st}}
\newcommand\cris{\mathrm{cris}}

\newcommand{\W}{\mathbf W}
\newcommand{\bM}{\mathbf M}
\newcommand\im{\mathrm{Im}}

\newcommand\cyc{\mathrm{cyc}}
\newcommand\sel{\mathrm{sel}}

\newcommand{\X}{\mathbf{X}}
\newcommand\Vrigdag{V^\dagger_{\mathrm{rig}}}
\newcommand\pr{\mathrm{pr}}
\newcommand{\DdagrigE}{\bD^{\dagger}_{\mathrm{rig},E}}
\newcommand{\id}{\mathrm{id}}
\newcommand{\ft}{\mathrm{ft}}
\newcommand{\perf}{\mathrm{perf}}
\newcommand{\ur}{\mathrm{ur}}
\newcommand{\spl}{\mathrm{spl}}
\newcommand{\Hodge}{\mathrm{Hodge}}
\newcommand{\norm}{\mathrm{norm}}
\newcommand{\inv}{\mathrm{inv}}
\newcommand{\rk}{\mathrm{rk}}

\newcommand{\fq}{\mathfrak q}

\usepackage[T1]{fontenc}

\usepackage{etoolbox}
\patchcmd{\section}{\scshape}{\bfseries\scshape}{}{}
\patchcmd{\subsection}{\scshape}{\bfseries}{}{}
\makeatletter
\renewcommand{\@secnumfont}{\bfseries}
\makeatother

\newtheorem*{definition}{Definition}
\newtheorem{mytheorem}[subsubsection]{Theorem}
\newtheorem{myproposition}[subsubsection]{Proposition}

\newtheorem{mylemma}[subsubsection]{Lemma}
\newtheorem{mycorollary}[subsubsection]{Corollary}
\newtheorem*{theoremI}{Theorem I}
\newtheorem*{theoremII}{Theorem II}
\newtheorem*{theoremIII}{Theorem III}
\newtheorem*{theoremIV}{Theorem IV}

\makeindex             % used for the subject index
\begin{document}
\title{$P$-adic heights and  $P$-adic Hodge theory}
\author{Denis Benois}
\address{Institut de Math\'ematiques,
Universit\'e de  Bordeaux, 351, cours de la Lib\'eration  33405
Talence, France} 
\email{denis.benois@math.u-bordeaux1.fr}
\date{December 18, 2014}

\begin{abstract} Using the theory of $(\Ph,\Gamma)$-modules and the formalism of Selmer complexes we construct the $p$-adic height pairing
for $p$-adic representations with coefficients in an affinoid algebra
over $\Qp.$ For $p$-adic representations that are potentially semistable
at $p,$ we relate our contruction to universal norms and compare it 
to the $p$-adic height pairing of Nekov\'a\v r. 
 \end{abstract}

\subjclass{2000 Mathematics Subject Classification. 11R23, 11F80, 11S25, 11G40}

\maketitle
\tableofcontents 
\section*{Introduction}

\subsection{Selmer complexes} 
\subsubsection{}
Let $F$ be a  number field. We denote by $S_f$ and $S_\infty$  the set of non-archimedean and archimedian  places of $F$ respectively. Fix a prime number $p$
and   denote by $S_p$ the set of places $\fq$ above $p.$ 

Let $S$ be  a finite set of non-archimedian places of $F$ containing $S_p.$ 
To simplify notation, set $\Sigma_p=S\setminus S_p.$ We denote by $G_{F,S}$ the Galois group of the maximal algebraic extension of $F$ unramified outside $S\cup S_\infty.$ For each $\fq \in S$ we denote by $F_\fq$ the completion of $F$ with respect to $\fq$ and by $G_{F_\fq}$ the absolute Galois group of $F_\fq$ which we identify with the decomposition group at $\fq.$ We will write $I_\fq$ for the 
inertia subgroup of $G_{F_\fq}$ and $\Fr_\fq$ for the relative Frobenius over $F_\fq.$

If $G$ is a topological group and $M$ is a topological $G$-module, we denote by
$C^{\bullet}(G,M)$  the complex of continuous cochains of $G$ with coefficients
in $M.$ 

\subsubsection{}
Let $T$ be a continuous representation of $G_{F,S}$ with coefficients in a complete local  noetherian ring $A$ with a finite residue field
of characteristic $p.$ 
%To simplify the discussion, assume that $A$ is a Gorenstein ring.
A local condition at $\fq \in S$ is a morphism of complexes
\[
g_\fq \,:\,U^{\bullet}_{\fq} (T) \rightarrow C^{\bullet} (G_{F_\fq},T).
\]
To each collection $U^{\bullet}(T)=(U_\fq^{\bullet}(T), g_\fq)_{\fq \in S}$ of local conditions
one can associate the following diagram
\begin{equation}
\label{diagram selmer complexes}
\xymatrix{
C^{\bullet}(G_{F,S},T) \ar[r]^{} &\underset{\fq \in S}\bigoplus C^{\bullet}(G_{F_\fq},T)\\
& \underset{\fq \in S}\bigoplus U^{\bullet}_{\fq}(T) \ar[u]_(.4){(g_\fq)},
}
\end{equation}
where the upper row is the restriction map. The Selmer complex associated to 
the local conditions $U^{\bullet}(T)$ is defined 
as the  mapping cone 
\begin{equation*}
S^{\bullet}(T,U^{\bullet}(T))=\mathrm{cone} \left (
C^{\bullet}(G_{F,S},T) \oplus \left (\underset{\fq \in S}\bigoplus U^{\bullet}_{\fq}(T)\right )\rightarrow \underset{\fq \in S}\bigoplus C^{\bullet}(G_{F_\fq},T)
\right ) [-1]. 
\end{equation*}
This notion was  introduced by Nekov\'a\v r in  \cite{Ne06},
where the machinery of Selmer complexes was  developed in 
full generality.

\subsubsection{}
The most important example of local conditions is provided by Greenberg's 
local conditions. If $\fq \in S,$ we will denote by $T_\fq$ the restriction of  $T$ on $G_{F_\fq}.$ Fix, for each $\fq \in S_p,$ a subrepresentation
$M_\fq$ of $T_\fq$ and define 
\begin{equation*}
\label{greenberg local conditions}
U^{\bullet}_{\fq}(T)=C^{\bullet}(G_{F_\fq}, M_\fq) \qquad \fq \in S_p.
\end{equation*}
For $\fq \in \Sigma_p$ we consider the unramified local conditions
\begin{equation*}
\label{unramified local conditions}
U^{\bullet}_{\fq}(T)=C_{\ur}^{\bullet}(T_\fq)=\left [ T^{I_\fq} \xrightarrow{\Fr_\fq-1} T^{I_\fq} \right ],
\end{equation*} 
where the  terms are placed in degrees $0$ and $1.$
To simplify notation, we will write $S^{\bullet}(T,M)$ for 
the Selmer complex  associated to these conditions and  $\RG (T,M)$  for the corresponding object of the derived category. 
Let $T^*(1)$ be  the Tate dual of $T$ 
equipped with Greenberg local conditions $N=(N_\fq)_{\fq \in S_p}$ such  that $M$ and $N$ are orthogonal  to each other under 
the canonical duality $T\times T^*(1) \rightarrow A.$ 
In this case, the general construction of cup products
for cones gives a pairing
\begin{equation*}
\cup \,:\,\RG (T,M) \otimes_A^{\mathbf{L}}  \RG (T^*(1),N)\rightarrow A [-3]
\end{equation*} 
(see \cite{Ne06}, Section~6.3).
Nekov\'a\v r constructed the $p$-adic height pairing 
\begin{equation*} 
h^{\sel}\,:\,\RG (T,M)\otimes_A^{\mathbf{L}}  \RG (T^*(1),N) 
\rightarrow A [-2]
\end{equation*} 
as the composition of $\cup$ with the Bockstein map
\footnote{See \cite{Ne06}, Section~11.1 or Section~3.2 below for the definition of the Bockstein map.} 
$\beta_{T,M}\,:\,\RG (T,M)\rightarrow \RG (T,M)[1]$:
\begin{equation*}
h^{\sel} (x,y)=\beta_{T,M}(x)\cup y.
\end{equation*}
Passing to  cohomology groups $H^i(T,M)=\R^i\Gamma (T,M),$
we obtain a pairing
\begin{equation*} 
h_{1}^{\sel}\,:\,H^1(T,M)\otimes_A  H^1 (T^*(1),N) 
\rightarrow A.
\end{equation*} 

\subsubsection{} The relationship of these constructions to the 
traditional treatements is the following. Let $A=\CO_E$ be the ring of integers of a local field $E/\Qp$  and let $T$ be a Galois stable $\CO_E$-lattice of a $p$-adic Galois representation $V$ with coefficients in $E.$ Assume that $V$ is semistable at all $\fq \in S_p.$  We say that  $V$ satisfies the Panchishkin
condition at $p$ if, for each $\fq \in S_p,$ there exists a subrepresentation
$V_\fq^+\subset V_\fq$ such that all Hodge--Tate weights
\footnote{We call Hodge--Tate weights the jumps of the Hodge--Tate filtration 
on the associated de Rham module.} 
 of $V_\fq/V_\fq^+$ are  $\geqslant 0.$ Set $T_\fq^+=T\cap V_\fq^+,$
$T^+=(T^+_\fq)_{\fq \in S_p}$ and consider the associated Selmer
complex $\RG (T,T^+).$
The first cohomology group $H^1(T, T^+)$  of $\RG (T,T^+)$   is very close to the Selmer group defined by Greenberg \cite{Gr89}, \cite{Gr94b} and therefore to the Bloch--Kato Selmer group \cite{Fl90}. It can be shown (see \cite{Ne06}, Chapter 11), that, 
under some mild conditions, the pairing $h_1^\sel$ coincides  with the $p$-adic height 
pairing constructed in \cite{Ne92} and  \cite{PR92} using universal norms (see also \cite{Sch82}).
Note, that  Nekov\'a\v r's construction has many advantages over the classical definitions. In particular, it allows to study the variation of the $p$-adic heights in ordinary families
of $p$-adic representations (see \cite{Ne06}, Section~0.16 and 
Chapter~11 for further discussion).

\subsection{Selmer complexes and $(\Ph,\Gamma)$-modules}

\subsubsection{} In this paper we consider Selmer complexes with local conditions coming   from the theory of  $(\Ph,\Gamma)$-modules and extend a part of Nekov\'a\v r's machinery to this situation. Let now $A$ be a $\Qp$-affinoid algebra and let $V$ be a $p$-adic representation of $G_{F,S}$ with coefficients in $A.$
In \cite{Po13}, Pottharst studied Selmer complexes associated to 
the diagrams of the form (\ref{diagram selmer complexes}) in this context. We will consider a slightly more general situation because, 
for the local conditions $U_\fq^{\bullet} (V)$ that we have in mind, the  maps $g_\fq\,:\,U^{\bullet}_\fq (V) \rightarrow  C^{\bullet}(G_{F_\fq},V)$ 
are not defined on the level of complexes but only in the derived category of $A$-modules.

 For each $\fq \in S_p$ we denote by $\Gamma_\fq$ the Galois group
of the cyclotomic $p$-extension of $F_\fq.$ As before, we denote by $V_\fq$ the restriction of $V$ on the decomposition group at 
$\fq.$  The theory of $(\Ph,\Gamma)$-modules associates to $V_\fq$
a finitely generated projective module $\DdagrigA (V)$ over the Robba ring $\CR_{F_\fq,A}$ equipped with a semilinear Frobenius map $\Ph$ and a continuous action of  $\Gamma_\fq$ which commute to each other (see \cite{Fo90}, \cite{CC1}, \cite{Cz08}, \cite{KL10}).
In \cite{KPX}, Kedlaya, Pottharst and Xiao extended the results 
of  Liu \cite{Li07} about the cohomology of  $(\Ph,\Gamma)$-modules  to the relative case. Their results play a key role in this paper. 

Namely, to each $(\Ph,\Gamma_\fq)$-module $\bD$ over 
$\CR_{F_\fq,A}$ one can associate the Fontaine--Herr complex $C^{\bullet}_{\Ph,\g_\fq}(\bD)$ 
of $\bD.$ The cohomology  $H^*(\bD)$ of $\bD$ is defined as the cohomology of  $C^{\bullet}_{\Ph,\g_\fq}(\bD).$
 If $\bD=\DdagrigA (V),$ there exist isomorphisms 
$H^*(\DdagrigA (V))\simeq H^*(F_\fq,V),$ but 
 the complexes $C^{\bullet}_{\Ph,\g_\fq}(\Ddagrig (V))$ and
 \linebreak 
 $C^{\bullet}(G_{F_\fq}, V_\fq)$  are not quasi-isomorphic.
A simple argument allows us to construct a complex
$K^{\bullet}(V_\fq)$  together with quasi-isomorphisms
\linebreak 
 $\xi_\fq \,:\,C^{\bullet}(G_{F_\fq},V) \rightarrow K^{\bullet}(V_\fq)$ and $\alpha_{\fq}\,:\,C^{\bullet}_{\Ph,\g_\fq}(\DdagrigA (V_\fq))
 \rightarrow K^{\bullet}(V_\fq)$
\footnote{This complex was first introduced in \cite{Ben15}}.
For each $\fq \in S_p,$ we choose a $(\Ph,\Gamma_\fq)$-submodule 
$\bD_\fq$ of $\DdagrigA (V_\fq)$ that is a $\CR_{F_\fq,A}$-module 
direct summand of $\DdagrigA (V_\fq)$ and set $\bD=(\bD_\fq)_{\fq\in S_p}.$  Set 
\[
K^{\bullet}(V)=\left (\underset{\fq \in\Sigma_p}\bigoplus C^{\bullet}(G_{F_\fq},V)\right )\bigoplus 
\left (\underset{\fq \in S_p}\bigoplus K^{\bullet} (V_\fq)\right )
\]
and 
\begin{equation}
\nonumber 
U_\fq^{\bullet}(V,\bD)= \begin{cases} 
C^{\bullet}_{\Ph,\g_\fq}(\bD_\fq), & \textrm{if $\fq\in S_p$,}\\
C^{\bullet}_{\ur}(V_\fq), &\textrm{if $\fq\in \Sigma_p.$}
\end{cases}
\end{equation}
For each $\fq \in S_p,$ we have morphisms
\begin{align*}
&f_\fq\,:\,C^{\bullet}(G_{F,S},V)\xrightarrow{\res_\fq}
C^{\bullet}(G_{F_\fq},V) \xrightarrow{\xi_\fq} K^{\bullet}(V_\fq),\\
&g_\fq \,:\,U_\fq^{\bullet}(V,\bD) \xrightarrow{} 
C^{\bullet}_{\Ph,\g_\fq}(\DdagrigA (V_\fq))\xrightarrow{\alpha_\fq}
K^{\bullet}(V_\fq).
\end{align*}
If $\fq \in\Sigma_p,$ we define the maps $f_\fq\,:\,C^{\bullet}(G_{F,S},V) \rightarrow C^{\bullet}(G_{F_\fq},V)$ and 
$g_\fq \,:\,C^{\bullet}_{\ur}(V_\fq)\rightarrow C^{\bullet}(G_{F_\fq},V)$  exactly as in the case of Greenberg local conditions. 
Consider the diagram
\begin{equation*}
\nonumber 
\xymatrix{
C^{\bullet}(G_{F,S},V) \ar[r]^{(f_\fq)_{\fq\in S}} &K^{\bullet}(V) \\
&\underset{\fq\in S}\bigoplus U_\fq^{\bullet}(V,\bD). 
\ar[u]^{\underset{\fq\in S}\oplus g_\fq}
}
\end{equation*}
We denote by $S^{\bullet}(V,\bD)$ the Selmer complex associated to this diagram and by $\RG (V,\bD)$ the corresponding object in the derived category of $A$-modules. 
\newpage
\subsection{$p$-adic height pairings}
\subsubsection{} Let $V^*(1)$ be the Tate dual of $V.$ We equip 
$V^*(1)$ with orthogonal local conditions $\bD^{\perp}$ setting
\[
\bD_\fq^{\perp}=\Hom_{\CR_{F_\fq,A}} \bigl (\DdagrigA (V_\fq)/\bD_\fq, \CR_{F_\fq,A} \bigr ),
\qquad \fq \in S_p.
\]
The general machinery gives us a cup product pairing
\begin{equation*}
\cup_{V,\bD}\,:\,\RG (V,\bD)\otimes_A^{\mathbf{L}}\RG (V^*(1),\bD^{\perp}) \rightarrow A[-3].
\end{equation*}
This allows us to construct the $p$-adic height pairing exactly in the same way as in the case of Greenberg local conditions. 

\begin{definition} The $p$-adic height pairing associated to the data $(V,\bD)$ 
is defined as the morphism 
\begin{multline*}
\label{definition of height via selmer}
h_{V,\bD}^{\sel}\,:\,\RG(V,\bD) \otimes_A^{\mathbf L} \RG(V^*(1),\bD^{\perp})
\xrightarrow{\delta_{V,\bD}}\\
\rightarrow \RG(V,\bD) [1] \otimes_A^{\mathbf L} \RG(V^*(1),\bD^{\perp}) \xrightarrow{\cup_{V,\bD}} A[-2],
\end{multline*}
where $\delta_{V,\bD}$ denotes the Bockstein map.
\end{definition}

The height pairing  $h_{V,\bD,M}^{\sel}$ induces a pairing on cohomology groups 
\begin{equation*}
h_{V,\bD,1}^{\sel}\,\,:\,\,H^1(V,\bD)\times H^1(V^*(1),\bD^{\perp}) \rightarrow A.
\end{equation*}
Applying the machinery of Selmer complexes, we obtain the following result (see Theorem~\ref{theorem symmetricity of h^sel} below). 
\begin{theoremI} 
We have a commutative diagram
\begin{equation}
\nonumber
\xymatrix{
\RG(V,\bD) \otimes_A^{\mathbf L} \RG(V^*(1),\bD^{\perp}) 
\ar[rr]^(.7){h_{V,\bD}^{\sel}} \ar[d]^{s_{12}} &&A[-2]\ar[d]^{=}\\
\RG(V^*(1),\bD^{\perp}) \otimes_A^{\mathbf L} \RG(V,\bD)
\ar[rr]^(.7){h_{V^*(1),\bD^{\perp}}^{\sel}}
&&A[-2],
}
\end{equation}
where $s_{12}(a\otimes b)=-(1)^{\deg (a)\deg(b)} b\otimes a.$  In particular, the pairing $h^{\sel}_{V,\bD,1}$ is skew symmetric.
\end{theoremI}

\subsubsection{} Assume that $A=E,$ where $E$ is a finite extension of $\Qp.$ For each family $\bD=(\bD_\fq)_{\fq\in S_p}$ satisfying 
the conditions {\bf N1-2)} of Section~\ref{section universal norms} we construct a pairing
\begin{equation*}
h_{V,\bD}^{\norm}\,\,:\,\,H^1(V,\bD)\times H^1(V^*(1),\bD^{\perp}) \rightarrow E,
\end{equation*}
which can be seen as a direct generalization of the $p$-adic height pairing, constructed for representations satisfying the Panchishkin condition using universal norms. The following theorem generalizes
Theorem~11.3.9 of \cite{Ne06}  (see Theorem~\ref{theorem comparision sel and norm heights} below). 

\begin{theoremII}
Let  $V$ be a $p$-adic representation 
of $G_{F,S}$ with coefficients in a finite extension $E$ of $\Qp.$  Assume that  the family $\bD=(\bD_\fq)_{\fq\in S_p}$ satisfies the conditions {\bf N1-2)}.  Then
\[
h^{\norm}_{V,D}=h^{\sel}_{V,\bD,1}.
\]
\end{theoremII}

\subsubsection{} We denote by $\Dd,$ $\Dst$ and $\Dst$ Fontaine's classical functors. Let $V$ be a $p$-adic representation with coefficients in $E/\Qp.$  Assume  that the restriction of $V$ on $G_{F_\fq}$ is potentially semistable for all $\fq \in S_p,$ and  that $V$ satisfies the following condition 

{\bf S)}  $\Dc (V)^{\Ph=1}=\Dc (V^*(1))^{\Ph=1}=0, \qquad \forall \fq \in S_p.$
\newline
For each $\fq \in S_p$ we fix a splitting $w_\fq\,:\,\Dd (V_\fq)
\rightarrow \Dd (V_\fq)/F^0\Dd (V_\fq)$ of the canonical projection $\Dd (V_\fq)\rightarrow \Dd (V_\fq)/\F^0\Dd (V_\fq)$
and set $w=(w_\fq)_{\fq\in S_p}.$ In this situation, Nekov\'a\v r
\cite{Ne92} constructed a $p$-adic height pairing 
\[
h_{V,w}^{\Hodge}\,:\, H^1_f(V)\times H^1_f(V^*(1)) \rightarrow E
\]
on the Bloch--Kato Selmer groups of $V$ and $V^*(1),$ which is defined using the Bloch--Kato exponential map and depends on the choice of splittings $w.$ 

Let $\fq \in S_p,$ and let $L$ be a finite extension of $F_\fq$ such that $V_\fq$ is semistable over $L.$ The semistable module $\bD_{\st/L} (V_\fq)$ is a finite dimensional vector space over the maximal unramified subextension $L_0$ of $L,$ equipped  with a Frobenius $\Ph,$ a monodromy $N,$ and an action of 
$G_{L/F_\fq}=\Gal (L/F_\fq).$

\begin{definition} Let $\fq \in S_p.$ We say that a $(\Ph,N,G_{L/F_\fq})$-submodule $D_\fq$ of 
\linebreak
$\bD_{\st/L}(V_\fq)$ is a splitting submodule if 
\[
\bD_{\dR/L} (V_\fq)=D_{\fq,L} \oplus \F^0\bD_{\dR/L} (V_\fq),\qquad D_{\fq,L}=D_\fq\otimes_{L_0}L
\]
as $L$-vector spaces.
\end{definition}

It is easy to see, that each splitting submodule $D_\fq$ defines a splitting 
of the Hodge filtration of $\Dd (V),$ which we denote by $w_{D,\fq}.$  For  each family $D=(D_\fq)_{\fq\in S_p}$ of splitting submodules we construct a 
pairing
\[
h_{V,D}^{\spl}\,:\, H^1_f(V)\times H^1_f(V^*(1)) \rightarrow E
\]
using the theory of $(\Ph,\Gamma)$-modules and prove that 
\[
h_{V,D}^{\spl}=h^{\Hodge}_{V,w_D}
\] 
(see Proposition~\ref{proposition comparision heights spl and Hodge}).
Let $\bD_\fq$ denote the $(\Ph,\Gamma_\fq)$-submodule of $\Ddagrig (V_\fq)$
associated to $D_\fq$ by Berger  \cite{Ber08} and let $\bD=(\bD_\fq)_{\fq \in S_p}.$
In the following theorem we compare this pairing with previous constructions (see Theorem~\ref{theorem comparision h^spl and h^norm} and Corollary~\ref{corollary comparision pairings}).

\begin{theoremIII}
\label{theoremIII}
Assume that  $(V,D)$ satisfies the conditions {\bf S)} and {\bf N2)}.
Then 

i) $H^1(V,\bD)=H^1_f(V)$ and $H^1(V^*(1),\bD^{\perp})=H^1_f(V^*(1));$   

ii) The height pairings $h_{V,\bD,1}^\sel,$ $h^{\norm}_{V,D}$ and $h^{\spl}_{V,D}$ 
coincide.
\end{theoremIII}

\subsubsection{} If $F=\Q,$ we can slightly relax the conditions on 
$(V,D)$ in Theorem~III. We replace these conditions by the conditions {\bf F1-2)} of Section~\ref{subsection extended Selmer groups}, which reflect 
the conjectural behavior of $V$ at $p$ in the presence of trivial zeros 
\cite{Ben11}, \cite{Ben14}.  We show, that under these conditions we have a
canonically splitting exact sequence
\begin{equation*}
\xymatrix{
0 \ar[r] &H^0(\bD') \ar[r] &H^1(V,\bD) \ar[r] &H^1_f(V)
 \ar@<-1ex>[l]_{\spl_{V,f}}
\ar[r] &0,
}
\end{equation*}
where $\bD'=\Ddagrig (V_p)/\bD.$ By modifying previous constructions, one can define a pairing
\begin{equation*}
h^{\norm}_{V,D}\,:\,H^1_f(V)\times H^1_f(V^*(1)) \rightarrow E
\end{equation*} 
The following result  is  a simplified form of Theorem~\ref{theorem comparision of heights for extended selmer} of Section~\ref{section extended selmer}, which can be seen as a generalization of  
Theorem 11.4.6 of \cite{Ne06}.
\begin{theoremIV} Let $V$ be a $p$-adic representation of $G_{\Q,S}$ that is 
potentially semistable at $p$ and satisfies the conditions {\bf F1-2)}. Then 

i) $h^{\norm}_{V,D}=h_{V,D}^{\spl};$

ii) For all $x\in H^1_f(V)$ and $y\in H^1_f(V^*(1))$ we have  
\begin{equation*}
h_{V,\bD}^{\sel}(\spl_{V,f}(x), \spl_{V^*(1),f}(y))=h_{V,D}^{\norm}(x,y).
\end{equation*}
\end{theoremIV}

\subsection{The organization of this paper} This paper is very technical by the nature, and in Sections~\ref{section complexes and products}-\ref{section cohomology of phi-Gamma modules} we assemble necessary preliminaries. In Section~\ref{section complexes and products}, we recall the formalism of cup products for cones following \cite{Ne06} (see also \cite{Ni93}), and prove some auxiliary cohomological results.  
In Section~\ref{section cohomology of phi-Gamma modules}, we review the theory of $(\Ph,\Gamma)$-modules and define the complex $K^{\bullet}(V_\fq).$ The reader familiar with $(\Ph,\Gamma)$-modules  can skip these sections on first reading. 
In Section~\ref{subsection Selmer complexes}, we construct Selmer complexes $\RG (V,\bD)$ and study their first properties. The $p$-adic height pairing $h^{\sel}_{V,\bD}$ is defined in Section~\ref{subsection construction of h^sel}. Theorem I (Theorem~\ref{theorem symmetricity of h^sel} of Section~\ref{subsection construction of h^sel}) follows directly from the definition
of $h^{\sel}_{V,\bD}$ and the machinery of Section~\ref{section complexes and products}. Preliminary results about splitting submodules are assembled in Section~\ref{section splitting submodules}. In Sections~\ref{section universal norms}-\ref{section extended selmer} we construct the pairings $h^{\norm}_{V,D}$ and 
$h^{\spl}_{V,D}$ and prove Theorems II, III and IV.

\subsection{Remarks} 1) In the forthcoming joint paper with K. B\"uy\"ukboduk \cite{BenB}, we prove
the Rubin style formula for our height pairing and apply it to 
the study of extra-zeros of $p$-adic $L$-functions. 

2) During the preparation of this paper, I learned from L. Xiao of an independent work in progress of Rufei Ren on this subject.  

\section{Complexes and products}
\label{section complexes and products}

\subsection{The complex $T^{\bullet}(A^{\bullet})$} 

\subsubsection{}
%{Notation} 
If  $R$ is a commutative ring, we write  $\mathcal K(R)$  for the category of
complexes of $R$-modules and $\mathcal K_{\ft}(R)$ for the subcategory of
$\mathcal K(R)$ consisting of complexes $C^{\bullet}=(C^n,d_{C^{\bullet}}^n)$
such that $H^n(C^{\bullet})$ are finitely generated over $R$ for all $n\in \Z.$ 
We write  $\mathcal D(R)$  and   $\mathcal D_{\ft}(R)$ for the corresponding
derived categories and denote by $[\,\cdot \,]\,:\,\mathcal K_*(R) 
\rightarrow \mathcal D_*(R),$ ($*\in\{\emptyset, \ft\}$) the obvious functors.
We will also consider the subcategories $\mathcal K_{\ft}^{[a,b]}(R),$ ($a\leqslant b$)
consisting of objects of    $\mathcal K_{\ft}(R)$ whose cohomologies 
are concentrated in degrees $[a,b].$
A perfect complex of $R$-modules is one of the form
\[
0\rightarrow P_a\rightarrow P_{a+1}\rightarrow \ldots \rightarrow P_b\rightarrow 0,
\]
where each  $P_i$ is a finitely generated projective $R$-module.  
If $R$ is noetherian, we denote by $\mathcal D^{[a,b]}_{\perf}(R)$ the full subcategory of 
 $\mathcal D_{\ft}(R)$ consisting of objects quasi-isomorphic to perfect complexes concentrated in degrees $[a,b].$

If $C^{\bullet}= (C^n, d_{C^{\bullet}}^n)_{n\in \Z}$ is a complex of 
$R$-modules and $m\in \Z,$ we will denote by $C^{\bullet}[m]$ the complex 
defined by $C^{\bullet}[m]^n=C^{n+m}$ and $ d_{C^{\bullet}[m]}^n(x)= 
(-1)^m d_{C^{\bullet}} (x).$ We will often write $d^n$ or just simply $d$ instead $d_{C^{\bullet}}^n.$ For each $m,$ the truncation $\tau_{\geqslant m}C^{\bullet}$
of $C^{\bullet}$  is 
the complex 
\begin{equation}
\nonumber 
0\rightarrow \mathrm{coker}(d^{m-1})\rightarrow C^{m+1}\rightarrow C^{m+2}
\rightarrow \cdots .
\end{equation}  
Therefore 
\begin{equation}
\nonumber
H^i(\tau_{\geqslant m}C^{\bullet})=
\begin{cases} 0, &{\text{\rm if $i<m$},}\\
H^i(C^{\bullet}), &{\text{\rm if $i\geqslant m$ .}}
\end{cases}
\end{equation}

The tensor product  $A^{\bullet}\otimes B^{\bullet}$ of two complexes  $A^{\bullet}$ and $B^{\bullet}$  is defined by
\begin{align}
&(A^{\bullet}\otimes B^{\bullet})^n=\underset{i\in \Z}\bigoplus 
\left (A^i\otimes B^{n-i}\right ), \nonumber \\
& d(a_i\otimes b_{n-i})=dx_i\otimes y_{n-i}+(-1)^ia_i\otimes b_{n-i},\qquad a_i\in A^i,
\quad b_{n-i}\in B^{n-i}. \nonumber 
\end{align}
We denote by $s_{12}\,:\,A^{\bullet}\otimes B^{\bullet} \rightarrow B^{\bullet}\otimes A^{\bullet}$ the transposition  
\begin{equation}
s_{12} (a_n\otimes b_m)=(-1)^{nm}b_m\otimes a_n,\qquad a_n\in A^n,\quad b_m\in B^m.
\nonumber 
\end{equation}
It is easy to check that $s_{12}$ is a morphism of complexes. We will also 
consider the map  $s_{12}^*\,:\,A^{\bullet}\otimes B^{\bullet} \rightarrow B^{\bullet}\otimes A^{\bullet}$ given by 
\begin{equation}
s_{12}^*(a_n\otimes b_m)= b_m\otimes a_n,
\nonumber
\end{equation}
which is not a morphism of complexes in general. 

Recall that a homotopy $h\,:\,f\rightsquigarrow g$ between two morphisms
$f,g\,:\, A^{\bullet}\rightarrow B^{\bullet}$ is a family of maps 
$h=(h^n\,:\,A^{n+1}\rightarrow B^n)$ such that $dh+hd=g-f.$ 
We will sometimes write $h$ instead $h^n.$
A second order homotopy $H\,:\, h\rightsquigarrow k$ between homotopies 
$h,k\,:\,f \rightsquigarrow g$ is a collection of maps 
$H=(H^n\,:\, A^{n+2}\rightarrow B^n)$ such that  $Hd-dH=k-h.$

If $f_i\,:A_1^{\bullet} \rightarrow B_1^{\bullet}$ ($i=1,2$) and
$g_i\,:A_2^{\bullet} \rightarrow B_2^{\bullet}$ ($i=1,2$) are 
morphisms of complexes and $h\,:\,f_1  \rightsquigarrow f_2$  and 
 $k\,:\,g_1  \rightsquigarrow g_2$ are homotopies between them, then 
the formula
\begin{equation}
\label{homotopy (h otimes k)_1)}
(h\otimes k)_1(x_n\otimes y_m)=h(x_n)\otimes g_1(y_m)+(-1)^n f_2(x_n)\otimes k(y_m), 
%&(h\otimes k)_2(x_n\otimes y_m)=h(x_n)\otimes g_2(y_m)+(-1)^n f_1(x_n)\otimes k(y_m),
%\nonumber
\end{equation} 
where $x_n\in A_1^n,$ $y_m\in A_2^m,$ defines a homotopy 
\begin{equation}
(h\otimes k)_1\,:\,f_1\otimes g_1 \rightsquigarrow f_2\otimes g_2. \nonumber
\end{equation}

\subsubsection{} 
\label{subsubsection cones}
For the content of this subsection we refer the reader to
\cite{Ve}, \S3.1.
If $f\,:\,A^{\bullet}\rightarrow B^{\bullet}$ is a morphism of complexes,
the cone of $f$ is defined to be the complex 
\begin{equation}
\mathrm{cone} (f)=A^{\bullet}[1]\oplus B^{\bullet},\nonumber 
\end{equation}
with  differentials
\begin{equation*} 
d^n(a_{ n +1} , b_n)= (-d^ {n+1} (a_{n+1}) , f(a_{n+1})+ d^n(b_n)).
\end{equation*}
We have a canonical distinguished triangle
\begin{equation}
A^{\bullet}\xrightarrow f B^{\bullet} \rightarrow \mathrm{cone}(f) \rightarrow A^{\bullet}[1].
\nonumber
\end{equation}
We say that a diagram of complexes of the form 
\begin{equation}
\label{diagram commutation up to homotopy}
\xymatrix{
A^{\bullet}_1\ar[rr]^{f_1} \ar[dd]^{\alpha_1} 
&&B_1^{\bullet} \ar[dd]^{\alpha_2}\\
& &  \\
A^{\bullet}_2\ar[rr]_(.4){f_2} &\ar @/^/ @{=>}[ur]^{h} &B^{\bullet}_2
}
\end{equation} 
is commutative up to homotopy, if there exists a homotopy 
\[
h\,:\,\alpha_2\circ f_1 \rightsquigarrow f_2\circ \alpha_1.
\]
In this case, the formula 
\begin{equation*}
c (\alpha_1,\alpha_2,h)^n(a_{n+1},b_n)= (\alpha_1(a_{n+1}), \alpha_2(b_n)+h^n(a_n))
\end{equation*}
defines a morphism of complexes 
\begin{equation}
\label{induced morphism for cones}
c (\alpha_1,\alpha_2,h) \,:\, 
\mathrm{cone}(f_1) \rightarrow \mathrm{cone}(f_2).
\end{equation}
Assume that, in addition to (\ref{diagram commutation up to homotopy}), we have
a diagram 
\begin{equation*}
\xymatrix{
A^{\bullet}_1\ar[rr]^{f_1} \ar[dd]^{\alpha_1'} 
&&B_1^{\bullet} \ar[dd]^{\alpha_2'}\\
& &  \\
A^{\bullet}_2\ar[rr]_(.4){f_2} & 
\ar @/^/ @{=>}[ur]^{h'}
&B^{\bullet}_2
}
\end{equation*} 
together with homotopies 
\begin{align*}
&k_1\,:\,\alpha_1 \rightsquigarrow \alpha_1'\\
&k_2\,:\,\alpha_2 \rightsquigarrow \alpha_2'
\end{align*}
and a second order homotopy
\begin{equation*}
H\,:\,f_2\circ k_1+h' \rightsquigarrow k_2\circ f_1+h.
\end{equation*}
Then the map 
\begin{equation}
\label{homotopy for cones}
(a_{n+1},b_n) \mapsto  (-k_1(a_{n+1}), k_2(b_n)+H(a_{n+1}))
\end{equation}
defines a homotopy $c (\alpha_1,\alpha_2,h)\rightsquigarrow c (\alpha_1',\alpha_2',h').$

\subsubsection{}
In the remainder of this section $R$ is a commutative ring and all complexes are 
complexes of $R$-modules. 
Let $A^{\bullet}=(A^n,d^n)$ be a complex 
equipped with a morphism $\Ph\,:\,A^{\bullet} \rightarrow A^{\bullet}.$
By definition, the total complex 
\[
T^{\bullet}(A^{\bullet})=\Tot \,(A^{\bullet} \xrightarrow{\Ph-1}A^{\bullet}).
\]
is given by  $T^n(A^{\bullet})=A^{n-1}\oplus A^n$
with differentials 
\[
d^n(a_{n-1},a_n)=(d^{n-1}a_{n-1}+ (-1)^n(\Ph-1)a_n,d^na_n),\quad (a_{ n-1}, a _n)\in T^n(A^{\bullet}).
\]
If  $A^{\bullet}$ and  $B^{\bullet}$ are  two complexes
equipped with morphisms $\Ph\,:\,A^{\bullet} \rightarrow A^{\bullet}$ 
and $\psi\,:\,B^{\bullet} \rightarrow B^{\bullet},$ and if 
$\alpha\,:\,A^{\bullet}\rightarrow B^{\bullet}$ is a morphism 
such that $\alpha\circ \Ph= \psi\circ \alpha,$ then $\alpha$ induces 
a morphism $T(\alpha)\,:\,T^{\bullet}(A^{\bullet})\rightarrow T^{\bullet}(B^{\bullet}).$
We will often write $\alpha$ instead $T(\alpha)$ to simplify notation.

\begin{mylemma}
\label{homotopy for T(A) to T(B)}
Let $A^{\bullet}$ and  $B^{\bullet}$ be two complexes
equipped with morphisms $\Ph~\,~:~\,~A^{\bullet} \rightarrow A^{\bullet}$ 
and $\psi\,:\,B^{\bullet} \rightarrow B^{\bullet},$ and let 
$\alpha_i\,:\,A^{\bullet}\rightarrow B^{\bullet}$ ($i=1,2$) 
be two morphisms such that 
\[
\alpha_i\circ \Ph=\psi \circ \alpha_i\,\qquad i=1,2.
\] 
If $h\,:\,\alpha_1 \rightsquigarrow \alpha_2$ is a homotopy 
between $\alpha_1$ and $\alpha_2$ such that
\linebreak
$h\circ \Ph=\psi\circ h,$  then the 
collection of maps 
$h_T=(h_T^n\,:\,T^{n+1}(A^{\bullet})\rightarrow T^n(B^{\bullet}))$ defined by 
$h_T^n(a_{n}, a_{n+1})=(h(a_{n}), h(a_{n+1}))$ is a homotopy between 
$T(\alpha_1)$ and $T(\alpha_2).$
\end{mylemma}
\begin{proof} The proof of this lemma  is a direct computation and is omitted here.
\end{proof} 
 
In the remainder of this subsection we will consider triples $(A^{\bullet}_1, A^{\bullet}_2,A^{\bullet}_3)$ of complexes of $R$-modules 
equipped with the following structures

{\bf A1)} Morphisms $\Ph_i\,:\,A^{\bullet}_i \rightarrow A^{\bullet}_i$ ($i=1,2,3$). 
%and $f_C\,:\,C^{\bullet}\rightarrow C^{\bullet}.$

{\bf A2)} A morphism $\cup_A\,:\,A^{\bullet}_1\otimes A^{\bullet}_2 \rightarrow A^{\bullet}_3$
which satisfies  

\[
\cup_A \circ (\Ph_1\otimes \Ph_2)=\Ph_3\circ \cup_A.
\]

\begin{myproposition}
\label{cup-product for T(A)}
 Assume that a triple  $(A^{\bullet}_i, \Ph_i)$ ($1\leqslant i\leqslant 3$)
satisfies the conditions {\bf A1-2)}.
Then the map
\[
\cup_A^T\,\,:\,\,T^{\bullet}(A_1^{\bullet})\otimes T^{\bullet}(A_2^{\bullet}) \rightarrow T^{\bullet}(A_3^{\bullet})
\]
given by
\[
(x_{n-1},x_n)\cup_A^T (y_{m-1},y_m)= (x_n\cup_A y_{m-1} +(-1)^m x_{n-1}\cup_A \Ph_2(y_m), x_n\cup_A y_m),
\]
is a morphism of complexes.
\end{myproposition}
\begin{proof} This proposition is well known to the experts
(compare, for example, to  \cite{Ni93} Proposition 3.1).
It follows from a direct computation which we recall for the convenience
of the reader. Let $(x_{n-1},x_{n})\in T^n(A^{\bullet}_1)$ and $(y_{m-1},y_m)\in T^m(A^{\bullet}_2).$
Then
\begin{multline}
d((x_{n-1},x_n)\cup_A^T (y_{m-1},y_m)=\\
=d(x_n\cup_A y_{m-1}+(-1)^mx_{n-1}\cup_A \Ph_2(y_m),x_n\cup_A y_m)=
(z_{n+m}, z_{n+m+1}), \nonumber
\end{multline}
where
\begin{multline}
z_{n+m}=
dx_n\cup_A y_{m-1}+(-1)^nx_n\cup_A dy_{m-1} +
(-1)^mdx_{n-1}\cup_A \Ph_2(y_m)+\\(-1)^{m+n-1}x_{n-1}\cup_A d(\Ph_2(y_m))+
(-1)^{n+m}(\Ph_3-1)(x_n\cup_A y_m) \nonumber 
\end{multline}
and $z_{n+m+1}=d(x_n\cup_A y_m).$
On the other hand
\begin{eqnarray}
&&\cup_A^T \circ d((x_{n-1},x_n)\otimes (y_{m-1},y_m))= \nonumber\\
&&=\cup_A^T \circ ((dx_{n-1}+(-1)^n(\Ph_1-1)x_n,dx_n)\otimes (y_{m-1},y_m))+ \nonumber \\
&&+(-1)^n \cup_A^T \circ ((x_{n-1},x_n)\otimes (dy_{m-1}+(-1)^m(\Ph_2-1)y_m,dy_m))=
\nonumber
\\
&&=(u_{n+m},u_{n+m+1}),\nonumber
\end{eqnarray}
where
\begin{multline}
u_{n+m}=dx_n\cup_A y_{m-1}+(-1)^m(dx_{n-1}+(-1)^n (\Ph_1-1)x_n)\cup  \Ph_2(y_m)+\\
(-1)^nx_n\cup (dy_{m-1}+(-1)^m(\Ph_2-1)y_m)+(-1)^{n+m-1}x_{n-1}\cup \Ph_2(dy_m),
\nonumber
\end{multline}
and 
$u_{n+m+1}=dx_n\cup_A y_m+(-1)^nx_n\cup_A dy_m.$
Now the proposition follows from the formula 
$$
d(x_n\cup_A y_m)=dx_n\cup_A y_m+(-1)^nx_n\cup_A dy_m
$$
and the assumption {\bf A2)} that reads $\Ph_1(x_n)\cup_A \Ph_2(y_m)=\Ph_3(x_n\cup_A y_m).$ 
\end{proof}

\begin{myproposition} 
\label{homotopy for T(A)}
Let $(A^{\bullet}_i, \Ph_i)$  
and $(B^{\bullet}_i, \psi_i)$ ($1\leqslant i\leqslant  3$)  be two triples of complexes that satisfy the conditions
{\bf A1-2)}. Assume that they are equipped with morphisms
\[
\alpha_i\,\,:\,\,A^{\bullet}_i\rightarrow B^{\bullet}_i,\quad %i=1,2,3 
\]
such that  $\alpha_i\circ \Ph_i=\psi_i\circ \alpha_i$ for all $1\leqslant i\leqslant 3$. 
Assume, in addition,  that in the diagram 
\begin{equation*}
\xymatrix{
A^{\bullet}_1\otimes A^{\bullet}_2 \ar[rr]^{\cup_A}\ar[dd]^{\alpha_1\otimes \alpha_2} & &
A^{\bullet}_3 \ar[dd]^{\alpha_3}\\
& &  \ar @/_/ @{=>}[dl]_{h}\\
B^{\bullet}_1\otimes B^{\bullet}_2 \ar[rr]^(.4){\cup_B} & &B^{\bullet}_3
}
\end{equation*}
there exists a homotopy 
\[
h\,:\,\alpha_3\circ \cup_A \rightsquigarrow    \cup_B\circ (\alpha_1\otimes \alpha_2).
\]
such that $h\circ (\Ph_1\otimes \Ph_2)=\psi_3\circ h.$ 
Then the collection  $h_T$ of  maps 
\[
h_T^{k}\,:\,\underset{m+n=k+1}\bigoplus \left (T^{n}(A_1^{\bullet})\otimes 
T^m(A_2^{\bullet})\right )
\rightarrow T^k(B_3^{\bullet})  
\]
defined by 
\begin{multline}
h_T^k((x_{n-1},x_n)\otimes (y_{m-1}\otimes y_m))=\\
=\bigl (h(x_n\otimes y_{m-1})+(-1)^mh(x_{n-1}\otimes \Ph_2(y_m)), h(x_n\otimes y_m)\bigr ). \nonumber
\end{multline}
provides  a homotopy 
$
h_T\,:\,\alpha_3\circ \cup_A^T \rightsquigarrow 
\cup_B^T\circ (\alpha_1\otimes \alpha_2) :
$
\begin{equation*}
\xymatrix{
T^{\bullet}(A^{\bullet}_1)\otimes T^{\bullet}(A^{\bullet}_2) \ar[rr]^{\cup_A^T}\ar[dd]^{\alpha_1\otimes \alpha_2} & &T^{\bullet}(A^{\bullet}_3) \ar[dd]^{\alpha_3}\\
& & \ar @/_/ @{=>}[dl]_{h_T}\\
T^{\bullet}(B^{\bullet}_1)\otimes T^{\bullet}(B^{\bullet}_2) \ar[rr]^{\cup_B^T} & &T^{\bullet}(B^{\bullet}_3).
}
\end{equation*}
\end{myproposition}  
\begin{proof} Again,  the proof is a routine computation. 
Let  $(x_{n-1},x_n)\in T^n(A^{\bullet}_1)$ and $(y_{m-1},y_m)\in T^m(A^{\bullet}_2).$
We have 
\begin{multline}
d ((x_{n-1}, x_n) \otimes (y_{m-1},y_m))=
(dx_{n-1} +(-1)^n(\Ph_1-1)x_n,dx_n)\otimes (y_{m-1},y_m)+\\
+(-1)^n (x_{n-1},x_n)\otimes (dy_{m-1}+(-1)^m(\Ph_2-1)y_m,dy_m),
\nonumber
\end{multline}
and therefore
\begin{equation}
h_T\circ d ((x_{n-1},x_n)\otimes (y_{m-1},y_m))= (a,b), \nonumber
\end{equation}
where 
\begin{align}
a&=h(d x_{n}\otimes y_{m-1})+(-1)^mh((dx_{n-1}+(-1)^n(\Ph_1-1)x_n)\otimes \Ph_2(y_m))+ \nonumber\\
&+(-1)^n(h(x_n\otimes(dy_{m-1}+(-1)^m(\Ph_2-1)y_m))+ \nonumber \\
&+(-1)^{n+m-1}h(x_{n-1}\otimes \Ph_2(dy_m))= \nonumber \\
&=h\circ d(x_n\otimes y_{m-1})+(-1)^m h\circ d (x_{n-1}\otimes \Ph_2(y_m))+ \nonumber \\
&+(-1)^{n+m}(\psi_3-1)\circ h(x_n\otimes y_m) \nonumber 
\end{align}
and 
\begin{eqnarray}
&&b=h(dx_n\otimes y_m)+ (-1)^n h(x_n\otimes dy_m)=h\circ d(x_n\otimes y_m). \nonumber 
\end{eqnarray}

%\begin{eqnarray}
%&&(h(d x_{n}\otimes y_{m-1})+(-1)^mh(dx_{n-1}+(-1)^n(f_1-1)x_n)\otimes f_2(y_m),h(dx_n%\otimes y_m))+ \nonumber\\
%&&+(-1)^n(h(x_n\otimes(dx_{n-1}+(-1)^m(f_2-1)y_m))+(-1)^{m-1}h(x_{n-1}\otimes %f_2(y_m)),h(x_n\otimes dy_m))= \nonumber\\
%&&(h\circ d(x_n\otimes y_{m-1}))+(-1)^m h\circ d (x_{n-1}\otimes f_2(y_m))+
%(-1)^{n+m}(f_3-1)\circ h(x_n\otimes y_m), h\circ d(x_n\otimes y_m)).
%\nonumber
%\end{eqnarray}
On the other hand 
\begin{align}
&d\circ h_T ((x_{n-1},x_n)\otimes (y_{m-1}, y_m))=\nonumber \\
&=d (h(x_n\otimes y_{m-1})+(-1)^mh(x_{n-1}\otimes \Ph_2(y_m)),h(x_n\otimes y_m))=
\nonumber\\
&=(d\circ h(x_n\otimes y_{m-1})+(-1)^m d\circ h (x_{n-1}\otimes \Ph_2(y_m))+ \nonumber \\
&+(-1)^{n+m-1}(\psi_3-1)h(x_n\otimes y_m),
d\circ h(x_n\otimes y_m)).
\nonumber
\end{align}
Thus
\begin{align}
&(h_T d+d h_T) ((x_{n-1},x_n)\otimes (y_{m-1},y_m))=\nonumber \\
&=((h d+d h)(x_n\otimes y_{m-1})+ (-1)^m (h d+d h)(x_{n-1}\otimes \Ph_2(y_m)), \nonumber \\
&\qquad \qquad \qquad 
\qquad \qquad \qquad \qquad 
\qquad \qquad \qquad  
(h d+d h)(x_n\otimes y_m))= \nonumber \\
&=((\alpha_1(x_n)\cup_B \alpha_2(y_{m-1})-\alpha_3 (x_n\cup_A y_{m-1}))+ \nonumber \\
&\qquad \qquad (-1)^m(\alpha_1(x_{n-1})\cup_B\Ph_2(\alpha_2(y_m)) 
-\alpha_3 (x_{n-1}\cup_A \Ph_2(y_m)), \nonumber \\
& \qquad \qquad \qquad \qquad \qquad \qquad \qquad  \alpha_1(x_n)\cup_B\alpha_2(y_m)-\alpha_3(x_n\cup_Ay_m))=\nonumber \\
&=(\cup_B^T\circ (\alpha_1\otimes \alpha_2)-\alpha_3\circ \cup_A^T)((x_{n-1},x_n)\otimes (y_{m-1},y_m)).\nonumber 
\end{align}
and the proposition is proved.

\end{proof}

\begin{myproposition}
\label{transpositions for T(A)}
 Let $A^{\bullet}_i$
($1\leqslant i\leqslant 4$)
be four complexes equipped with morphisms $\Ph_i\,:\,A^{\bullet}_i\rightarrow
A^{\bullet}_i$ and  such that 

a) The triples $(A^{\bullet}_1,A^{\bullet}_2,A^{\bullet}_3)$ and 
$(A^{\bullet}_1,A^{\bullet}_2,A^{\bullet}_4)$ satisfy  {\bf A1-2)}.

b) The complexes $A^{\bullet}_i$ ($i=1,2$) are equipped with
morphisms $\mathcal T_i\,:\,A^{\bullet}_i\rightarrow A^{\bullet}_i$
which commute with morphisms $\Ph_i$
\[
\mathcal T_i\circ \Ph_i=\Ph_i\circ \mathcal T_i, \qquad   i=1,2. 
\]

c) There exists a morphism $\mathcal T_{34}\,:\,A^{\bullet}_3 \rightarrow A^{\bullet}_4$ such that
\[
\mathcal T_{34}\circ \Ph_3=\Ph_4\circ \mathcal T_{34}.
\]

d) The diagram
\begin{displaymath}
\xymatrix{
A^{\bullet}_1\otimes A^{\bullet}_2 \ar[rr]^{\cup_A} 
\ar[dd]^{s_{12}\circ (\mathcal T_1\otimes \mathcal T_2)} & &A^{\bullet}_3 \ar[dd]^{\mathcal T_{34}}\\
& &\\
A^{\bullet}_2\otimes A^{\bullet}_1 \ar[rr]^{\cup_A} & &A^{\bullet}_4.
}
\end{displaymath}
commutes.
%$$
%\cup_C \circ s_{12}\circ (\Cal T_1\otimes \Cal T_2)= 
%\Cal T_{34}\circ \cup_C.
%$$

Let $\mathcal T_i\,:\,T^{\bullet}(A^{\bullet}_i) \rightarrow 
T^{\bullet}(A^{\bullet}_i)$
($i=1,2$)
and $\mathcal T_{34}\,:\,T^{\bullet}(A^{\bullet}_3) 
\rightarrow T^{\bullet}(A^{\bullet}_4)$
be the morphisms (which we denote again  by the same letter) defined by
\[
\mathcal T_i(x_{n-1},x_n)=(\mathcal T_i (x_{n-1}),\mathcal T_i(x_n)),
\quad
\mathcal T_{34}(x_{n-1},x_n)=(\mathcal T_{34} (x_{n-1}),\mathcal T_{34}(x_n)).
\]
 
Then in the diagram
\begin{displaymath}
\xymatrix{
T^{\bullet}(A^{\bullet}_1)\otimes T^{\bullet}(A^{\bullet}_2) \ar[rr]^{\cup_A^T} 
\ar[dd]^{s_{12}\circ (\mathcal T_1\otimes \mathcal T_2)} & 
&T^{\bullet}(A^{\bullet}_3) \ar[dd]^{\mathcal T_{34}}\\
& & \ar @/_/ @{=>}[dl]_{h_{\mathcal T}}\\
T^{\bullet}(A^{\bullet}_2)\otimes T^{\bullet}(A^{\bullet}_1) \ar[rr]^{\cup_A^T} & &T^{\bullet}(A^{\bullet}_4)
}
\end{displaymath}
the maps $\mathcal T_{34}\circ \cup^T_A $ and $\cup_A^T\circ s_{12}\circ 
(\mathcal T_1\otimes \mathcal T_2)$
are homotopic. 
\end{myproposition}

\begin{proof} Let $(x_{n-1},x_n)\in T^n(A^{\bullet}_1)$ and  
$(y_{m-1},y_m)\in T^m(A^{\bullet}_2).$ Then
\begin{multline}
\label{prop.3 formula 1}
\mathcal T_{34}((x_{n-1},x_n)\cup_A^T(y_{m-1},y_m))=\\
=(\mathcal T_{34}(x_n\cup_A y_{m-1})+(-1)^m \mathcal T_{34}(x_{n-1}\cup_A\Ph_2(y_m)), \mathcal T_{34}(x_n\cup_Ay_m)) 
%\\&=( \mathcal T_2y_{m-1}\cup_A \mathcal T_1 x_n+
%(-1)^m \Ph_2(\mathcal T_2 y_m)\cup_A \mathcal T_1x_{n-1}, \mathcal T_2 y_m\cup_A\mathcal T_1x_n).
\end{multline}
and 
\begin{equation}
\begin{aligned}
& \cup^T_A\circ s_{12}\circ (\mathcal T_1\otimes \mathcal T_2)((x_{n-1},x_n)\otimes (y_{m-1},y_m))= 
 \\
&=(-1)^{mn} \mathcal T_2(y_{m-1},y_m)\cup_A\mathcal T_1(x_{n-1},x_n)=
\\
&
\!\begin{multlined}
=(-1)^{mn}(\mathcal T_2(y_m)\cup_A\mathcal T_1(x_{n-1})+ (-1)^n \mathcal T_2(y_{m-1})\cup_A\Ph_1(\mathcal T_1(x_n)),\\
\mathcal T_2(y_m)\cup_A \mathcal T_1(x_n))= 
\end{multlined}
 \\
%(-1)^{mn} ((-1)^{m(n-1)}\mathcal T_{34}(x_{n-1}\cup_A y_m)+
%(-1)^{mn}\mathcal T_{34}   (\Ph_1(x_n)\cup_A y_{m-1}), \\(-1)^{mn}
%\mathcal T_{34}(x_n\cup_A y_m))=\\
&=((-1)^m\mathcal T_{34}(x_{n-1}\cup_A y_m)+\mathcal T_{34}(\Ph_1(x_n)\cup_A y_{m-1}),
\mathcal T_{34}(x_n\cup_A y_m)). 
\end{aligned}
\end{equation}
Define
\begin{equation}
h_{\mathcal T}^{k}\,:\,
\underset{m+n=k+1}\bigoplus  \left (T^{n}(A^{\bullet}_1)\otimes T^{m}(A^{\bullet}_2)
\right ) \rightarrow T^{k}(A^{\bullet}_4),\nonumber
\end{equation}
by
\begin{equation}
\label{definition of the homotopy}
h_{\mathcal T}^{k}((x_{n-1},x_n)\otimes (y_{m-1}\otimes y_m))=
(-1)^{n-1} (\mathcal T_{34}(x_{n-1}\cup_A y_{m-1}),0).
\end{equation}
Then 
\begin{align}
&d h_{\mathcal T} ((x_{n-1},x_n)\otimes (y_{m-1}\otimes y_m))=  \\
&=(-1)^{n-1}d(\mathcal T_{34}(x_{n-1}\cup_A y_{m-1}),0)= \nonumber \\
&=(-1)^{n-1}(\mathcal T_{34}(d x_{n-1}\cup_A y_{m-1}+(-1)^{n-1}x_{n-1}\cup_A dy_{m-1}),0)= \nonumber \\
&=((-1)^{n-1}\mathcal T_{34}(d x_{n-1}\cup_A y_{m-1})+\mathcal T_{34}(x_{n-1}\cup_A dy_{m-1}),0), \nonumber 
\end{align}
and 
\begin{align}
\label{Prop. 3 formula 5}
&h_{\mathcal T} d ((x_{n-1},x_n)\otimes (y_{m-1}\otimes y_m))= \\
&=h_{\mathcal T} ((dx_{n-1}+(-1)^n(\Ph_1-1)x_n,dx_n)\otimes (y_{m-1},y_m)+ \nonumber  \\
&+(-1)^n (x_{n-1},x_n)\otimes (dy_{m-1}+(-1)^m(\Ph_2-1)y_m,dy_m))= \nonumber\\
&=((-1)^{n}\mathcal T_{34}(dx_{n-1}\cup_A y_{m-1})+\mathcal T_{34}(\Ph_1(x_n)\cup_Ay_{m-1})- \nonumber \\
&-\mathcal T_{34}(x_n\cup_Ay_{m-1})
-\mathcal T_{34}(x_{n-1}\cup_Ady_{m-1})- \nonumber \\
&-(-1)^m\mathcal T_{34}(x_{n-1}\cup_A\Ph_2(y_m))
+(-1)^m \mathcal T_{34}(x_{n-1}\cup_Ay_m),0). \nonumber 
\end{align}
From (\ref{prop.3 formula 1})-(\ref{Prop. 3 formula 5}) it follows that
\[
\cup^T_A\circ s_{12}\circ (\mathcal T_1\otimes \mathcal T_2)-
\mathcal T_{34}\circ \cup_A^T=d h_{\mathcal T}+h_{\mathcal T} d
\]
and the proposition is proved. 
\end{proof}

\subsection{Products}
\label{products} 
\subsubsection{} In this subsection we review the construction of products for
cones following Nekov\'a\v r \cite{Ne06} and Nizio\l \cite{Ni93}.
We will work with the following data

{\bf P1)}  Diagrams 
\begin{equation}
A_i^{\bullet}\xrightarrow{f_i} C_i^{\bullet}\xleftarrow{g_i}B_i^{\bullet},\qquad i=1,2,3,
\nonumber
\end{equation}
where $A_i^{\bullet},$ $B_i^{\bullet}$ and $C_i^{\bullet}$ are complexes of $R$-modules.

{\bf P2)} Morphisms
\begin{align}
&\cup_A\,:\,A_1^{\bullet}\otimes A_2^{\bullet} \rightarrow A_3^{\bullet}, \nonumber\\
&\cup_B\,:\,B_1^{\bullet}\otimes B_2^{\bullet} \rightarrow B_3^{\bullet}, \nonumber \\
&\cup_C\,:\,C_1^{\bullet}\otimes C_2^{\bullet} \rightarrow C_3^{\bullet}. \nonumber
\end{align}
 
{\bf P3)} A pair of homotopies $h=(h_f,h_g)$ 
\begin{align}
&h_f\,:\, \cup_C\circ (f_1\otimes f_2) \rightsquigarrow f_3\circ \cup_A, \nonumber\\
&h_g\,:\, \cup_C\circ (g_1\otimes g_2) \rightsquigarrow g_3 \circ \cup_B. \nonumber 
\end{align}
Define 
\begin{equation}
E_i^{\bullet}=\mathrm{cone} \left (A_i^{\bullet}\oplus B_i^{\bullet}\xrightarrow{f_i-g_i}C_i^{\bullet} \right ) [-1].
\end{equation}
Thus 
\[
E_i^n=A_i^n\oplus B_i^n\oplus C_i^{n-1}
\]
with $d(a_n,b_n,c_{n-1})=(d a_n, d b_n, -f_i(a_n)+g_i(b_n)-dc_{n-1}). $

\begin{myproposition}
\label{construction of cup product for E_1, E_2}
i) Given the data {\bf P1-3)}, for each $r\in R$ the formula 
\begin{eqnarray}
&&(a_n, b_n, c_{n-1} )\cup_{r,h}(a'_m,b'_m, c'_{m-1})= \nonumber\\ 
&&(a_n\cup_A a'_m,b_n\cup_B b'_m, 
c_{n-1}\cup_C (rf_2(a'_m)+(1-r)g_2(b'_m))+ \nonumber \\
&&(-1)^{n}((1-r)f_1(a_n)+rg_1(b_n))\cup_C c'_{m-1}-
(h_f(a_n\otimes a_m')-h_g(b_n\otimes b_m'))) \nonumber
\end{eqnarray}
defines a morphism in $\mathcal K(R)$
\[
\cup_{r,h}\,:\, E_1^{\bullet}\otimes E_2^{\bullet} \rightarrow E_3^{\bullet}.
\]

ii) If $r_1,r_2\in R$, then the map
\begin{equation}
k\,:\, E_1^{\bullet}\otimes E_2^{\bullet} \rightarrow E_3^{\bullet}[-1],\nonumber 
\end{equation}
given by 
\begin{equation}
k((a_n, b_n, c_{n-1} )\otimes (a'_m,b'_m, c'_{m-1}))=
(0,0, (-1)^{n}(r_1-r_2) c_{n-1}\cup_C c_{m-1}') \nonumber 
\end{equation}
for all $(a_n, b_n, c_{n-1} )\in E_1^n$ and  $(a'_m,b'_m, c'_{m-1})\in E_2^m,$ 
defines a homotopy 
\linebreak
$k\,:\,\cup_{r_1,h} \rightsquigarrow  \cup_{ r_2,h}.$

iii) If $h'=(h'_f,h'_g)$ is another pair of homotopies as in {\bf P3)}, 
and if $\alpha \,:\, h_f \rightsquigarrow h_f'$ and $\beta \,:\,
h_g \rightsquigarrow h_g'$ is a pair of second order homotopies, then 
the map
\begin{align}
&s\,:\, E_1^{\bullet}\otimes E_2^{\bullet} \rightarrow E_3^{\bullet}[-1],\nonumber \\
&s((a_n, b_n, c_{n-1} )\otimes (a'_m,b'_m, c'_{m-1}))=
(0,0, \alpha (a_n\otimes a'_m) -\beta (b_n, b_m')) \nonumber 
\end{align}
defines a homotopy $s\,:\, \cup_{r,h} \rightsquigarrow \cup_{r,h'}.$
\end{myproposition}
\begin{proof} See \cite{Ni93}, Proposition 3.1.
\end{proof}

\subsubsection{} Assume that, in addition to {\bf P1-3)}, we are given the following 
data

{\bf T1)} Morphisms of complexes
\begin{align}
&\mathcal T_A\,:\, A_i^{\bullet} \rightarrow A_i^{\bullet}, \nonumber \\
&\mathcal T_B\,:\, B_i^{\bullet} \rightarrow B_i^{\bullet}, \nonumber \\
&\mathcal T_C\,:\, C_i^{\bullet} \rightarrow C_i^{\bullet}, \nonumber 
\end{align}
for $i=1,2,3.$

{\bf T2)} Morphisms of complexes 
\begin{align}
&\cup'_A\,:\,A_2^{\bullet}\otimes A_1^{\bullet} \rightarrow A_3^{\bullet}, \nonumber \\
&\cup'_B\,:\,B_2^{\bullet}\otimes B_1^{\bullet} \rightarrow B_3^{\bullet}, \nonumber \\
&\cup'_C\,:\,C_2^{\bullet}\otimes C_1^{\bullet} \rightarrow C_3^{\bullet}. \nonumber 
\end{align}
 
{\bf T3)} A pair of homotopies $h'=(h'_f,h'_g)$ 
\begin{align}
&h'_f\,:\, \cup'_C\circ (f_2\otimes f_1) \rightsquigarrow f_3\circ \cup'_A, \nonumber \\
&h'_g\,:\, \cup'_C\circ (g_2\otimes g_1) \rightsquigarrow g_3 \circ \cup'_B. \nonumber
\end{align}

{\bf T4)} Homotopies 
\begin{eqnarray}
&U_i\,:\,f_i\circ \mathcal T_A   \rightsquigarrow \mathcal T_C\circ f_i,\nonumber \\
&V_i\,:\,g_i\circ \mathcal T_B   \rightsquigarrow \mathcal T_C\circ g_i,\nonumber
\end{eqnarray}
for $i=1,2,3.$

{\bf T5)} Homotopies 
\begin{eqnarray}
&&t_A\,:\, \cup'_A\circ s_{12}\circ (\mathcal T_A\otimes \mathcal T_A) 
\rightsquigarrow \mathcal T_A\circ \cup_A,\nonumber \\
&&t_B\,:\, \cup'_B\circ s_{12}\circ (\mathcal T_B\otimes \mathcal T_B) 
\rightsquigarrow \mathcal T_B\circ \cup_B,\nonumber \\
&&t_C\,:\, \cup'_C\circ s_{12}\circ (\mathcal T_C\otimes \mathcal T_C) 
\rightsquigarrow \mathcal T_C\circ \cup_C.\nonumber
\end{eqnarray}

{\bf T6)} A second order homotopy $H_f$ trivializing the boundary of the cube 
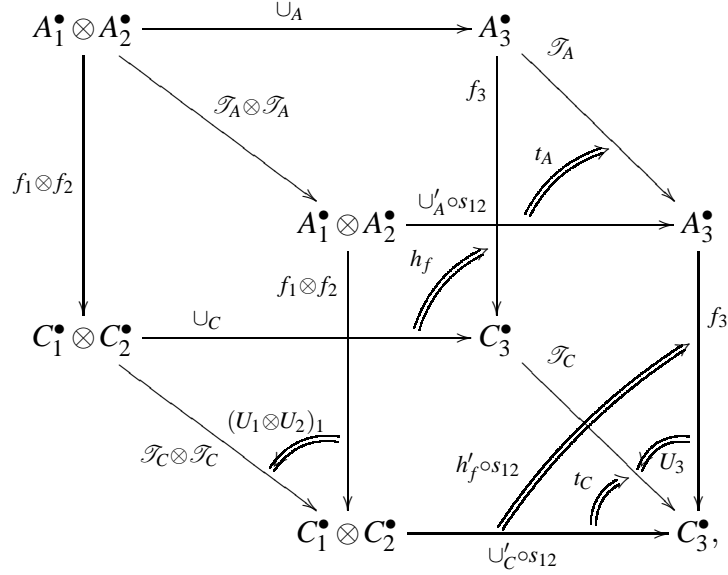
\begin{displaymath}
\xymatrix{
A^{\bullet}_1\otimes A^{\bullet}_2 \ar[rrr]^{\cup_A} 
\ar[ddd]_{f_1\otimes f_2} \ar[ddrr]^{\mathcal T_A\otimes \mathcal T_A} & &
&A^{\bullet}_3 \ar[ddd]_(.2){f_3}_(.7){}="d" \ar[ddrr]^(.2){\mathcal T_A}^(.6){}="b"&&\\
& & & & &\\
& &A_1^{\bullet}\otimes A_2^{\bullet} \ar[rrr]^(.3){\cup'_A\circ s_{12}}^(.5){}="a" \ar[ddd]_(.2){f_1\otimes f_2}_(.7){}="e"& 
& 
&A_3^{\bullet}\ar[ddd]^(.3){f_3}_(.7){}="g"\\
C_1^{\bullet}\otimes C_2^{\bullet} \ar[ddrr]_{\mathcal T_C\otimes \mathcal T_C}^(.7){}="f" 
\ar[rrr]^(.3){\cup_C}^(.8){}="c"
&  
& 
&C_3^{\bullet}\ar[ddrr]^(.2){\mathcal T_C}^(.7){}="h" & &\\
& & & & &\\
& &C_1^{\bullet}\otimes C_2^{\bullet} \ar[rrr]_{\cup'_C\circ s_{12}}^(.7){}="i"
&
\ar @/^/ @{=>}[uurr]^(.2){h_f'\circ s_{12}}
&
&C_3^{\bullet}, 
\ar @/^/  @{=>}^{t_A} "a"; "b"
\ar @/^/  @{=>}^{h_f} "c"; "d"
\ar @/_/ @{=>}_{(U_1\otimes U_2)_1} "e"; "f"
\ar @/_/ @{=>}^{U_3} "g"; "h"
\ar @/^/  @{=>}^{t_C} "i"; "h"
}
\end{displaymath}
i.e. a system $H_f=(H_f^i)_{i\in \mathbf Z}$ of maps $H_f^i\,:\,(A_1\otimes A_2)^i \rightarrow  C_3^{i-2}$ such that 
\begin{multline}
dH_f-H_fd=-t_C\circ (f_1\otimes f_2)-\mathcal T_C\circ h_f+U_3\circ \cup_A+\\
+f_3\circ t_A+h_f'\circ (s_{12}\circ (\mathcal T_A\otimes \mathcal T_A)) -(\cup'_C \circ s_{12})\circ (U_1\otimes U_2)_1.
\nonumber
\end{multline}
In this formula, $(U_1\otimes U_2)_1$ denotes the homotopy defined by
(\ref{homotopy (h otimes k)_1)}).

{\bf T7)} A second order homotopy $H_g$ trivializing the boundary of the cube 
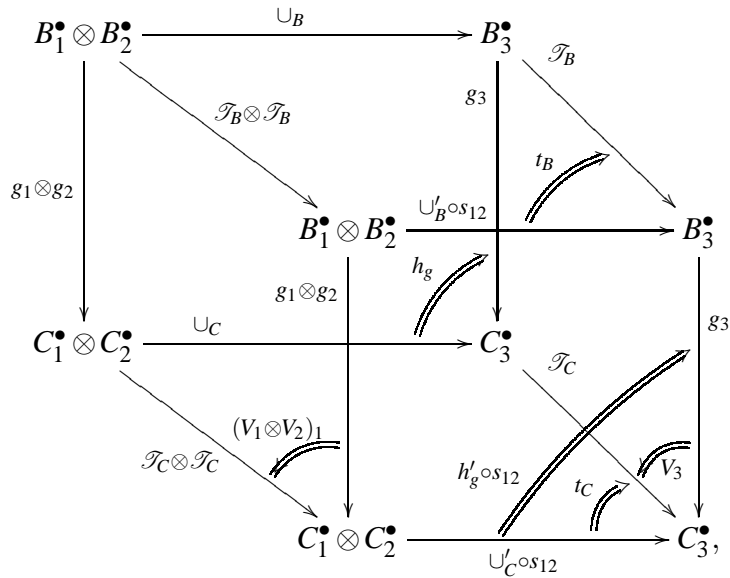
\begin{displaymath}
\xymatrix{
B^{\bullet}_1\otimes B^{\bullet}_2 \ar[rrr]^{\cup_B} 
\ar[ddd]_{g_1\otimes g_2} \ar[ddrr]^{\mathcal T_B\otimes \mathcal T_B} & &
&B^{\bullet}_3 \ar[ddd]_(.2){g_3}_(.7){}="d" \ar[ddrr]^(.2){\mathcal T_B}^(.6){}="b"&&\\
& & & & &\\
& &B_1^{\bullet}\otimes B_2^{\bullet} \ar[rrr]^(.3){\cup'_B\circ s_{12}}^(.5){}="a" \ar[ddd]_(.2){g_1\otimes g_2}_(.7){}="e"& 
& 
&B_3^{\bullet}\ar[ddd]^(.3){g_3}_(.7){}="g"\\
C_1^{\bullet}\otimes C_2^{\bullet} \ar[ddrr]_{\mathcal T_C\otimes \mathcal T_C}^(.7){}="f" 
\ar[rrr]^(.3){\cup_C}^(.8){}="c"
&  
& 
&C_3^{\bullet}\ar[ddrr]^(.2){\mathcal T_C}^(.7){}="h" & &\\
& & & & &\\
& &C_1^{\bullet}\otimes C_2^{\bullet} \ar[rrr]_{\cup'_C\circ s_{12}}^(.7){}="i"
&
\ar @/^/ @{=>}[uurr]^(.2){h_g'\circ s_{12}}
&
&C_3^{\bullet}, 
\ar @/^/  @{=>}^{t_B} "a"; "b"
\ar @/^/  @{=>}^{h_g} "c"; "d"
\ar @/_/ @{=>}_{(V_1\otimes V_2)_1} "e"; "f"
\ar @/_/ @{=>}^{V_3} "g"; "h"
\ar @/^/  @{=>}^{t_C} "i"; "h"
}
\end{displaymath}
i.e. a system $H_g=(H_g^i)_{i\in \mathbf Z}$ of maps $H_g^i\,:\,(B_1\otimes B_2)^i \rightarrow  C_3^{i-2}$ such that 
\begin{multline}
dH_g-H_gd=-t_C\circ (g_1\otimes g_2)-\mathcal T_C\circ h_g+V_3\circ \cup_B+\\
+g_3\circ t_B+h_g'\circ (s_{12}\circ (\mathcal T_B\otimes \mathcal T_B)) -(\cup'_C \circ s_{12})\circ (V_1\otimes V_2)_1.
\nonumber
\end{multline}

\begin{myproposition}
\label{proposition commutativity of products}
 i) Given the data {\bf P1-3)} and {\bf T1-7)},
the formula
\[
\mathcal T_i(a_n, b_n, c_{n-1})= (\mathcal T_A(a_n), \mathcal T_B (b_n), 
\mathcal T_C (c_{n-1}) +U_i (a_n) -V_i (b_n))
\]
defines morphisms of complexes 
\[
\mathcal T_i \,:\, E_i^{\bullet} \rightarrow E_i^{\bullet}, \qquad i=1,2,3
\]
such that, for any $r\in R,$ the diagram
\begin{equation}
\nonumber 
\xymatrix{
E_1^{\bullet}\otimes E_2^{\bullet}
\ar[rr]^{\cup_{r,h}} \ar[dd]^{s_{12}\circ (\mathcal T_1\otimes \mathcal T_2)}
& &E_3^{\bullet} \ar[dd]^{\mathcal T_3}\\
& &\\
E_2^{\bullet}\otimes E_1^{\bullet}
\ar[rr]^{\cup'_{1-r,h'}} &
&E_3^{\bullet}}
\end{equation}
commutes up to homotopy.
\end{myproposition}
\begin{proof}
See \cite{Ne06}, Proposition 1.3.6. 
\end{proof}

\subsubsection{Bockstein maps}
Assume that, in addition to  {\bf P1-3)}, we are given the following data

{\bf B1)} Morphisms of complexes
\begin{equation}
\beta_{Z,i}\,:\,Z_i^{\bullet} \rightarrow Z_i^{\bullet}[1], \qquad Z_i^{\bullet}=A_i^{\bullet},B_i^{\bullet}, C_i^{\bullet}, \quad i=1,2.
\nonumber
\end{equation}

{\bf B2)} Homotopies 
\begin{eqnarray}
&u_i\,:\, f_i[1]\circ \beta_{A,i} \rightsquigarrow \beta_{C,i}\circ f_i, \nonumber \\
&v_i\,:\, g_i[1]\circ \beta_{B,i} \rightsquigarrow \beta_{C,i}\circ g_i \nonumber
\end{eqnarray} 
for $i=1,2.$

{\bf B3)} Homotopies 
\begin{equation}
h_Z\,:\, \cup_Z[1]\circ (\mathrm{id}\otimes \beta_{Z,2}) \rightsquigarrow
\cup_Z[1]\circ ( \beta_{Z,1}\otimes \mathrm{id}), \nonumber 
\end{equation} 
for $Z^{\bullet}=A^{\bullet}, B^{\bullet}, C^{\bullet}.$

{\bf B4)} A second order homotopy trivializing the boundary of the following diagram 
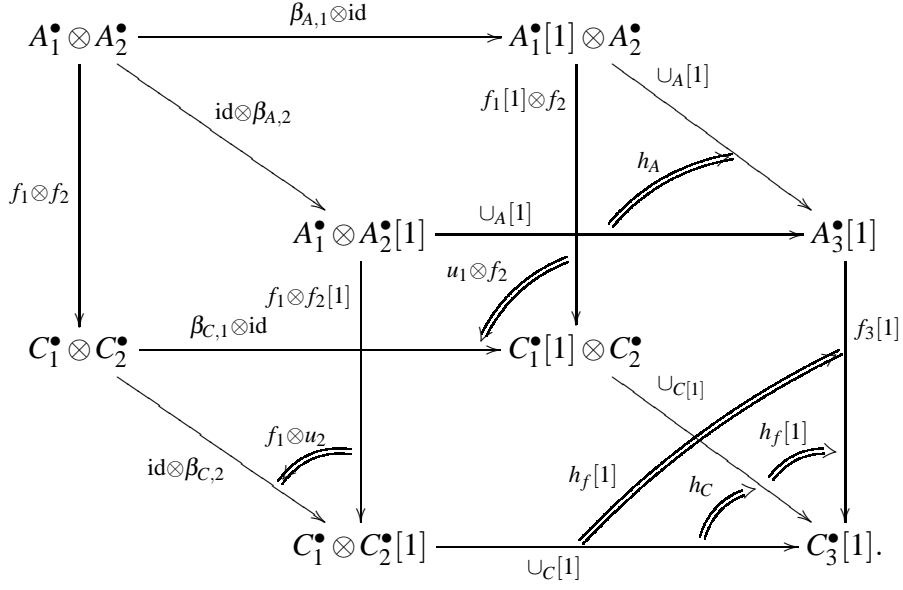
\begin{displaymath}
\xymatrix{
A^{\bullet}_1\otimes A^{\bullet}_2 \ar[rrr]^{\beta_{A,1}\otimes \id} 
\ar[ddd]_{f_1\otimes f_2} \ar[ddrr]^{\id\otimes \beta_{A,2}} & &
&A_1^{\bullet}[1]\otimes A_2^{\bullet} \ar[ddd]_(.2){f_1[1]\otimes f_2}_(.7){}="d" \ar[ddrr]^(.3){\cup_A[1]}^(.6){}="b"&&\\
& & & & &\\
& &A_1^{\bullet}\otimes A_2^{\bullet}[1] \ar[rrr]^(.3){\cup_A[1]}^(.5){}="a" \ar[ddd]_(.2){f_1\otimes f_2[1]}_(.7){}="e"& 
& 
&A_3^{\bullet}[1]\ar[ddd]^(.3){f_3[1]}_(.7){}="g"\\
C_1^{\bullet}\otimes C_2^{\bullet} \ar[ddrr]_{\id \otimes \beta_{C,2}}^(.7){}="f" 
\ar[rrr]^(.3){\beta_{C,1}\otimes \id}^(.8){}="c"
&  
& 
&C_1^{\bullet}[1]\otimes C_2^{\bullet} \ar[ddrr]^(.3){\cup_{C[1]}}^(.7){}="h" & &\\
& & & & &\\
& &C_1^{\bullet}\otimes C_2^{\bullet}[1] \ar[rrr]_(.4){\cup_C[1]}^(.7){}="i"
&
\ar @/^/ @{=>}[uurr]^(.2){h_f[1]}
&
&C_3^{\bullet}[1]. 
\ar @/^/  @{=>}^{h_A} "a"; "b"
\ar @/_/  @{=>}_{u_1\otimes f_2} "d"; "c"
\ar @/_/ @{=>}_{f_1\otimes u_2} "e"; "f"
\ar @/^/ @{=>}^{h_f[1]} "h"; "g"
\ar @/^/  @{=>}^{h_C} "i"; "h"
}
\end{displaymath}

{\bf B5)} A second order homotopy trivializing the boundary of the cube 

\begin{displaymath}
\xymatrix{
B^{\bullet}_1\otimes B^{\bullet}_2 \ar[rrr]^{\beta_{B,1}\otimes \id} 
\ar[ddd]_{g_1\otimes g_2} \ar[ddrr]^{\id\otimes \beta_{B,2}} & &
&B_1^{\bullet}[1]\otimes B_2^{\bullet} \ar[ddd]_(.2){g_1[1]\otimes g_2}_(.7){}="d" \ar[ddrr]^(.3){\cup_B[1]}^(.6){}="b"&&\\
& & & & &\\
& &B_1^{\bullet}\otimes B_2^{\bullet}[1] \ar[rrr]^(.3){\cup_B[1]}^(.5){}="a" \ar[ddd]_(.2){g_1\otimes g_2[1]}_(.7){}="e"& 
& 
&B_3^{\bullet}[1]\ar[ddd]^(.3){g_3[1]}_(.7){}="g"\\
C_1^{\bullet}\otimes C_2^{\bullet} \ar[ddrr]_{\id \otimes \beta_{C,2}}^(.7){}="f" 
\ar[rrr]^(.3){\beta_{C,1}\otimes \id}^(.8){}="c"
&  
& 
&C_1^{\bullet}[1]\otimes C_2^{\bullet} \ar[ddrr]^(.3){\cup_{C[1]}}^(.7){}="h" & &\\
& & & & &\\
& &C_1^{\bullet}\otimes C_2^{\bullet}[1] \ar[rrr]_(.4){\cup_C[1]}^(.7){}="i"
&
\ar @/^/ @{=>}[uurr]^(.2){h_g[1]}
&
&C_3^{\bullet}[1]. 
\ar @/^/  @{=>}^{h_B} "a"; "b"
\ar @/_/  @{=>}_{v_1\otimes f_2} "d"; "c"
\ar @/_/ @{=>}_{g_1\otimes v_2} "e"; "f"
\ar @/^/ @{=>}^{h_g[1]} "h"; "g"
\ar @/^/  @{=>}^{h_C} "i"; "h"
}
\end{displaymath}

\begin{myproposition}
\label{properties of products implies symmeticity}
 i) Given the data  {\bf P1-3)} and {\bf B1-5)}, the formula 
\begin{equation}
\beta_{E,i}(a_n,b_n , c_{n-1})= (\beta_{A,i}(a_n), \beta_{B,i}(b_n), 
-\beta_{C,i}( c_{n-1})- u_i(a_n) +v_i(b_n))
\nonumber 
\end{equation}
defines a morphism of complexes 
\begin{equation}
\beta_{E,i}\,:\, E_i^{\bullet} \rightarrow E_i^{\bullet}[1]
\nonumber 
\end{equation} 
such that for any $r\in R$ the diagram 
\begin{displaymath}
\xymatrix{
E_1^{\bullet}\otimes E_2^{\bullet} \ar[rr]^{\beta_{E,1}\otimes \id} 
\ar[d]^{\id \otimes \beta_{E,2}}
& 
&E_1^{\bullet}[1]\otimes E_2^{\bullet} \ar[d]^{\cup_{r,h}[1]}\\
E_1^{\bullet}\otimes E_2^{\bullet}[1] \ar[rr]^{\cup_{r,h}[1]}
&
&E_3^{\bullet}[1]
}
\end{displaymath}
commutes up to homotopy.

ii) Given the data {\bf P1-3)}, {\bf T1-7)} and {\bf B1-5)},
for each $r\in R$ the diagram
\begin{displaymath}
\xymatrix{
E_1^{\bullet}\otimes E_2^{\bullet}\ar[d]^{s_{12}} \ar[rr]^{\beta_{E,1}\otimes \id}
&
& E_1^{\bullet}[1]\otimes E_2^{\bullet} \ar[rr]^{\cup_{r,h}[1]}
&
&E_3^{\bullet}[1]\ar[rr]^{\mathcal T_3[1]}
&
&E_3^{\bullet}[1] \ar[d]^{\id}\\
E_2^{\bullet}\otimes E_1^{\bullet} \ar[rr]^{\beta_{E,2}\otimes \id}
&
&E_2^{\bullet}[1]\otimes E_1^{\bullet} \ar[rr]^{\mathcal T_2[1]\otimes \mathcal T_1}
&
&E_2^{\bullet}[1]\otimes E_1^{\bullet} \ar[rr]^{\cup'_{1-r,h'}[1]}
&
&E_3^{\bullet}[1]
}
\end{displaymath}
is commutative up to a  homotopy.
\end{myproposition} 
\begin{proof} See \cite{Ne06}, Propositions 1.3.9 and 1.3.10.
\end{proof}

\section{Cohomology of $(\Ph, \Gamma_K)$-modules}
\label{section cohomology of phi-Gamma modules}

\subsection{$(\Ph,\Gamma_K)$-modules} 
\subsubsection{} Throughout this section, $K$ 
denotes a finite extension of $\Qp.$ Let $k_K$ be  the 
 residue field of $K$, $O_K$ its ring of integers
 and $K_0$ the maximal unramified subfield of $K.$ 
We denote by $K_0^{\ur}$  the maximal unramified extension of $K_0$ and by $\sigma$  the absolute Frobenius acting on   $K_0^{\ur}$. Fix an algebraic closure  
$\overline K$  of $K$ and set $G_K=\text{Gal}(\overline K/K).$ 
Let  $\mathbf C_p$ be  the $p$-adic completion of $\overline K.$  
We denote by   $v_p\,\,:\,\,\mathbf C_p \rightarrow \mathbb R\cup\{\infty\}$ the $p$-adic valuation on  $\mathbf C_p$ normalized so that $v_p(p)=1$ and 
set $\vert x\vert_p=\left (\frac{1}{p}\right )^{v_p(x)}.$
 Write $A(r,1)$ for the $p$-adic annulus 
\[A(r,1)=\{ x\in \mathbf C_p \,\mid \, r\leqslant \vert x\vert_p <1 \}.
\]

%As usual,  $\mu_{p^n}$  denotes the group of $p^n$-th roots of  unity.
Fix a system of primitive $p^n$-th roots of unity   $\ep=(\zeta_{p^n})_{n\geqslant 0}$
%$\,\zeta_{p^n} \in \mu_{p^n} $
 such that $\zeta_{p^{n+1}}^p=\zeta_{p^{n}}$ for all $n\geqslant 0$.
Let
$K^{\cyc}= \bigcup_{n=0}^{\infty}K(\zeta_{p^n})$,
$H_K=\Gal (\overline K/K^{\cyc})$, 
$\Gamma_K =\G(K^{\cyc}/K)$
and let  $\chi_K \,:\,\Gamma_K \rightarrow \mathbb Z_p^*$ denote the cyclotomic character.

Recall the constructions of some of Fontaine's rings of $p$-adic periods. 
Define
\[
\widetilde {\mathbf E}^+=
\varprojlim_{x\mapsto x^p} O_{\mathbf C_p}/\,p\,O_{\mathbf C_p}\,=\,\{x=(x_0,x_1,\ldots ,x_n,\ldots )\,\mid \,
x_i^p=x_i, \quad \forall i\in \N\}.
\]
Let $x=(x_0,x_1,\ldots )\in \widetilde {\mathbf E}^+.$ For each $n,$
choose a lift  $\hat x_n\in O_{\mathbf C_p}$ of $x_n$.
Then, for all $m\geqslant 0,$ the sequence $\hat x_{m+n}^{p^n}$ converges to 
$x^{(m)}=\lim_{n\to \infty} \hat x_{m+n}^{p^n}\in O_{\mathbf C_p},$
which does not depend on the choice of  lifts.
 The ring $\widetilde {\mathbf E}^+,$ equipped with the valuation $v_{\mathbf E}(x)=v_p(x^{(0)}),$ is a
complete local ring of characteristic $p$ with  residue field
 $\bar k_K$. Moreover, it is integrally closed in its field
of fractions $\widetilde {\mathbf E}=\mathrm{Fr}(\widetilde {\mathbf E}^+)$. 

Let $\widetilde \A=W(\widetilde \bE)$ be the ring of Witt vectors with coefficients
in $\widetilde \bE$. Denote by $[\,\cdot\,]\,:\,\widetilde \bE \rightarrow W(\widetilde \bE)$  the Teichm\"uller lift.
Each $u=(u_0,u_1,\ldots )\in \widetilde \A$ can be written in the form
$$
u=\underset{n=0}{\overset{\infty}{\sum}} [u_n^{p^{-n}}]p^n.
$$

Set $\pi=[\ep]-1$, $\A_{\Qp}^+=\Zp[[\pi]]$ and denote by $\A_{\Qp}$  the $p$-adic completion
of $\A_{\Qp}^+\left [1/{\pi}\right ]$ in $\widetilde \A.$

Let $\widetilde\boB= \widetilde \A\left [{1}/{p}\right ]$,  $\boB_{\Qp}=
\A_{\Qp}\left [{1}/{p}\right ]$ and
let $\boB$ denote  the completion of the maximal unramified extension of $\boB_{\Qp}$ in $\widetilde \boB$.
%Set $\A=\boB\cap \widetilde \A$, $\widetilde \A^+=W(\bE^+)$, $\A^+= \widetilde %\A^+\cap \A$ and $\boB^+=\A^+\left [{1}/{p}\right ].$
All these rings are endowed with natural  actions of the Galois group $G_K$ and   the Frobenius operator  $\Ph$, and we set $\boB_K=\boB^{H_K}.$ Note that 
\begin{align}
&\g (\pi)=(1+\pi)^{\chi_K (\tau)}-1,\qquad \g \in \Gamma_K, \nonumber\\
&\Ph (\pi)=(1+\pi)^p-1. \nonumber
\end{align}

For any $r>0$ define
$$
\widetilde {\mathbf B}^{\dagger,r}\,=\,\left \{ x\in \widetilde
{\mathbf B}\,\,|\,\, \lim_{k\to +\infty} \left (
v_{\E}(x_k)\,+\,\displaystyle \frac{pr}{p-1}\,k\right )\,=\,+\infty
\right \}.
$$
Set ${\mathbf B}^{\dagger,r}=\boB \cap \widetilde\boB^{\dagger,r}$, 
$\boB_{K}^{\dagger,r}=\boB_{K} \cap \boB^{\dagger,r}$, 
${\mathbf B}^{\dagger}= \displaystyle\bigcup_{r>0} \boB^{\dagger,r}$
and $\mathbf B^\dag_K=\displaystyle\bigcup_{r>0}  \boB_{K}^{\dagger,r}$.

Let $L$ denote the maximal unramified subextension of $K^{\cyc}/\Qp$ and let
$e_K= [K^{\cyc} :L^{\cyc}].$ 
It can be shown (see \cite{CC1}) that there exists $r_K \geqslant 0$ 
and $\pi_K \in \boB_K^{\dagger, r_K}$ such that for all $r\geqslant r_K$ 
the ring $\boB_K^{\dagger, r}$ has the following explicit description

\begin{multline}
\boB_{K}^{\dagger,r}=\left \{ f(\pi_K)=\sum_{k\in \mathbb Z}
a_k\pi^k_K\,\mid \, \text{\rm $a_k\in L$ and $f$ is holomorphic} \right.
\nonumber
\\ 
\left. \text{\rm and bounded on $A(p^{-1/e_Kr},1)$} \right \}.
\end{multline}
Note that, if $K/\Qp$ is unramified, $L=K_0$ and one can take $\pi_{K}=\pi.$ 

Define
\begin{multline}
\boB^{\dag,r}_{\text{rig},K}\,=\,\left \{ f(\pi_K)=\sum_{k\in \mathbb Z}
a_k\pi_K^k\,\mid \, \text{\rm $a_k\in L$ and $f$ is holomorphic} \right.
\nonumber 
\\
\left. \text{\rm on  $A(p^{-1/e_Kr},1)$} \right \}.
\end{multline}
%Let  $\boB_K^{\dagger}= \underset{r\geqslant p-1}\to \cup \boB^{\dag,r}_{K}$,
%$\bold A_K^{\dag}=\{f(\pi) = \sum_{k\in  Z}
%a_k\pi^k\,\mid \,f(\pi)\in \boB_K^{\dagger}\,\, \text{\rm and} \,\, a_k\in O_K\}$
The rings $\boB_{K}^{\dagger,r}$  and $\boB^{\dag,r}_{\text{rig},K}$
are stable under $\Gamma_K,$ and the Frobenius   $\Ph$ sends $\boB_{K}^{\dagger,r}$ into
$\boB_{K}^{\dagger,pr}$ and $\boB^{\dag,r}_{\text{rig},K}$ into 
$\boB^{\dag,pr}_{\text{rig},K}.$  The ring    
\[\CR_K =\underset{r\geqslant r_K}{\bigcup} \, \boB^{\dag,r}_{\mathrm{rig},K}
\]
is isomorphic to the Robba ring over $L.$ Note that it is stable under 
$\Gamma_K$ and $\Ph.$ 
As usual, we set
$$
t=\log (1+\pi)=\underset{n=1}{\overset{\infty}{\sum}} (-1)^{n+1} \frac{\pi^n}{n} \in 
\CR_{\Qp}.
$$
Note that $\Ph (t)=pt$ and $\g (t)=\chi_K (\g) t$, $\g \in \Gamma_K.$

To simplify notation, for each $r\geqslant r_K$ we set $\CR^{(r)}_K=\boB^{\dag,r}_{\text{rig},K}.$
The ring $\CR^{(r)}_K$ is equipped with a canonical Fr\'echet topology 
(see \cite{Ber02}).  Let $A$ be an affinoid algebra over $\Qp.$ Define 
\[
\CR_{K,A}^{(r)}=A\widehat\otimes_{\Qp} \CR^{(r)}_K, \qquad 
\CR_{K,A}=\underset{r\geqslant r_K}\cup \CR_{K,A}^{(r)}.
\]
If the field $K$ is clear from the context, we will often write $\CR_{A}^{(r)}$ instead $\CR_{K,A}^{(r)}$
and $\CR_{A}$ instead  $\CR_{K,A}.$

\begin{definition} i) A $(\Ph,\Gamma_K)$-module over $\CR_{A}^{(r)}$ is 
a finitely generated projective $\CR_{A}^{(r)}$-module $\bD^{(r)}$
equipped with the following structures:

a) A $\Ph$-semilinear map 
\[
\bD^{(r)} \rightarrow \bD^{(r)}\otimes_{\CR_{A}^{(r)}}\CR_{A}^{(pr)}
\]
such that the induced linear map 
\[
\Ph^* \,\,:\,\, \bD^{(r)} \otimes_{\CR_{A}^{(r)}, \Ph}\CR_{A}^{(pr)}
 \rightarrow \bD^{(r)}\otimes_{\CR_{A}^{(r)}}\CR_{A}^{(pr)}
 \]
is an isomorphism of $\CR_{A}^{(pr)}$-modules;

b) A semilinear continuous action of $\Gamma_K$ on $\bD^{(r)}.$

ii) $\bD$ is a $(\Ph,\Gamma_K)$-module over $\CR_A$ 
if  $\bD=\bD^{(r)}\otimes_{\CR_{A}^{(r)}} \CR_A$ for some 
$(\Ph,\Gamma_K)$-module $\bD^{(r)}$ over $\CR_A^{(r)},$ with $r\geqslant r_K.$
\end{definition}

If $\bD$ is a $(\Ph,\Gamma_K)$-module over $\CR_A,$ we  write
$\bD^*=\mathrm{Hom}_{\CR_A}(\bD,A)$  for the dual $(\Ph,\Gamma)$-module.
Let $\mathbf M^{\Ph,\Gamma}_{\CR_A}$ denote  the $\otimes$-category of 
$(\Ph,\Gamma_K)$-modules over $\CR_A.$ 

\subsubsection{}
A $p$-adic representation of $G_K$ with coefficients in an affinoid 
$\Qp$-algebra  $A$ is a 
finitely generated  projective $A$-module equipped with a continuous $A$-linear action 
of $G_K.$ Note that, as $A$ is a noetherian ring, a finitely generated 
$A$-module is projective if and only if it is flat. 
Let $\Rep_A(G_K)$ denote the $\otimes$-category of $p$-adic 
representations with coefficients in $A.$ 
The relationship between $p$-adic representations and $(\Ph,\Gamma_K)$-modules first appeared in the pioneering paper of Fontaine \cite{Fo90}.
The key result of this theory is the following theorem.

\begin{mytheorem}[{\sc Fontaine, Cherbonnier--Colmez, Kedlaya}] Let $A$ be an affinoid algebra over $\Qp.$ 

i) There exists a fully faithul functor 
\[
\DdagrigA \,\,:\,\,\Rep_A(G_K) \rightarrow  \mathbf M^{\Ph,\Gamma}_{\CR_A},
\]
which commutes with base change. More precisely, let  $\mathscr X=\mathrm{Spm}(A).$
For each  $x\in \mathscr X,$  denote by $\mathfrak m_x$ the maximal ideal 
of $A$ associated to $x$ and set  $E_x=A/\mathfrak m_x.$ If $V$ (resp. $\bD$)
is an object of $\Rep_A(G_{\Qp})$ (resp. of 
$\mathbf M^{\Ph,\Gamma}_{\CR_A}$), set $V_x=V\otimes_AE_x$
(resp. $\bD_x=\bD\otimes_AE_x$). 
Then the diagram
\[
\xymatrix
{\Rep_A(G_{\Qp}) \ar[r]^-{\DdagrigA} \ar[d]^{\otimes E_x} &\mathbf M^{\Ph,\Gamma}_{\CR_A}
\ar[d]^{\otimes E_x}
\\
\Rep_{E_x}(G_{\Qp}) \ar[r]^-{\bD_{\mathrm{rig},E_x}^{\dagger}}  &\mathbf M^{\Ph,\Gamma}_{\CR_{E_x}}
}
\]
commutes, 
{\it i.e.}, 
$
\DdagrigA (V)_x\simeq \Ddagrig (V_x).
$

ii) If $E$ is a finite extension of $\Qp,$ then the essential image of 
$\DdagrigE$ is the subcategory of $(\Ph,\Gamma_K)$-modules of slope $0$ 
in the sense of Kedlaya \cite{Ke04}

\end{mytheorem}
\begin{proof} This follows from the main results of \cite{Fo90}, \cite{CC1}
and \cite{Ke04}. See also \cite {Cz08}. 

\end{proof}
\subsubsection{\bf Remark.} Note that in general  the essential image of $\DdagrigA$ does not
coincide with the subcategory of \'etale modules. See \cite{BCz}
\cite{KPX}, \cite{Hel12} for further discussion. 
\newpage
\subsection{Relation to the $p$-adic Hodge theory}
\subsubsection{}
In \cite{Fo90}, Fontaine proposed to classify  the $p$-adic representations
arising in the $p$-adic Hodge theory in terms of $(\Ph,\Gamma_K)$-modules (Fontaine's
program). More precisely, the problem is  to recover  classical Fontaine's functors
$\Dd (V),$ $\Dst (V)$ and
$\Dc (V)$ (see for example \cite{Fo94b}) from $\Ddagrig (V).$ The complete solution was
obtained by Berger in \cite{Ber02}, \cite{Ber08}. His theory also  allowed him 
to prove that each de Rham representation is potentially semistable. 
 In this subsection, we 
review some of  results of Berger.
See also   \cite{Cz03} for introduction and relation to the
theory of $p$-adic differential equations.
Let $E$ be a fixed finite extension of $\Qp.$

\begin{definition} i) A filtered  module over $K$ with coefficients in $E$
is a free
\linebreak
$K\otimes_{\Qp}E$-module $M$ of finite rank equipped with a decreasing 
exhaustive filtration $(\F^iM)_{i\in \Z}.$  We denote by  $\mathbf M\mathbf F_{K,E}$
the $\otimes$-category of such modules. 

ii) A filtered $(\Ph,N)$-module over $K$ with coefficients in $E$ is a free
\linebreak
$K_0\otimes_{\Qp}E$-module $M$ of finite rank equipped with the following structures:

a) An exhaustive decreasing filtration $(\F^iM_K)_{i\in \Z}$ on
$M_K=M\otimes_{K_0}K$;

b) A $\sigma$-semilinear bijective operator $\Ph\,:\,M\rightarrow M$;

c) A $K_0\otimes_{\Qp}E$-linear operator $N$ such that $N\,\Ph=p\,\Ph N.$

iii) A filtered $\Ph$-module over $K$ with coefficients in $E$ is 
a filtered $(\Ph,N)$-module  such that $N=0.$

We denote by  $\mathbf M\mathbf F^{\Ph,N}_{K,E}$ the $\otimes$-category 
of filtered $(\Ph,N)$-module over $K$ with coefficients in $E$
and by $\mathbf M\mathbf F^{\Ph}_{K,E}$ the category of filtered $\Ph$-modules. 

iv) If $L/K$ is a finite Galois  extension and $G_{L/K}=\Gal (L/K),$
then a filtered $(\Ph,N, G_{L/K})$-module is a filtered 
$(\Ph,N)$-module $M$ over $L$ equipped with a semilinear action of $G_{L/K}$ 
 such that the filtration $(\F^iM_L)_{i\in \Z}$ is stable under 
the action of $G_{L/K}.$

v)  We say that $M$ is a filtered $(\Ph,N,G_K)$-module if $M=K_0^{\mathrm{ur}}\otimes_{L_0}M',$ where $M'$ is a filtered  $(\Ph,N, G_{L/K})$-module for some $L/K.$ We denote by
$\mathbf M\mathbf F^{\Ph,N, G_K}_{K,E}$ the $\otimes$-category of 
$(\Ph,N,G_K)$-modules. 
\end{definition}

%Consider the following categories:

%$\bullet$ The categoryA filtered $(\Ph,N)$-module over $K$ of finite dimensional $E$-vector spaces $M$ equipped
%with an exhaustive decreasing filtration $(\F^iM)_{i\in\mathbf Z}$;
%A filtered $(\Ph,N)$-module over $K$

%$\bullet$ The category  $\mathbf M\mathbf F^{\Ph,N}_E$ of   
%finite dimensional $E$-vector spaces $M$ equipped
%with an exhaustive decreasing filtration $(\F^iM)_{i\in\mathbf Z},$ 
%a linear  bijective Frobenius  
%map $\Ph \,:\,M\rightarrow M$   and a nilpotent operator (monodromy) 
%$N\,:\,M\rightarrow M$ such that $\Ph N=p\,\Ph N.$ 

%$\bullet$ The subcategory $\mathbf M\mathbf F^{\Ph}_E$ of $\mathbf M\mathbf F^{\Ph,N}_E$
%formed by filtered $(\Ph,N)$-modules $M$ such that $N=0$ on $M$.

Let  $K^{\cyc}((t))$
denote the ring of formal Laurent power series with coefficients in $K^{\cyc}$
equipped with the filtration 
$\F^iK^{\cyc}((t))=t^iK^{\cyc}[[t]]$ and the action of $\Gamma_K$ given by
\[
\g \left (\displaystyle\underset{k\in\mathbf Z}\sum a_kt^k \right )=
\displaystyle\underset{k\in\mathbf Z}\sum \g (a_k)\chi_K (\g)^kt^k,
\qquad \g \in \Gamma_K.
\]
The ring $\CR_{K,E}$ can not be naturally embedded in 
$E\otimes_{\Qp} K^{\cyc}((t)),$
but for any $r\geqslant r_K$   there exists a $\Gamma_K$-equivariant  embedding 
$i_n\,:\,\CR_{K,E}^{(r)}\rightarrow E\otimes_{\Qp} K^{\cyc}((t))$ which sends  $\pi$
to $\zeta_{p^n}e^{t/p^n}-1.$ 
Let $\bD$ be a $(\Ph,\Gamma_K)$-module over $\CR_{K,E}$ 
and let $\bD=\bD^{(r)}\otimes_{\CR_{K,E}^{(r)}}\CR_{K,E}$
for some $r\geqslant r_K.$
 Then
\[
\mathscr D_{\dR/K} (\bD)=\left ( E\otimes_{\Qp} K^{\cyc}((t))\otimes_{i_n}\bD^{(r)} \right )^{\Gamma_K}
\]
is a free $E\otimes_{\Qp}K$-module of finite rank equipped with a decreasing filtration
\[
\F^i\mathscr D_{\dR/K} (\bD)=\left ( E\otimes_{\Qp} \F^iK^{\cyc}((t))\otimes_{i_n}\bD^{(r)} \right )^{\Gamma_K},
\]
which does not depend on the choice of $r$ and  $n.$ 

Let $\CR_{K,E}[\log{\pi}]$ denote the ring of power series in variable 
$\log{\pi}$ with coefficients 
in $\CR_{K,E}.$ Extend the actions of  $\Ph$ and $\Gamma_K$ to $\CR_{K,E}[\log{\pi}]$ setting
\begin{align*}
&\Ph (\log{\pi}) =p\log{\pi}+ \log \left (\frac{\Ph (\pi)}{\pi^p} \right ),
\\
&\g (\log{\pi}) =\log{\pi} + \log \left (\frac{\g (\pi)}{\pi} \right ),\quad \g\in \Gamma_K.
\end{align*}
( Note that $\log (\Ph (\pi)/\pi^p)$ and 
$\log ( \tau (\pi)/\pi)$  converge in $\CR_{K,E}$.) Define a monodromy operator 
$
N\,:\,\CR_{K,E}[\log{\pi}]\rightarrow \CR_{K,E}[\log{\pi}]
$
by  
\[N = -\displaystyle\left (1-\frac{1}{p} \right )^{-1}\displaystyle 
\frac{d}{d \log{\pi}}.
\]
For any $(\Ph,\Gamma_K)$-module $\bD$ define 
\begin{eqnarray}
&&\mathscr D_{\st/K} (\bD)= \left (\bD\otimes_{\CR_{K,E}} \CR_{K,E}[\log{\pi}, 1/t  ] \right )^{\Gamma_K}, \quad t=\log (1+\pi),
\nonumber \\
&&\mathscr D_{\cris/K} (\bD)= \CDst (\bD)^{N=0}= (\bD [1/t])^{\Gamma_K}. \nonumber
\end{eqnarray}
Then $\CDst (\bD)$ is a free  $E\otimes_{\Qp}K_0$-module of finite rank equipped with  natural actions of $\Ph$ and $N$ such that $N\Ph=p\,\Ph N.$ Moreover, it is equipped with a canonical exhaustive decreasing filtration induced by the embeddings $i_n.$
If $L/K$ is a finite extension and $\bD$ is a $(\Ph,\Gamma_K)$-module, 
the tensor product  $\bD_L=\CR_{L,E}\otimes_{\CR_{K,E}}\bD$ has a natural
structure of a $(\Ph,\Gamma_L)$-module, and we define
\begin{equation}
\nonumber
\mathscr D_{\mathrm{pst}/K} (\bD)=\underset{L/K}\varinjlim \mathscr D_{\st/L}(\bD_L).
\end{equation}
Then $\mathscr D_{\mathrm{pst}/K} (\bD)$ is a $K_0^{\mathrm{ur}}$-vector space equipped 
with natural actions of $\Ph$ and $N$ and a discrete action of $G_K.$
Therefore, we have 
four functors
\begin{align}
&\mathscr D_{\dR/K} \,:\,{\mathbf M}^{\Ph,\Gamma}_{\CR_{K,E}} \rightarrow \mathbf M\mathbf F_{K,E},
\nonumber \\
&\mathscr D_{\st/K} \,:\,{\mathbf M}^{\Ph,\Gamma}_{\CR_{K,E}} \rightarrow
\mathbf M\mathbf F^{\Ph,N}_{K,E},
\nonumber \\ 
&\mathscr D_{\mathrm{pst}/K} \,:\,{\mathbf M}^{\Ph,\Gamma}_{\CR_{K,E}} \rightarrow
\mathbf M\mathbf F^{\Ph,N,G_K}_{K,E}, \nonumber \\
&\mathscr D_{\cris/K} \,:\,{\mathbf M}^{\Ph,\Gamma}_{\CR_{K,E}} \rightarrow 
\mathbf M\mathbf F^{\Ph}_{K,E}.
\nonumber
\end{align}
If the field $K$ is fixed and understood from context, we will omit it 
and simply write $\mathscr D_{\dR},$ $\mathscr D_{\st},$ 
$\mathscr D_{\mathrm{pst}}$ and $\mathscr D_{\cris}.$

\begin{mytheorem}[\text{\sc Berger}]
\label{berger theorem1}
 Let $V$ be a $p$-adic representation of $G_{K}.$
Then 
\[
\bD_{*/K}(V)\simeq \mathscr D_{*/K}(V),\qquad *\in\{\text{\rm dR},\text{\rm st},\text{\rm pst},\text{\rm cris}\}.
\]
\end{mytheorem}
\begin{proof} See \cite{Ber02}.
\end{proof}

For any $(\Ph,\Gamma_K)$-module over $\CR_{K,E}$ one has 
\[
\textrm{rk}_{E\otimes K_0} \mathscr D_{\cris/K} (\bD)\leqslant 
\textrm{rk}_{E\otimes K_0}\mathscr D_{\st/K} (\bD) \leqslant
\textrm{rk}_{E\otimes K_0} \mathscr D_{\dR/K}(\bD)\leqslant 
 \textrm{rk}_{\CR_{K,E}} (\bD).
\]

\begin{definition} One says that $\bD$ is de Rham (resp. semistable, resp. potentially semistable, resp. crystalline) if 
$\mathrm{rk}_{E\otimes K_0} \mathscr D_{\dR/K} (\bD)= \mathrm{rk}_{\CR_{K,E}} (\bD)$
(resp. $\mathrm{rk}_{E\otimes K_0} \mathscr D_{\st/K} (\bD)  = \mathrm{rk}_{\CR_{K,E}} (\bD)$, resp.
$\mathrm{rk}_{E\otimes K_0} \mathscr D_{\mathrm{pst}/K} (\bD)= \mathrm{rk}_{\CR_{K,E}} (\bD)$),
resp. $\mathrm{rk}_{E\otimes K_0} \mathscr D_{\cris/K} (\bD)= \mathrm{rk}_{\CR_{K,E}} (\bD)$).
\end{definition}
Let
$\mathbf M^{\Ph,\Gamma}_{\CR_E,\textrm{st}},$ $\mathbf M^{\Ph,\Gamma}_{\CR_E,\textrm{pst}}$ and $\mathbf M^{\Ph,\Gamma}_{\CR_E,\cris}$
denote  the categories
of semistable, potentially semistable  and crystalline 
$(\Ph,\Gamma)$-modules  respectively.
If $\bD$ is de Rham, the jumps of the filtration 
$\F^i\mathscr D_{\dR} (\bD)$ will be  called the  Hodge--Tate weights of $\bD.$ 

\begin{mytheorem}[\text{\sc Berger}]
\label{berger theorem2}
i) The functors 
\begin{align}
\label{equivalence Dst and Dcris for phi-gamma mod}
&\CDst \,:\,\mathbf M^{\Ph,\Gamma}_{\CR_{K,E},\mathrm{st}} \rightarrow 
\mathbf M\mathbf F^{\Ph,N}_{K,E}, \nonumber \\
&\mathscr D_{\mathrm{pst}} \,:\,{\mathbf M}^{\Ph,\Gamma}_{\CR_{K,E},\mathrm{pst}} \rightarrow
\mathbf M\mathbf F^{\Ph,N,G_K}_{K,E}, \nonumber \\ 
&\CDcris \,:\,\mathbf M^{\Ph,\Gamma}_{\mathscr R_{K,E},\mathrm{cris}}\rightarrow \mathbf M\mathbf F^{\Ph}_{K,E}  
\nonumber
\end{align}
are equivalences of $\otimes$-categories.

ii) Let $\bD$ be a $(\Ph,\Gamma_K)$-module. Then $\bD$ is potentially semistable if and only if 
$\bD$ is de Rham.
\end{mytheorem}
\begin{proof}
These results are  proved in \cite{Ber08}. See  Theorem A,  
Theorem III.2.4 and Theorem V.2.3 of {\it op. cit.}.
\end{proof}

%\subsection{Isoclinic $(\Ph,\Gamma)$-modules}

%\begin{mylemma} Let $\bD$ be a semistable  $(\Ph,\Gamma_K)$-module over 
%$\CR_{K,E}$ such that
%$\CDst (\bD)^{\Ph=1}=\CDst (\bD)$ and $\F^0\CDdr (\bD)=\CDdr (\bD).$
%Then $\bD$ is crystalline and decomposes (non-canonically) into a direct sum 
%\[
%\bD \simeq \underset{i=1}{\overset{d}\bigoplus} \CR_{K,E} (x^{-n_i}),\qquad n_i\geqslant 0.
%\]
%\end{mylemma}
%\begin{proof} From $N\,\Ph=p\Ph\,N$ it follows that $N=0$ on $\CDst (\bD),$ and
%therefore $\bD$ is crystalline. 
%We prove the lemma by induction on $d$. Assume that $d=1$. Then
%$\CDcris (\bD)$ is a one dimensional $K_0$-vector space equipped with a trivial action of
%$\Ph.$ 
%Let $n\geqslant 0$ be the unique jump of Hodge filtration on $\CDdr (\bD).$
%It is clear that  $\CDst (\bD)=\CDst (\CR_{K,E}(x^{-n})),$ and therefore
%$\bD\simeq \CR_{K,E}(x^{-n})$ by Theorem~\ref{berger theorem2}. 

%Now let $\bD$ be a $(\Ph,\Gamma_K)$-module of rank $d>1$  satisfying the 
%conditions of Lemma. Let $n_0$ be the lowest Hodge-Tate weight of $\bD.$
%Take a  submodule $M_0 \subset \F^{n_0}\CDcris (\bD)$ of rank one 
%such that $M_1=\CDcris (\bD)/M_0$ is free 
%\end{proof}
\newpage
\subsection{The complex $C_{\Ph,\g_K}(\bD)$}
\subsubsection{} We keep previous notation and conventions.
Set $\Delta_K=\Gal (K(\zeta_p)/K).$ Then $\Gamma_K=\Delta_K\times \Gamma_{K}^{\,0},$
where $\Gamma_{K}^{\,0}$ is a pro-$p$-group isomorphic to $\Zp.$
Fix a topological  generator $\g_K$ of $\Gamma_K^{\,0}.$
For each    $(\Ph,\Gamma_K)$-module $\bD$ over  $\CR_A=\CR_{K,A}$ 
define
\[
C_{\g_K}^{\bullet}(\bD)\,\,:\,\,\bD^{\Delta_K} \xrightarrow{\g_K-1} 
\bD^{\Delta_K} \,,
\]
where the first term is placed in degree $0.$ If $\bD'$ and $\bD''$ are two $(\Ph,\Gamma_K)$-modules, we will denote by 
\[
\cup_{\g}\,\,:\,\,C_{\g_K}^{\bullet}(\bD') \otimes C_{\g_K}^{\bullet}(\bD'')
\rightarrow C_{\g_K}^{\bullet}(\bD'\otimes \bD'')
\]
the bilinear map 
\begin{equation}
\nonumber 
\cup_{\g}(x_n\otimes y_m)=
\begin{cases} x_n\otimes \g_K^n(y_m) &{\text{if $x_n\in C_{\g_K}^n(\bD'),$ $y_m\in C_{\g_K}^m(\bD''),$}}\\&{\text{and $n+m=0$ or $1$},}\\
0 &\text{if $n+m\geqslant 2.$}
\end{cases}
\end{equation}
Consider the total complex 
\begin{equation}
C_{\Ph,\g_K}^{\bullet}(\bD)=\Tot \left (C_{\g_K}^{\bullet}(\bD)
\xrightarrow{\Ph-1} C_{\g_K}^{\bullet}(\bD) \right ).
\nonumber
\end{equation}
More explicitly, 
\begin{equation}
C^{\bullet}_{\Ph,\g_K}(\bD)\,\,:\,\,
0\rightarrow \bD^{\Delta_K} \xrightarrow{d_0}\bD^{\Delta_K}\oplus \bD^{\Delta_K} 
\xrightarrow{d_1}\bD^{\Delta_K} \rightarrow 0,
\nonumber 
\end{equation}
where $d_0(x)=((\Ph-1)\,x,(\g_K-1)\,x)$ and $d_1(x,y)=(\g_K-1)\,x-(\Ph-1)\,y.$  Note that  $C_{\Ph,\gamma_K}^{\bullet}(\bD)$ coincides with the  complex of Fontaine--Herr 
(see \cite{H1}, \cite{H2} \cite{Li07}).
We will write $H^*(\bD)$ for the cohomology of $C^{\bullet}_{\Ph,\g}(\bD).$
If $\bD'$ and $\bD''$ are two $(\Ph,\Gamma_K)$-modules, the cup product $\cup_{\g}$
induces, by Proposition~\ref{cup-product for T(A)},   a bilinear map
\begin{equation}
\cup_{\Ph,\g}\,:\,C^{\bullet}_{\Ph,\g_K}(\bD') \otimes C^{\bullet}_{\Ph,\g_K} (\bD'')
\rightarrow  C^{\bullet}_{\Ph,\g_K}(\bD'\otimes \bD'').
\nonumber
\end{equation}
Explicitly 
\begin{equation}
\cup_{\Ph,\g}((x_{n-1},x_n)\otimes (y_{m-1},y_m))=
(x_n\cup_{\g}y_{m-1}+(-1)^m x_{n-1}\cup_{\g}\Ph (y_m), x_n\cup_{\g}y_m),
\nonumber
\end{equation}
if $(x_{n-1},x_{n})\in C^{n}_{\Ph,\g_K}(\bD')=C_{\g_K}^{n-1}(\bD')\oplus C_{\g_K}^n(\bD')$ and
$(y_{m-1},y_{m})\in C^{m}_{\Ph,\g}(\bD'')=C_{\g}^{m-1}(\bD'')\oplus C_{\g}^m(\bD'').$
An easy computation gives the following formulas 
\begin{eqnarray}
&&{\begin{cases} C^0_{\Ph,\g_K}(\bD')\otimes C^0_{\Ph,\g_K}(\bD'') \rightarrow C^0_{\Ph,\g_K}(\bD'\otimes \bD''),\\
x_0\otimes y_0 \mapsto x_0\otimes y_0,
\end{cases}} \nonumber\\
&& \nonumber\\
&&{\begin{cases} C^0_{\Ph,\g_K}(\bD')\otimes C^1_{\Ph,\g_K}(\bD'')\rightarrow C^1_{\Ph,\g_K}(\bD'\otimes \bD''),\\
x_0\otimes (y_0,y_1) \mapsto (x_0\otimes y_0, x_0\otimes y_1),
\end{cases}} \nonumber\\
&& \nonumber \\
&&{\begin{cases} C^1_{\Ph,\g_K}(\bD')\otimes C^0_{\Ph,\g_K}(\bD'') \rightarrow C^1_{\Ph,\g_K}(\bD'\otimes \bD''),\\
(x_0,x_1) \otimes y_0 \mapsto (x_0\otimes \Ph(y_0), x_1\otimes \g_K(y_0)),
\end{cases}} \nonumber \\
&& \nonumber \\
&&{\begin{cases} C^1_{\Ph,\g_K}(\bD')\otimes C^1_{\Ph,\g_K}(\bD'') \rightarrow C^2_{\Ph,\g_K}(\bD'\otimes \bD''),\\
(x_0,x_1) \otimes (y_0,y_1)  \mapsto (x_1\otimes \g_K(y_0)- x_0\otimes \Ph (y_1)),
\end{cases}} \nonumber \\
&& \nonumber  \\
&&{\begin{cases} C^0_{\Ph,\g_K}(\bD')\otimes C^2_{\Ph,\g_K}(\bD'') \rightarrow C^2_{\Ph,\g_K}(\bD'\otimes \bD''),\\
x_0\otimes y_1 \mapsto  x_0\otimes y_1,
\end{cases}} \nonumber \\
&& \nonumber  \\
&&{\begin{cases} C^2_{\Ph,\g_K}(\bD')\otimes C^0_{\Ph,\g_K}(\bD'') \rightarrow C^2_{\Ph,\g_K}(\bD'\otimes \bD''),\\
x_1\otimes y_0 \mapsto  x_1\otimes \g_K (\Ph (y_1)).
\end{cases}}
\nonumber
\end{eqnarray}
Here the zero components are omitted.

\subsubsection{} For each $(\Ph,\Gamma_K)$-module $\bD$ we denote by 
\begin{equation}
\nonumber 
\RG (K,\bD)=\left [ C_{\Ph,\gamma_K}^{\bullet}(\bD) \right ]
\end{equation}
the corresponding object of the derived category $\mathcal D (A).$ 
The cohomology of  $\bD$ is defined by 
\begin{equation}
H^{i} (\bD)= \bold R^i \Gamma (K,\bD)= H^{i}(C^{\bullet}_{\Ph,\gamma_K}(\bD)),
\qquad i\geqslant 0.
\nonumber
\end{equation}
There exists a canonical isomorphism in $\mathcal D (A)$
\begin{equation*}
\label{Brauer for phi-Gamma}
\mathrm{TR_K}\,:\,\tau_{\geqslant 2}\RG (K, \CR_A (\chi_K)) \simeq A[-2]
\end{equation*}
( see \cite{H2}, \cite{Li07}, \cite{KPX}).  Therefore, for each $\bD$ we have 
morphisms
\begin{multline}
\label{construction of duality for phi-Gamma}
\RG (K,\bD) \otimes_A^{\bold L} \RG (K,\bD^*(\chi_K)) 
\xrightarrow{\cup_{\Ph,\g}} \RG (K,\bD \otimes \bD^*(\chi_K)) \\ \xrightarrow{\mathrm{duality}} \RG (K, \CR_A (\chi_K))
\rightarrow \tau_{\geqslant2}\RG (K, \CR_A (\chi_K)) \simeq A[-2].
\end{multline}

\begin{mytheorem}[\sc Kedlaya--Pottharst--Xiao] 
\label{Theorem KPX}
Let $\bD$ be a 
$(\Ph,\Gamma_K)$-module over $\CR_{K,A},$ where $A$ is an 
affinoid algebra.  

i) Finiteness. We have  $\RG (K,\bD)\in \mathcal D^{[0,2]}_{\perf}(A).$

ii) Euler--Poincar\'e characteristic formula. We have 

\begin{equation}
\nonumber
\underset{i=0}{\overset{2}\sum} (-1)^i \mathrm{rk}_A H^i(\bD)=
- [K:\Qp]_,\mathrm{rk}_{\CR_{K,A}}(\bD).
\end{equation}

iii) Duality.
The morphism  (\ref{construction of duality for phi-Gamma}) induces an isomorphism
\begin{equation}
\nonumber 
\RG (K, \bD^*) \simeq \bold R\mathrm{Hom}_A (\RG (K, \bD), A).
\end{equation} 
In particular, we have cohomological pairings 
\begin{equation}
\cup \,\,:\,\, H^i(\bD)\otimes H^{2-i}(\bD^*(\chi_K)) \rightarrow  H^{2}(\CR_A(\chi_K))
\simeq A, \qquad i\in \{0,1,2\}.
\nonumber 
\end{equation}

iv) Comparision with Galois cohomology.
Let $V$ is a $p$-adic representation of $G_K$ with coefficients in $A.$
There exist  canonical (up to the choice of $\g_K$) 
and functorial isomorphisms
\begin{equation}
H^{i}(K,V)\iso H^{i}(\Ddagrig (V))
\nonumber 
\end{equation}
which are compatible with cup-products. In particular, we have a commutative diagram
\begin{equation*}
\xymatrix{
H^2(\CR_A(\chi_K)) \ar[d]^{\simeq} \ar[rr]^{\mathrm{TR_K}} &&A \ar[d]^{=}\\
H^2(K, A(\chi_K)) \ar[rr]^{\inv_K} &&A,
}
\end{equation*}
where $\inv_K$ is the canonical isomorphism of the local class field
theory.
\end{mytheorem}
\begin{proof} See Theorem~4.4.5 of \cite{KPX} and Theorem 2.8 of \cite{Po13}.
\end{proof}

\subsubsection{\bf Remark} The explicit construction of the isomorphism 
$\mathrm{TR}_K$ is given in \cite{H2} and \cite{Ben00}, Theorem 2.2.6.

\subsection{The complex $K^{\bullet}(V)$} 
\label{subsection K(V)}
\subsubsection{}
In this section,  we give the derived version of  isomorphisms  
\linebreak
$H^{i}(K,V)\iso H^{i}(\Ddagrig (V))$ of Theorem~\ref{Theorem KPX} iv).
We write 
$C^{\bullet}_{\Ph,\gamma_K}(V)$ instead $C^{\bullet}_{\Ph,\gamma_K}(\Ddagrig (V))$
to simplify notation.
Let $K$ be a finite extension of $\Qp.$
If $M$ is a topological $G_K$-module, we denote by
$C^{\bullet}(G_K,M)$ the complex of continuous cochains with coefficients 
in $M.$ 
Let $V$ be a $p$-adic representation of $G_K$
with coefficients in an affinoid algebra $A.$
Then 
\begin{equation*}
C^{\bullet}(G_K,V) \in \mathcal K^{[0,2]}_{\ft}(A)
\end{equation*}
and for the associated object $\RG (K,V)$ of the derived category 
\begin{equation*}
\RG (K,V)=\left [C^{\bullet}(G_K,V)\right ]\in \mathcal D^{[0,2]}_{\perf}(A)
\end{equation*}
(see \cite{Po13}, Theorem 1.1).

The continuous Galois cohomology of $G_K$ with coefficients in $V$ is defined by
\begin{equation}
H^*(K,V)=H^*(C^{\bullet}(G_K,V)).
\nonumber
\end{equation}
In \cite{Ber02},  Berger constructed, for each $r\geqslant r_K,$ a ring
$\widetilde{\mathbf{B}}^{\dagger, r}_{\mathrm{rig}}$ which is the completion of 
$\mathbf{B}^{\dagger, r}$ with respect to  Frechet topology.
Set $\widetilde{\mathbf{B}}^{\dagger, r}_{\mathrm{rig},A}=
\widetilde{\mathbf{B}}^{\dagger, r}_{\mathrm{rig}}\widehat \otimes_{\Qp} A$
and $\BrdA=\underset{r\geqslant r_K}\cup
\widetilde{\mathbf{B}}^{\dagger, r}_{\mathrm{rig},A}.$ 
For each $r\geqslant r_K$ we have an exact sequence
\begin{equation*}
0\rightarrow \Qp \rightarrow  
\widetilde{\mathbf{B}}^{\dagger, r}_{\mathrm{rig}}
\xrightarrow{\Ph-1}
\widetilde{\mathbf{B}}^{\dagger, rp}_{\mathrm{rig}}
\rightarrow 0
\end{equation*} 
(see \cite{Ber03}, Lemma I.7). Since the  completed tensor product by an orthonormalizable Banach space is exact in the category of Frechet spaces (see, for example, \cite{Bel12}, proof of Lemma 3.9),   the sequence 
\begin{equation}
0\rightarrow A \rightarrow  
\widetilde{\mathbf{B}}^{\dagger, r}_{\mathrm{rig},A}
\xrightarrow{\Ph-1}
\widetilde{\mathbf{B}}^{\dagger, rp}_{\mathrm{rig},A}
\rightarrow 0.
\nonumber
\end{equation}
is also exact. 
Passing to the direct limit we obtain an  exact sequence
\begin{equation}
\label{exact sequence with BrdA}
0\rightarrow A \rightarrow \BrdA \xrightarrow{\Ph-1}\BrdA 
\rightarrow 0.
\end{equation}
Set  $\Vrigdag =V\otimes_A \BrdA$ and consider the complex 
$C^{\bullet}(G_K,\Vrigdag).$ 
Then (\ref{exact sequence with BrdA}) induces an exact sequence
\begin{equation}
0\rightarrow C^{\bullet}(G_K,V)\rightarrow C^{\bullet}(G_K,\Vrigdag)\xrightarrow{\Ph-1} C^{\bullet}(G_K,\Vrigdag)\rightarrow 0.
\nonumber
\end{equation}
Define
\begin{equation}
K^{\bullet}(V)=T^{\bullet}(C^{\bullet}(G_K,\Vrigdag))=
\text{\rm Tot}\left (C^{\bullet}(G_K,\Vrigdag)
\xrightarrow{\Ph-1}C^{\bullet}(G_K,\Vrigdag)
\right ).
\nonumber
\end{equation}
Consider the map 
\begin{equation*}
\alpha_{V}\,\,:\,\, C^{\bullet}_{\g_K}(V) \rightarrow C^{\bullet}(G_K,\Vrigdag)
\end{equation*}
defined by 
\begin{equation}
\nonumber 
\begin{cases} 
\alpha_{V} (x_0)=x_0,  &x_0\in  C^0_{\g_K}(V), \nonumber \\
\alpha_{V} (x_1) (g)=\displaystyle\frac{\g_K^{\kappa (g)}-1}{\g_K-1} (x_1), &x_1\in C^1_{\g_K}(V), \nonumber 
\end{cases}
\end{equation}
where $g\in G_K$ and  $\g_K^{\kappa (g)}= g\vert_{\Gamma_K^0}.$
It is easy to check that $\alpha_{V}$ is a morphism of complexes which commutes with 
$\Ph.$ By fonctoriality, we obtain a morphism (which we denote again by $\alpha_{V}$):
\begin{equation}
\alpha_{V}\,:\,C^{\bullet}_{\Ph,\g_K}(V) \rightarrow  K^{\bullet}(V).
\nonumber
\end{equation}

\begin{myproposition}
\label{quasi-iso with K(V)}
The map $\alpha_{V}\,:\,C^{\bullet}_{\Ph,\g_K}(V) \rightarrow  K^{\bullet}(V)$
and the map 
\begin{eqnarray}
\xi_{V}\,:\,&&C^{\bullet}(G_K,V)\rightarrow K^{\bullet}(V),\nonumber \\
&&x_n\mapsto (0,x_n),\qquad x_n\in C^n(G_K,V) \nonumber
\end{eqnarray}
are quasi-isomorphisms and we have   
a  diagram
\begin{equation}
\xymatrix{
C^{\bullet}(G_K,V)\ar[r]_{\simeq}^{\xi_{V}} &K^{\bullet}(V)\\
& C^{\bullet}_{\Ph,\g_K}(V). \ar[u]^{\simeq}_{\alpha_{V}}}
\nonumber
\end{equation}
\end{myproposition}
\begin{proof}
This is Proposition~9   of \cite{Ben14}. 
\end{proof}

\subsubsection{} If $M$ and $N$ are two  Galois modules, the cup-product 
\begin{equation}
\cup_c\,:\,C^{\bullet}(M)\otimes C^{\bullet}(M)
\rightarrow C^{\bullet}(M\otimes N)
\nonumber
\end{equation}
defined by
\begin{multline}
(x_n\cup_c y_m)(g_1,g_2,\ldots ,g_{n+m})=\\
=x_n(g_1,\ldots ,g_n)\otimes
(g_1g_2\cdots g_n)y_m(g_{n+1},\ldots ,g_{n+m}),
\nonumber
\end{multline}
where $x_n\in C^n(G_K,M)$ and $y_m\in C^m(G_K,N),$
is a morphism of complexes. Let $V$ and $U$ be two Galois representations 
of $G_K.$ Applying  Proposition~\ref{cup-product for T(A)} to the complexes
$C^{\bullet}(G_K,\Vrigdag)$ and $C^{\bullet}(G_K,U^{\dagger}_{\text{\rm rig}})$
we obtain a morphism
\begin{equation}
\cup_K\,:\,K^{\bullet}(V)\otimes K^{\bullet}(U)
\rightarrow K^{\bullet}(V\otimes U).
\nonumber
\end{equation} 

The following proposition will not be used in the remainder of this
paper, but we state it here for completeness.

\begin{myproposition} 
\label{homotopy C and K}
In the  diagram
\begin{equation}
\xymatrix{
C_{\Ph,\g_K}^{\bullet}(V)\otimes C_{\Ph,\g_K}^{\bullet}(U) \ar[rr]^{\cup_{\Ph,\g_K}}\ar[dd]^{\alpha_V\otimes \alpha_U} & &
C_{\Ph,\g}^{\bullet}(V\otimes U) \ar[dd]^{\alpha_{V\otimes U}}\\
& &  \ar @/_/ @{=>}[dl]_{h_{\Ph,\g}}\\
K^{\bullet}(V)\otimes K^{\bullet}(U) \ar[rr]^{\cup_K} & &K^{\bullet}(V\otimes U)
}
\nonumber
\end{equation}
the maps $\alpha_{V\otimes U}\circ \cup_{\Ph,\g}$ and $\cup_K\circ (\alpha_V\otimes \alpha_U)$
are homotopic.
\end{myproposition}
We need the following lemma.

\begin{mylemma} 
\label{lemma for c_x,y}
For any $x\in C^1_{\g_K}(V)$, $y\in C^1_{\g_K}(U),$ let
$c_{x,y}\in C^1(\Gamma_K^0, \Ddagrig (V\otimes U))$ denote the  $1$-cochain  
defined by
\begin{equation}
\label{definition of c_x,y}
c_{x,y}(\g_K^n)=\underset{i=0}{\overset{n-1}\sum}
\g_K^i (x)\otimes \left (\frac{\g_K^n-\g_K^{i+1}}{\g_K-1}\right )\,(y),\qquad
\text{\it if  $n\neq 0,1$,}
\end{equation}
and $c_{x,y}(1)=c_{x,y}(\g_K)=0.$ Then

i) For each  $x\in C^1_{\Ph,\g_K}(V)$ and $y\in C^0_{\Ph,\g_K}(U)$ 
\[
c_{x,(\g_K-1)y}=\alpha_V(x)\cup_c \alpha_U(y)-\alpha_{V\otimes U}(x\cup_{\g}y).
\]

ii) If $x\in C^0_{\Ph,\g_K}(V)$ and $y\in C^1_{\Ph,\g_K}(U)$ then
\[
c_{(\g_K-1)x,y}= \alpha_{V\otimes U}(x\cup_{\g}y)-\alpha_V(x)\cup_c \alpha_U(y).
\]

iii) One has
\[
d^1(c_{x,y})=- \alpha_V(x)\cup_c \alpha_U(y).
\]
\end{mylemma}
\begin{proof}[Proof of the lemma] i) Note that $\Gamma_K^0$ is the 
profinite completion of the cyclic group $\left < \g_K\right >,$ 
and  an easy computation shows that the map $c_{x,y}$,  defined on  
$\left < \g_K\right >$ by (\ref{definition of c_x,y}), extends by
continuity to a unique cochain on $\Gamma_K^0$ which we denote again
by $c_{x,y}.$

For any natural $n\neq 0,1$ one has
\begin{eqnarray}
c_{x,(\g-1)y}(\g_K^n)&&=\underset{i=0}{\overset{n-1}\sum} \g_K^i(x)\otimes 
(\g_K^{n}-\g_K^{i+1})(y)= \nonumber\\
&&=\underset{i=0}{\overset{n-1}\sum}\g_K^i(x)\otimes \g_K^n(y)-
\underset{i=0}{\overset{n-1}\sum}\g_K^i(x)\otimes \g_K^{i+1}(y)=
\nonumber\\
&&=\frac{\g_K^n-1}{\g_K-1}\,(x)\otimes \g_K^n(y)-
\frac{\g_K^n-1}{\g_K-1}(x\otimes \g_K (y))=
\nonumber\\
&&=(g_V(x)\cup_c g_U(y)) (\g_K^n)-(g_{V\otimes U}(x\cup_{\g}y)) (\g_K^n). \nonumber 
\end{eqnarray}
By continuity, this implies that $c_{x,(\g_K-1)y}=\alpha_V(x)\cup_c \alpha_U(y)-
\alpha_{V\otimes U}(x\cup_{\g}y),$ and i) is proved.

ii) An easy induction proves the formula
\begin{multline}
\label{formula in ii) lemma 2}
\underset{i=0}{\overset{m}\sum}\g_K^i(\g_K-1)(x)\otimes \frac{\g_K^{i+1}-1}{\g_K-1}(y)=\\
=\g_K^{m+1}(x)\otimes \frac{\g_K^{m+1}-1}{\g_K-1}(y)-\frac{\g_K^{m+1}-1}{\g_K-1}(x\otimes y).
\end{multline}
Therefore 
\begin{align}
&c_{(\g_K-1)x,y}(\g_K^n)= \underset{i=0}{\overset{n-1}\sum}
(\g_K^{i+1}-\g_K^i)(x)\otimes \frac{\g_K^n-\g_K^{i+1}}{\g_K-1}(y)=
\nonumber\\
&=\underset{i=0}{\overset{n-1}\sum}
(\g_K^{i+1}-\g_K^i)(x)\otimes \frac{\g_K^n-1}{\g_K-1}(y)-
\underset{i=0}{\overset{n-1}\sum}
\g_K^i(\g_K-1)(x)\otimes\frac{\g_K^{i+1}-1}{\g_K-1}(y)
=
\nonumber\\
&\overset{\text{\rm by (\ref{formula in ii) lemma 2})}}=(\g_K^n-1) (x)\otimes \frac{\g_K^n-1}{\g_K-1}(y)+
\frac{\g_K^n-1}{\g_K-1} (x\otimes y)-\g_K^{n}(x)\otimes \frac{\g_K^n-1}{\g_K-1}(y)=
\nonumber\\
&=\frac{\g_K^n-1}{\g_K-1} (x\otimes y)-x\otimes \frac{\g_K^n-1}{\g_K-1}(y)
= \nonumber \\
&=(\alpha_{V\otimes U}(x\cup_{\g}y))(\g_K^n)-(\alpha_V(x)\cup_c \alpha_U(y))(\g_K^n),
\nonumber
\end{align}
and ii) is proved.

iii) One has
\begin{eqnarray}
&&d^1c_{x,y}(\g_K^n,\g_K^m)=\g_K^nc_{x,y}(\g_K^m)-
c_{x,y}(\g_K^{n+m})+c_{x,y}(\g_K^n)
=\nonumber \\ 
&&=\underset{i=0}{\overset{m-1}\sum}\g_K^{n+i}(x)\otimes 
\frac{\g_K^{n+m}-\g_K^{i+n+1}}{\g_K-1}\,(y)-
\nonumber\\
&&-
\underset{i=0}{\overset{n+m-1}\sum}\g_K^{i}(x)\otimes 
\frac{\g_K^{n+m}-\g_K^{i+1}}{\g_K-1}\,(y)+
\underset{i=0}{\overset{n-1}\sum}\g_K^{i}(x)\otimes 
\frac{\g_K^{n}-\g_K^{i+1}}{\g_K-1}\,(y)=
\nonumber\\
&&=-\underset{i=0}{\overset{n-1}\sum}\g_K^{i}(x)\otimes 
\frac{\g_K^{n+m}-\g_K^{i+1}}{\g_K-1}\,(y)+
\underset{i=0}{\overset{n-1}\sum}\g_K^{i}(x)\otimes 
\frac{\g_K^{n}-\g_K^{i+1}}{\g_K-1}\,(y)=
\nonumber
\\
&&=\underset{i=0}{\overset{n-1}\sum}\g_K^{i}(x)\otimes 
\frac{\g_K^{n}-\g_K^{n+m}}{\g_K-1}\,(y)=
-\frac{\g_K^n-1}{\g_K-1}(x)\otimes \g_K^n\frac{\g_K^m-1}{\g_K-1}(y)=
\nonumber
\\
&&=-(\alpha_V(x)\cup_c \alpha_U(y))(\g_K^n,\g_K^m).
\nonumber
\end{eqnarray}
By continuity,  $d^1c_{x,y}=-\alpha_V(x)\cup_c \alpha_U(y),$ 
and the lemma is proved.
\end{proof}

{\it Proof of Proposition~\ref{homotopy C and K}.}  Let 
\[
h_{\g}\,:\,C^{\bullet}_{\g_K}(V)\otimes 
C^{\bullet}_{\g_K}(U)
\rightarrow  C^{\bullet} (G_K, \Vrigdag \otimes U^{\dagger}_{\text{\rm rig}})[-1]
\]
be the map defined by
\[
h_{\g}(x,y)=\begin{cases} -c_{x,y} &\text{\rm if $x\in C^1_{\g_K}(V),$ $y\in C^1_{\g_K}(U)$},\\
0 &\text{\rm elsewhere}.
\end{cases}
\]
From Lemma~\ref{lemma for c_x,y} it follows that $h_{\g}$ defines  a homotopy 
\linebreak
$h_{\g}\,:\, \alpha_{V\otimes U}\circ \cup_{\g}  \rightsquigarrow
\cup_c\circ (\alpha_V\otimes \alpha_U).$ By Proposition~\ref{homotopy for T(A)}, $h_{\g}$
induces a homotopy $h_{\Ph,\g}\,:\, \alpha_{V\otimes U}\circ \cup_{\Ph,\g}  \rightsquigarrow
\cup_K\circ (\alpha_V\otimes \alpha_U).$ The proposition is proved.

\newpage
\subsection{Transpositions} 
\subsubsection{} Let $M$ be a continuous $G_K$-module. The complex
$C^{\bullet}(G_K,M)$ is equipped with a transposition
\begin{equation}
\mathcal T_{V,c}\,:\,C^{\bullet}(G_K,M) \rightarrow C^{\bullet}(G_K,M)
\nonumber
\end{equation}
which is defined by
\begin{equation}
\mathcal T_{V,c} (x_n) (g_1,g_2,\ldots ,g_n)= (-1)^{n(n+1)/2}
g_1g_2\cdots g_n(x_n(g_1^{-1},\ldots ,g_n^{-1}))
\nonumber
\end{equation}
(see \cite{Ne06}, Section 3.4.5.1). We will often write $\mathcal T_c$ instead $\mathcal T_{V,c}.$
 The map $\mathcal T_c$ satisfies the following properties

a) $\mathcal T_c$ is an involution, i.e. $\mathcal T^2_c=\text{\rm id}.$

b) $\mathcal T_c$ is functorially homotopic to the identity map.

c)  Let $s^*_{12}\,:\,C^{\bullet}(G_K, M\otimes N)
\rightarrow C^{\bullet}(C_K,N\otimes M)$
denote the map induced by the  involution $M\otimes N \rightarrow N\otimes M$
given by $x\otimes y\mapsto y\otimes x$ (see Section 1.1.1). Set $\mathcal T_{12}=
\mathcal T_c\circ s^*_{12}.$
Then for all  $x_n\in C^{n}(G_K,M)$ and  $y_m\in   C^{m}(G_K,N)$ one has 
\begin{equation*}
\mathcal T_{12}(x_n\cup y_m)=
(-1)^{nm}(\mathcal T_c y_m)\cup (\mathcal T_c x_n), 
\end{equation*}
{\it i.e.} the diagram
\begin{equation}
\label{commutativity of transpositions with cup-prod for cont. coh.}
\xymatrix{
C^{\bullet}(G_K,M)\otimes C^{\bullet}(G_K,N) \ar[r]^>>>>>{\cup_c} \ar[d]^{s_{12}} & C^{\bullet}(G_K,M\otimes N)\ar[d]^{\mathcal T_{12}}\\
C^{\bullet}(G_K,N)\otimes C^{\bullet}(G_K,M) \ar[r]^>>>>>{\cup_c} &C^{\bullet}(G_K, N\otimes M)}
\end{equation}
commutes (\cite{Ne06}, Section 3.4.5.3). 

\subsubsection{} 
\label{subsubsection homotopy a}
There exists a homotopy
\begin{equation}
\label{definition of homotopy a}
a=(a^n)\,:\,\id \rightsquigarrow \mathcal T_c
\end{equation}
which is functorial in $M$ (\cite{Ne06}, Section 3.4.5.5).  We remark, that from the discussion 
in {\it op. cit.} it follows, that one can take $a$ such that  $a^0=a^1=0.$

\subsubsection{}Let $V$ be a $p$-adic representation of $G_K.$ We denote by $\mathcal T_{K(V)},$ or simply
by $\mathcal T_K,$ the transposition
induced on the complex $K^{\bullet}(V)$ by $\mathcal T_c$, thus
\begin{equation}
\mathcal T_{K(V)}(x_{n-1},x_n)=(\mathcal T_c(x_{n-1}),\mathcal T_c(x_n)).
\nonumber
\end{equation}
From Proposition~\ref{transpositions for T(A)} it follows that in the diagram 
\begin{equation}
\label{homotopy for K(V)}
\xymatrix{
K^{\bullet}(V)\otimes K^{\bullet}(U) \ar[rr]^{\cup_K} 
\ar[dd]_{s_{12}\circ (\mathcal T_{K(V)}\otimes \mathcal T_{K(U)})} & &K^{\bullet}(V\otimes U) \ar[dd]^{\mathcal T_{K(V\otimes U)}\circ s^*_{12}}\\
& & \ar @/_/ @{=>}[dl]_{h_{\mathcal T}}\\
K^{\bullet}(U)\otimes K^{\bullet}(V) \ar[rr]^{\cup_K} & &K^{\bullet}(U\otimes V)
}
\end{equation}
the morphisms $\mathcal T_{K(V\otimes U)}\circ s^*_{12}\circ \cup_K $ and 
$\cup_K\circ s_{12}\circ (\mathcal T_{K(V)}\otimes \mathcal T_{K(U)})$
are homotopic. 

\begin{myproposition}
\label{homotopy for Cg,Ph to K(V)}

i) The diagram 
\begin{equation}
\xymatrix{
C^{\bullet}(G_K,V) \ar[r]^{\xi_V} \ar[d]^{\mathcal T_c} &K^{\bullet}(V) \ar[d]^{\mathcal T_{K(V)}}\\
C^{\bullet}(G_K,V) \ar[r]^{\xi_V} &K^{\bullet}(V).
}
\nonumber
\end{equation}
is commutative. The map $a_{K(V)}=(a,a)$ defines a homotopy 
$a_{K(V)}\,:\,\id_{K(V)} \rightsquigarrow \mathcal T_{K(V)}$ such that 
$a_{K(V)}\circ \xi_V=\xi_V\circ a.$

ii) We have a commutative diagram
\begin{equation}
\xymatrix{
C^{\bullet}_{\Ph,\g_K}(V) \ar[r]^{\alpha_V} \ar[d]^{\text{\rm id}} &K^{\bullet}(V) \ar[d]^{\mathcal T_K}\\
C^{\bullet}_{\Ph,\g_K}(V) \ar[r]^{\alpha_V} &K^{\bullet}(V).
}
\nonumber
\end{equation}
If $a\,:\,\id \rightsquigarrow \mathcal T_c$ is a homotopy such that 
$a^0=a^1=0,$ then $a_{K(V)}\circ \alpha_V=0.$ 
\end{myproposition}
\begin{proof} i) The first assertion follows from Lemma~\ref{homotopy for T(A)}.

ii) If  $x_1\in C^1_{\g_K}(V)$ then  
$\alpha_V(x_1)\in C^{\bullet}(G_K,\Vrigdag)$ satisfies
\begin{align}
&\mathcal T_c(\alpha_V(x_1)) (g)=-g(\alpha_V(x_1) (g^{-1}))=
\nonumber \\
&=-\g_K^{\kappa (g)}\left (\frac{\g_K^{-\kappa (g)}-1}{\g_K-1}(x_1)\right )=
\frac{\g_K^{\kappa (g)}-1}{\g_K-1}(x_1)=(\alpha_V(x_1))(g).
\nonumber
\end{align}
Thus $\mathcal T_c\circ \alpha_V=\alpha_V.$ By functoriality,
$\mathcal T_K\circ \alpha_V=\alpha_V.$ Finally, the identity 
$a_{K(V)}\circ \alpha_V=0$ follows directly from the definition of $\xi_V$ 
and the assumption that $a^0=a^1=0.$
\end{proof}
\newpage
\subsection{The Bockstein map}
\label{subsection Bockstein map}
\subsubsection{}
Consider the group algebra  $A[\Gamma_K^0]$  of $\Gamma_K^0$ over  $A.$
Let $\iota \,:\,A[\Gamma_K^0]\rightarrow A[\Gamma_K^0]$ denote the $A$-linear involution given by  $\iota (\g)=\g^{-1},$ $\g \in \Gamma_K^0.$ 
We equip  $A[\Gamma_K^0]$ with the following structures: 

a) The  natural Galois action given by $g (x)= \bar g x,$ 
where $g\in G_{K},$ $x \in A[\Gamma_K^0]$ and $\bar g$ is the 
image of $g$ under canonical projection of $G_{K}\rightarrow \Gamma_K^0.$   

b) The $A[\Gamma_K^0]$-module structure $A[\Gamma_K^0]^{\iota}$ given by the involution $\iota,$ namely
$\lambda (x)=\iota (\lambda)x$ for $\lambda\in A[\Gamma_K^0],$
$x\in A[\Gamma_K^0]^{\iota}.$ 

Let $J_K$ denote the kernel of the augmentation map $A[\Gamma_K^0]\rightarrow A.$ Then the element 
\[
\widetilde X=\log^{-1}(\chi_K (\g)) (\g-1)\pmod{J_K^2}\in J_K/J_K^2
\]
does not depend on the choice of $\g \in \Gamma_K^0$ and we have
an isomorphism of $A$-modules
\begin{align*}
&\theta_K\,:\,A\rightarrow J_K^2/J_K^2,\\
&\theta_K(a)=a\widetilde X.
\end{align*}  
The action of $G_K$ on the   quotient  $\widetilde A_K^{\iota}=
A[\Gamma_K^0]^{\iota}/J_K^2$ is  given by
\begin{equation}
g (1)=1+\log (\chi_K (g))\widetilde X, \qquad g\in G_K.
\nonumber
\end{equation}
We have an exact sequence of $G_K$-modules
\begin{equation}
\label{sequence widetilde A}
0\rightarrow A\xrightarrow{\theta_K} \widetilde A_K^{\iota} \rightarrow A \rightarrow 0.
\end{equation}
Let $V$ be a $p$-adic representation of $G_K$ with coefficients in $A.$
Set $\widetilde V_K=V\otimes_A \widetilde A_K^{\iota}.$ Then the sequence (\ref{sequence widetilde A}) induces an exact sequence
of $p$-adic representations
\begin{equation}
0\rightarrow V \rightarrow \widetilde V_K \rightarrow V \rightarrow 0.
\nonumber
\end{equation}
Therefore, we have  an exact sequence of complexes
\begin{equation}
\nonumber
0\rightarrow C^{\bullet}(G_K,V)\rightarrow C^{\bullet}(G_K,\widetilde V_K)
\rightarrow C^{\bullet}(G_K,V) \rightarrow 0
\end{equation}
which gives a distinguished triangle
\begin{equation}
\label{distinguished triangle of RG with widetilde V}
\RG (K,V) \rightarrow \RG (K,\widetilde V_K) \rightarrow \RG(K,V)
\rightarrow \RG (K,V)[1].
\end{equation}
The map $s\,:\,A\rightarrow \widetilde A_K^{\iota}$ that sends $a$ to $a\pmod{J_K^2}$ induces a canonical non-equivariant section   $s_V\,:\,V\rightarrow \widetilde V_K$   of the projection $\widetilde V_K\rightarrow V.$
Define a morphism $\beta_{V,c}\,:\,C^{\bullet}(G_K,V)
\rightarrow C^{\bullet}(G_K,V)[1]$ by
\begin{equation}
\beta_{V,c}(x_n)=
\frac{1}{\widetilde X} (d\circ s_V-s_V\circ d)(x_n),\qquad x_n\in C^{\bullet}(G_K,V).
\nonumber
\end{equation}
We will write $\beta_{c}$ instead $\beta_{V,c}$ if the representation $V$ is clear
from the context. 
\begin{myproposition}
\label{formula for beta_c}
i) The distinguished triangle~(\ref{distinguished triangle of RG with widetilde V}) can be represented
by the following distinguished triangle of complexes
\begin{equation}
C^{\bullet}(G_K,V)\rightarrow C^{\bullet}(G_K,\widetilde V_K)
\rightarrow C^{\bullet}(G_K,V)\xrightarrow{\beta_{V,c}}C^{\bullet}(G_K,V)[1].
\nonumber
\end{equation}

ii) For any $x_n\in C^n(G_K,V)$ one has
\begin{equation}
\beta_{V,c} (x_n)=-\log\chi_K \cup_c x_n.
\nonumber
\end{equation}
\end{myproposition}
\begin{proof} See \cite{Ne06}, Lemma 11.2.3. 
\end{proof}

\subsubsection{}
We will prove  analogs of this proposition for the complexes $C_{\Ph,\g_K}^{\bullet}(\bD)$ and 
$K^{\bullet}(V).$ Let $\bD$ be a $(\Ph,\Gamma_K)$-module with coefficients in $A.$  
Set $\widetilde\bD=\bD\otimes_A\widetilde A_K^{\iota}.$ 
The splitting $s$ induces a splitting of the exact sequence
\begin{equation}
\label{exact sequence with widetilde D}
\xymatrix{
0\ar[r] &\bD\ar[r] &\widetilde \bD  \ar[r] &\bD\ar[r] \ar@{.>}@<-4pt>[l]_{s_{\bD}}& 0
}
\end{equation}
which we denote  by $s_{\bD}$. Define
\begin{align}
\label{definition of beta_D}
&\beta_{\bD}\,:\,C^{\bullet}_{\Ph,\g_K}(\bD)
\rightarrow C^{\bullet}_{\Ph,\g_K}(\bD)[1],
\\
&\beta_{\bD}(x)=\frac{1}{\widetilde X}(d\circ s_{\bD}-s_{\bD}\circ d)(x),\qquad x\in C^n_{\Ph,\g_K}(\bD).
\nonumber
\end{align}

\begin{myproposition}
\label{formula for beta_D}
 i) The map $\beta_{\bD}$ induces 
the connecting maps $H^n(\bD)\rightarrow H^{n+1}(\bD)$ of the long cohomology sequence 
associated to the short exact sequence (\ref{exact sequence with widetilde D}).

ii)  For any $x\in C^n_{\Ph,\g_K}(\bD)$ one has 
\begin{equation}
\beta_{\bD}(x)=-(0,\log \chi_K (\g_K))\cup_{\Ph,\g} x,
\nonumber
\end{equation}
where $(0,\log \chi_K (\g_K))\in C^1_{\Ph,\g_K}(\Qp(0)).$
\end{myproposition}

\begin{proof} The first assertion follows directly from the definition
of the connecting map.  Now, let $x=(x_{n-1},x_n)\in C^n_{\Ph,\g_K}(\bD).$
Then
\begin{align}
&(d s_{\bD}-s_{\bD} d)(x)=\nonumber \\ 
&=d (x_{n-1}\otimes 1,x_n\otimes 1)-s_{\bD}((\g_K-1)x_{n-1}+(-1)^n (\Ph-1)x_n, (\g_K-1)x_n)=
\nonumber
\\
&=(\g_K (x_{n-1})\otimes \g_K-x_{n-1}\otimes 1+(-1)^n(\Ph-1)x_n\otimes 1,\g_K(x_n)\otimes \g_K -x_n\otimes 1)-
\nonumber
\\
&-((\g_K-1)(x_{n-1})\otimes 1+(-1)^n(\Ph-1)x_n\otimes 1, (\g_K-1)(x_n)\otimes 1)=
\nonumber
\\
&=( \g_K(x_{n-1})\otimes (\g_K-1),\g_K(x_n)\otimes (\g_K-1)).
\nonumber
\end{align}
From $\g_K=1+\widetilde X\log \chi_K (\g_K)$ it follows that $\g^{-1}_K-1\equiv -\widetilde X\log \chi_K (\g_K)\pmod{J_K^2}$ and we obtain
\begin{multline}
\beta_{\bD}(x)=\frac{1}{\widetilde X}\left ((\g_K(x_{n-1}),\g_K(x_n))\otimes 
(\g_K-1))\right )=\\
=- \log \chi_K (\g_K)(\g_K(x_{n-1}),\g_K(x_n))   \in C^{n+1}_{\Ph,\g_K}(\bD).
\nonumber
\end{multline}
On the other hand, 
\[
(0,\log \chi_K (\g_K))\cup_{\Ph,\g}(x_{n-1},x_n)=
\log \chi_K (\g_K)(\g_K(x_{n-1}),\g_K(x_n))
\]
 and ii) is proved.
\end{proof} 

The exact sequence 
\begin{equation}
0\rightarrow C^{\bullet}(G_K,\Vrigdag)
\rightarrow C^{\bullet}(G_K,(\widetilde V_K)_{\text{\rm rig}}^{\dagger})
\rightarrow 
C^{\bullet}(G_K,\Vrigdag)\rightarrow 0,
\nonumber
\end{equation}
induces an exact sequence
\begin{equation}
\label{exact sequence with K(widetilde V)}
0 \rightarrow K^{\bullet}(V)
\rightarrow K^{\bullet}(\widetilde V_K)
\rightarrow K^{\bullet}(V) \rightarrow 0. 
\end{equation}
Again, the splitting $s_V\,:\,V\rightarrow \widetilde V_K$ induces a  splitting $s_K\,:\,K^{\bullet}(V)\rightarrow K^{\bullet}(\widetilde V_K)$
of (\ref{exact sequence with K(widetilde V)}) and we have a distinguished triangle of complexes
\begin{equation}
K^{\bullet}(V)\rightarrow K^{\bullet}(\widetilde V)
\rightarrow K^{\bullet}(V)\xrightarrow{\beta_{K(V)}} K^{\bullet}(V)[1].
\nonumber
\end{equation}
We will often write $\beta_K$ instead $\beta_{K(V)}.$ 

\begin{myproposition}
\label{formula for beta_K}
 i) One has
\[
\beta_K(x)=- (0,\log \chi_K)\cup_K x,\qquad x\in K^n(V).
\]

ii) The following diagrams commute 
\begin{equation}
\xymatrix{
C^{\bullet}(G_K, V) \ar[d]^{\xi_V} \ar[r]^{\beta_c} &C^{\bullet}(G_K, V)[1] \ar[d]^{\xi_V[1]}\\
K^{\bullet}(V) \ar[r]^{\beta_K} &K^{\bullet}(V)[1]  
},\quad
\xymatrix{
C_{\Ph,\g_K}^{\bullet}(V) \ar[d]^{\alpha_V} \ar[r]^{\beta_{\Ddagrig (V)}} &C_{\Ph,\g_K}^{\bullet}(V)[1]
\ar[d]^{\alpha_V[1]} \\
K^{\bullet}(V) \ar[r]^{\beta_K} &K^{\bullet}(V)[1].  
}
\nonumber
\end{equation}
\end{myproposition}
\begin{proof} i) The proof is a routine computation. Let $x=(x_{n-1},x_n)\in K^n(V),$
where $x_{n-1}\in C^{n-1}(G_K,\Vrigdag),$ $x_n\in C^n(G_K,\Vrigdag).$
Since $s_K$ commutes with $\Ph$ one has
\begin{equation}
(d s_K-s_K d)x= ((d s_V-s_V d)x_{n-1},(d s_V-s_V d)x_n).
\nonumber
\end{equation}
On the other hand,
\begin{equation}
\left ((d s_V-s_V d)x_{n-1}\right )(g_1,g_2,\ldots ,g_n)=g_1x_{n-1}(g_2,\ldots ,g_n)\otimes (\bar{g}_1-1),
\nonumber
\end{equation}
where $\bar g_1$ denote the image of $g_1\in G_K$ in $\Gamma_K.$
As in the proof of Proposition~\ref{formula for beta_D}, we can write 
$\bar g_1-1\pmod{J_K^2}=\widetilde X\log \chi_K (g_1).$
Therefore 
\begin{equation}
(d\circ s_V-s_V\circ d)x_{n-1}(g_1,g_2,\ldots ,g_n)=
\log \chi_K (g_1) g_1x_{n-1}(g_2,\ldots ,g_n)\otimes \widetilde X.
\nonumber
\end{equation}
and
\begin{multline}
(d\circ s_V-s_V\circ d)x_{n}(g_1,g_2,\ldots ,g_n,g_{n+1})=\\
=\log \chi_K (g_1) g_1x_{n-1}(g_2,\ldots ,g_n,g_{n+1})\otimes \widetilde X.
\nonumber
\end{multline}
Since  $\iota (g_1-1)= -\widetilde X\log \chi_K (g_1),$ we have 
\begin{multline} 
\beta_K(x)(g_1,\ldots ,g_n)=\frac{1}{\widetilde X} (d\circ s_K-s_K\circ d)x (g_1,g_2,\ldots ,g_n)=\\
=- \log \chi_K (g_1)  ( g_1x_{n-1}(g_2,\ldots ,g_n,g_{n}),
 g_1x_{n-1}(g_2,\ldots ,g_n,g_{n+1})).
\nonumber
\end{multline} 

On the other hand, $(0,\log \chi_K)\cup_K(x_{n-1},x_n)=(z_n,z_{n+1}),$ where
\[
z_i(g_1,g_2,\ldots ,g_i)=\log \chi_K (g_1)g_1 x_i(g_2,\ldots ,g_i),\qquad i=n,\,n+1,
\]
and i) is proved.

ii) The second statement follows from the compatibility 
of the Bockstein morphisms $\beta_c$, $\beta_{\Ddagrig (V)}$ and $\beta_K$ with 
the maps  $\alpha_V$ and $\beta_V.$ This can be also proved using i) and 
Propositions~\ref{formula for beta_c} and \ref{formula for beta_D}.
\end{proof}

\subsection{Iwasawa cohomology}
\subsubsection{}
We keep previous notation and conventions. Set 
$K_{\infty}=(K^{\cyc})^{\Delta_K},$ where $\Delta_K=\Gal (K(\zeta_p)/K).$
Then $\Gal (K_{\infty}/K)\simeq \Gamma_K^0$ and we denote by $K_n$ the 
unique subextension of $K_{\infty}$ of degree $[K_n:K]=p^n.$ 
Let $E$ be a finite extension of $\Qp$ and let $\mathcal O_E$ be its 
ring of integers. We denote by $\La_{\mathcal O_E}=\mathcal O_E[[\Gamma_K^0]]$
the Iwasawa algebra of $\Gamma_K^0$ with coefficients in $\mathcal O_E.$ 
The choice of a generator $\g_K$ of $\Gamma_K^0$ fixes an isomorphism
$\La_{\mathcal O_E}\simeq \mathcal O_E[[X]]$ such that $\g_K\mapsto X+1.$ 
Let $\CH_E$ denote the algebra of formal power series $f(X)\in E[[X]]$
that converge on the open unit disk $A(0,1)=\{x\in \mathbf C_p \mid 
\vert x\vert_p<1\}$ and let 
\[
\CH_E(\Gamma_K^0)=\{f(\g_K-1) \mid f(X)\in \CH_E\}.
\] 
We consider $\La_{\mathcal O_E}$ as a subring of
$\mathcal H_E(\Gamma_K^0).$ The involution $\iota \,:\,\La_{\mathcal O_E}\rightarrow \La_{\mathcal O_E}$ extends to $\CH_E(\Gamma_K^0).$  Let $\La_{\mathcal O_E}^{\iota}$ (resp.
$\CH_E(\Gamma_K^0)^{\iota}$) denote $\La_{\mathcal O_E}$
(resp. $\CH_E(\Gamma_K^0)$) equipped with the $\La_{\mathcal O_E}$-module (resp. $\CH_E(\Gamma_K^0)$-module) structure given by 
$\alpha \star \lambda =\iota (\alpha)\lambda.$   
 
Let $V$ be a $p$-adic representation of $G_K$ with coefficients in $E.$
Fix a $\mathcal O_E$-lattice $T$ of $V$ stable under the action of $G_K$ 
and set $\Ind_{K_{\infty}/K} (T)=T\otimes_{\mathcal O_E} \La_{\mathcal O_E}^{\iota}.$   
We equip  $\Ind_{K_{\infty}/K} (T)$ with the following structures:

a) The diagonal action of $G_K,$ namely $g(x\otimes \lambda)=g(x)\otimes 
\bar g \lambda,$ for all $g\in G_K$ and $x\otimes \lambda \in 
\Ind_{K_{\infty}/K} (T);$ 

b) The structure of $\La_{\mathcal O_E}$-module given by 
$\alpha (x\otimes \lambda)=x\otimes \lambda \iota (\alpha)$ for all
$\alpha \in \La_{\mathcal O_E}$ and $x\otimes \lambda\in 
\Ind_{K_{\infty}/K} (T).$

Let $\RG_{\Iw} (K,T)$ denote the class of the complex 
$C^{\bullet}(G_K,\Ind_{K_{\infty}/K} (T))$ in the derived category 
$\mathcal D (\La_{\mathcal O_E})$ of $\La_{\mathcal O_E}$-modules.
The augmentation map $\La_{\mathcal O_E} \rightarrow \mathcal O_E$ 
induces an isomorphism
\[
\RG_{\Iw}(K,T)\otimes_{\La_{\mathcal O_E}}^{\mathbf L}\mathcal O_E
\simeq \RG (K,T).
\]
We write $H^i_{\Iw}(K,T)=\R^i\Gamma_{\Iw}(K,T)$ for the cohomology of 
$\RG_{\Iw}(K,T).$ From Shapiro's lemma it follows that 
\[
H^i_{\Iw}(K,T)=\underset{\mathrm{cores}}\varprojlim H^i(K_n,T)
\]
(see, for example, \cite{Ne06}, Sections 8.1-8.3).

We review the Iwasawa cohomology of $(\Ph,\Gamma_K)$-modules
(see \cite{CC2} and \cite{KPX}).
The map $\Ph\,:\,\mathbf B_{\mathrm{rig},K}^{\dagger, r } \rightarrow 
\mathbf B_{\mathrm{rig},K}^{\dagger, pr }$ equips $\mathbf B_{\mathrm{rig},K}^{\dagger, pr }$ with the structure of a free $\Ph\,:\,\mathbf B_{\mathrm{rig},K}^{\dagger, r }$-module of rank $p.$ Define 
\begin{align*}
&\psi \,:\, \mathbf B_{\mathrm{rig},K}^{\dagger, pr } \rightarrow 
\mathbf B_{\mathrm{rig},K}^{\dagger, r },
& \psi (x)=\frac{1}{p}\Ph^{-1}\circ \mathrm{Tr}_{\mathbf B_{\mathrm{rig},K}^{\dagger, pr }/\Ph (\mathbf B_{\mathrm{rig},K}^{\dagger, r })}(x).
\end{align*}
Since $\CR_{K,\Qp}= \underset{r\geqslant r_K}\cup
\mathbf B_{\mathrm{rig},K}^{\dagger, r},$ 
 the operator $\psi$ extends by linearity to an operator 
 $\psi\,:\,\CR_{K,E}\rightarrow \CR_{K,E}$  
such  that  $\psi \circ \Ph=\id.$

Let  $\bD$ is a $(\Ph,\Gamma_K)$-module over $\CR_{K,E}=\CR_K\otimes_{\Qp} E.$
If $e_1, e_2,\ldots ,e_d$ is a base of $\bD$ over $\CR_{K,E},$ then  
$\Ph (e_1),\Ph (e_2),\ldots ,\Ph (e_d)$ is again a base of $\bD$, and 
we define
\begin{align*}
&\psi \,:\,\bD \rightarrow \bD,\\
&\psi \left (\underset{i=1}{\overset{d}\sum} a_i\Ph (e_i)\right )=
\underset{i=1}{\overset{d}\sum} \psi(a_i) e_i.
\end{align*}
  The action of $\Gamma_K^0$ on $\bD^{\Delta_K}$ extends to
a natural action of $\CH_E(\Gamma_K^0)$ and we   consider the complex of $\CH_E(\Gamma_K^0)$-modules 
\[
C^{\bullet}_{\Iw}(\bD)\,\,:\,\, \bD^{\Delta_K}\xrightarrow{\psi-1}
\bD^{\Delta_K},
\]
where the terms are concentrated in degrees $1$ and $2.$ 
Let $\RG_{\Iw}(\bD)= \left [C^{\bullet}_{\Iw}(\bD)\right ]$
denote  the class of $C^{\bullet}_{\Iw}(\bD)$ in the derived category 
$\mathcal D(\CH_E(\Gamma_K^0)).$  We also consider  the complex
$C^{\bullet}_{\Ph,\g_K}(\Ind_{K_\infty/K}(\bD)),$ where $\Ind_{K_\infty/K}(\bD)=
\break\bD\otimes_E\CH_E(\Gamma_K^0)^{\iota},$ and set 
$\RG (K,\Ind_{K_\infty/K}(\bD))=
\left [C^{\bullet}_{\Ph,\g_K}(\overline \bD) \right ].$

\begin{mytheorem}
\label{theorem Iwasawa cohomology}
  Let $\bD$ be a $(\Ph,\Gamma_K)$-module over 
$\CR_{K,E}.$ Then 

i) The complexes $C^{\bullet}_{\Iw}(\bD)$ and 
$C_{\Ph,\g_K}(\overline \bD)$ are quasi-isomorphic and therefore 
\[
\RG_{\Iw}(\bD)\simeq \RG (K,\Ind_{K_\infty/K}(\bD)).
\]

ii) The cohomology groups $H^i_{\Iw}(\bD)=\R^i\Gamma_{\Iw} (\bD)$ are 
finitely-generated $\CH_E(\Gamma_K^0)$-modules.  Moreover, 
$\mathrm{rk}_{\CH_E(\Gamma_K^0)}H^1_{\Iw}(\bD)=[K:\Qp]\,\mathrm{rk}_{\CR_{K,E}}\bD$
and 
\linebreak
$H^1_{\Iw}(\bD)_{\mathrm{tor}}$ and
$H^2_{\Iw}(\bD)$ are finite-dimensional $E$-vector spaces. 

iii) We have an isomorphism
\[
C^{\bullet}_{\Ph,\g_K}(\Ind_{K_\infty/K}(\bD))\otimes_{\CH_E(\Gamma_K^0)}E \iso C^{\bullet}_{\Ph,\g_K}(\bD)
\]
which induces the Hochschild--Serre exact sequences 
\[
0\rightarrow H^i_{\Iw}(\bD)_{\Gamma_K^0}\rightarrow H^i(\bD)\rightarrow H_{\Iw}^{i+1}(\bD)^{\Gamma_K^0}\rightarrow 0.
\]

iv) Let $\omega =\mathrm{cone} \left [\mathcal{K}_E(\Gamma_K^0)\rightarrow \mathcal{K}_E(\Gamma_K^0)/\CH_E(\Gamma_K^0) \right ] [-1],$
where $\mathcal{K}_E(\Gamma_K^0)$ is the field of fractions of $\CH_E(\Gamma_K^0).$
Then the functor $\mathcal D=\Hom_{\CH_E(\Gamma_K^0)}(-,\omega)$  furnishes  a duality 
\[
\mathcal D\RG_{\Iw}(\bD)\simeq \RG_{\Iw}(\bD^*(\chi_K))^{\iota}[2].
\]

v) If $V$ is a $p$-adic representation of $G_K,$ then there are 
canonical and functorial isomorphisms
\begin{multline*}
\RG_{\Iw}(K,T)\otimes_{\La_{\mathcal O_E}}^{\mathbf L}\CH_E(\Gamma_K^0) 
\simeq \RG \left (K, T\otimes_{\mathcal{O}_E}\CH_E(\Gamma_K^0)^{\iota}\right )
\simeq \\
\simeq\RG (K,\Ind_{K_\infty/K}(\Ddagrig (V))).
\end{multline*}
\end{mytheorem}
\begin{proof}
See \cite{PoCIT}, Theorem 2.6.
\end{proof}

\subsubsection{}  We will need the following lemma.

\begin{mylemma} 
\label{lemma about H^2_Iw}
Let $E$ be a finite extension of $\Qp$ and 
let $\bD$ be a potentially semistable $(\Ph,\Gamma_K)$-module over $\CR_{K,E}.$ Then

i) $H^1_{\Iw}(\bD)_{\mathrm{tor}} \simeq \left (\bD^{\Delta_K}\right )^{\Ph=1}.$

ii) Assume that 
\[
\CDpst (\bD^*(\chi_K))^{\Ph=p^i}=0, \qquad \forall i\in \Z.
\]
Then $H^2_{\Iw}(\bD)=0.$
\end{mylemma}
\begin{proof} i) Consider the exact sequence 
\[
0\rightarrow \bD^{\Ph=1} \rightarrow \bD^{\psi=1}\xrightarrow{\Ph-1} \bD^{\psi=0}.
\]
Since $\left (\bD^{\Delta_K} \right )^{\psi=1}\simeq H^1_{\Iw}(\bD)$ and $\bD^{\psi=0}$ is $\CH_E(\Gamma_K^0)$-torsion 
free (\cite{KPX}, Theorem 3.1.1), $H^1_{\Iw}(\bD)_{\mathrm{tor}} \subset \left (\bD^{\Delta_K} \right )^{\Ph=1}.$ On the other hand, 
$\bD^{\Ph=1}$ is a finitely dimensional $E$-vector space (see, for example \cite{KPX}, 
Lemma 4.3.5) and therefore is $\CH_E(\Gamma_K^0)$-torsion. This proves the first statement.

ii)  By Theorem~\ref{theorem Iwasawa cohomology} iv), $H^2_{\Iw}(\bD)$ and 
$H^1_{\Iw}(\bD^*(\chi_K))_{\mathrm{tor}}$ are dual to each other  and it is enough to
show that $\bD^*(\chi_K)^{\Ph=1}=0.$
Since $\dim_E\bD^*(\chi_K)^{\Ph=1} < \infty,$ there exists $r$ such that 
$\bD^*(\chi_K)^{\Ph=1} \subset \bD^*(\chi_K)^{(r)},$ and for $n\gg 0$ the map
\linebreak 
$i_n=\Ph^{-n}\,:\,
\CR^{(r)}_{K,E}\rightarrow E\otimes_{\Qp}K^{\cyc}[[t]]$
gives an injection
\begin{multline*}
\bD^*(\chi_K)^{\Ph=1} \rightarrow \bD^*(\chi_K)^{(r)}\otimes_{i_n} \left (E\otimes_{\Qp}K^{\cyc}[[t]]\right ) \iso \\
\iso\F^0 \left (\CDdr(\bD^*(\chi_K))\otimes_{K}\otimes K^{\cyc}((t))\right ).
\end{multline*} 
Looking at the action of $\Gamma_K$ on $\F^0 \left (\CDdr(\bD^*(\chi_K))\otimes_{K} K^{\cyc}((t))\right )$ and using the fact that $\bD^*(\chi_K)^{\Ph=1}$ is finite-dimensional over $E,$
it is easy to prove, that there exists a finite extension $L/K$ such that
$\bD^*(\chi_K)^{\Ph=1},$ viewed as $G_L$-module,  is isomorphic to a finite direct sum of modules $\Qp(i),$ $i\in \Z.$  Therefore 
\[
\bD^*(\chi_K)^{\Ph=1}\simeq \left (\bD^*(\chi_K)^{\Ph=1} \otimes_{\Qp}\Qp (-i)
\right )^{\Gamma_L}\otimes_{\Qp}\Qp(i)
\]
as $G_L$-modules.  Since 
\begin{multline*}
\left (\bD^*(\chi_K)^{\Ph=1}\otimes_{\Qp}\Q_p(-i)\right )^{\Gamma_L}\subset 
\left (\bD^*(\chi_K)\otimes_{\CR_{K,E}} \CR_{L,E}[1/t,\ell_{\pi}] \right )^{\Ph=p^{-i},\Gamma_L}=\\
= \CDst^L(\bD^*(\chi_K))^{\Ph=p^{-i}}=0,
\end{multline*} 
we obtain that $\bD^*(\chi_K)^{\Ph=1}=0,$ and the lemma is proved.
\end{proof}

\subsection{The group $H^1_f(\bD)$}
\label{subsection group H^1_f}
\subsubsection{}
Let $\bD$ be a potentially semistable $(\Ph,\Gamma_K)$-module over $\CR_{K,E},$ where 
$E$ is a finite extension of $\Qp.$ As usual, we have the isomorphism
\[
H^1(\bD) \simeq \Ext^1_{\CR_{K,E}}(\CR_{K,E},\bD)
\]
which associates to each cocycle $x =(a,b)\in C^1_{\Ph,\g_K}(\bD)$  the extension
\[
0\rightarrow \bD \rightarrow \bD_{x}\rightarrow \CR_{K,E} \rightarrow 0
\]
such that $\bD_{x}=\bD\oplus \CR_{K,E}e$ with $\Ph (e)=e+a$ and $\g_K(e)=e+b.$
We say that $[x]=\mathrm{class} (x)\in H^1(\bD)$ is crystalline if 
\[
\rk_{E\otimes K_0} (\CDcris (\bD_{x}))=\rk_{E\otimes K_0} (\CDcris (\bD))+1
\]
and define 
\[
H^1_f(\bD)=\{ [x]\in H^1(\bD) \mid \text{\rm $\cl (x)$ is crystalline}\}.
\]

\begin{myproposition}
\label{proposition properties H^1_f}
Let $\bD$ be a potentially semistable $(\Ph,\Gamma_K)$-module over $\CR_{K,E}.$ Then 

i) $H^0(\bD)= \F^0 (\CDpst (\bD))^{\Ph=1,N=0,G_K}$ and    $H^1_f(\bD)$ is a $E$-subspace of $H^1(\bD)$ of dimension
\[
\dim_E H^1_f(\bD)=\dim_E \CDdr (\bD)- \dim_E \F^0\CDdr (\bD)+\dim_E H^0(\bD).
\]

ii) There exists an exact sequence 

\begin{equation*}
0\rightarrow H^0(\bD) \rightarrow \CDcris (\bD) \xrightarrow{(\pr,1-\Ph)} 
t_{\bD}(K)\oplus \CDcris (\bD)\rightarrow H^1_f(\bD)
\rightarrow 0,
\end{equation*}
where $t_{\bD}(K)=\CDdr (\bD)/\F^0\CDdr (\bD).$ 

iii) $H^1_f(\bD^*(\chi_K))$ is the orthogonal complement to $H^1_f(\bD)$ under the duality
$H^1(\bD)\times H^1(\bD^*(\chi_K))\rightarrow E.$

iv) Let
\[
0\rightarrow \bD_1\rightarrow \bD \rightarrow \bD_2 \rightarrow 0
\]
be an exact sequence of potentially semistable $(\Ph,\Gamma_K)$-modules. 
Assume that one of the following conditions holds

a) $\bD$ is crystalline;

b) $\mathrm{Im} ( (H^0(\bD_2) \rightarrow H^1(\bD_1)) \subset  H^1_f(\bD_1).$

Then one has an exact sequence 
\begin{equation*} 
0 \rightarrow H^0(\bD_1) \rightarrow H^0(\bD) \rightarrow  H^0(\bD_2)
\rightarrow H^1_f(\bD_1) \rightarrow  H^1_f(\bD) \rightarrow H^1_f(\bD_2)
\rightarrow 0.
\end{equation*}
\end{myproposition} 
\begin{proof} This proposition is proved in  \cite{Ben11}, Proposition 1.4.4,
and Corollaries 1.4.6 and 1.4.10.       
For another approach to $H^1_f(\bD)$ and an alternative proof see \cite{Na}, Section 2.
\end{proof}

\subsubsection{}In this subsection we assume that $K=\Qp.$ We review the computation of the cohomology 
of some isoclinic $(\Ph,\Gamma_{\Qp})$-modules given in \cite{Ben11}.
To simplify notation, we write $\chi_p$ and $\Gamma_p^0$ instead $\chi_{\Qp}$
and $\Gamma_{\Qp}^0$ respectively. 

\begin{myproposition}
\label{proposition isoclinic modules}
Let $\bD$ be a semistable $(\Ph,\Gamma_{\Qp})$-module of rank $d$
over  $\CR_{\Qp,E}$ such that $\CDst (\bD)^{\Ph=1}=\CDst (\bD)$ and 
$\F^0\CDst (\bD)=\CDst (\bD).$ Then 

i) $\bD$ is crystalline and $H^0(\bD)=\CDcris (\bD)$.

ii) One has  $\dim_EH^0(\bD)=d,$ $\dim_EH^1(\bD)=2d$ and $H^2(\bD)=0.$

iii) The map 
\begin{align*}
&i_{\bD}\,:\, \CDcris (\bD)\oplus \CDcris (\bD) \rightarrow H^1(\bD),\\
&i_{\bD}=\cl (-x,\log \chi_{p} (\g_{\Qp})y)
\end{align*}
is an isomorphism of $E$-vector spaces. Let $i_{\bD,f}$ and $i_{\bD,c}$ denote the restrictions of $i_{\bD}$
on the first and the second summand respectively. Then $\mathrm{Im} (i_{\bD,f})=H^1_f(\bD)$ 
and we have a  decomposition
\[
H^1(\bD)=H^1_f(\bD)\oplus H^1_c(\bD),
\]
where $H^1_c(\bD)=\mathrm{Im} (i_{\bD,c}).$

iv) Let $\bD^*(\chi_{p})$ be  the Tate dual of $\bD$. Then 
\[\CDcris (\bD^*(\chi_{p}))^{\Ph=p^{-1}}=
\CDcris (\bD^*(\chi_{p}))\]
 and $\F^0\CDcris (\bD^*(\chi_{p}))=0.$ In particular,
$H^0(\bD^*(\chi_{p}))=0.$ Let 
\[
[\,\,,\,\,]_{\bD}\,:\,\CDcris (\bD^*(\chi_{p}))\times \CDcris (\bD)\rightarrow E
\]
denote the canonical duality. Define a morphism
\[
i_{\bD^* (\chi_{p})}\,:\, \CDcris (\bD^*(\chi_{p})) \oplus \CDcris (\bD^* (\chi_{p}))\rightarrow H^1(\bD^* (\chi_{p}))
\]
by
\[
i_{\bD^*(\chi_p)}(\alpha,\beta)\cup i_{\bD}(x,y)=[\beta,x]_{\bD}-[\alpha,y]_{\bD}
\]
and denote by $\mathrm{Im} (i_{\bD^*(\chi_p),f})$ and $\mathrm{Im} (i_{\bD^*(\chi_p),c})$
the restrictions of $i_{\bD^*(\chi_p)}$ on the first and the second summand respectively. Then
$\mathrm{Im} (i_{\bD^*(\chi_p),f})= \break{H^1_f(\bD^*(\chi_p))}$ and again we have
\[
H^1(\bD^*(\chi_p))=H^1_f(\bD^*(\chi_p))\oplus H^1_c(\bD^*(\chi_p)),
\]
where $H^1_c(\bD^*(\chi_p))=\mathrm{Im} (i_{\bD^*(\chi_p),c}).$
\end{myproposition}
\begin{proof} See \cite{Ben11}, Proposition 1.5.9 and Section 1.5.10.
\end{proof}

We also need the following result.

\begin{myproposition}
\label{coinvariants of H^1_Iw =H^1_c}
Let $\bD$ be a crystalline  $(\Ph,\Gamma_{\Qp})$-module over $\CR_{\Qp,E}$ such that
$\CDcris (\bD)^{\Ph=p^{-1}}=\CDcris (\bD)$ and $\F^0\CDcris (\bD)=0.$ Then
\[
H^1_{\Iw}(\bD)_{\Gamma_{p}^0}=H^1_c(\bD).
\]
\end{myproposition}
\begin{proof} See \cite{Ben14}, Proposition 4.
\end{proof}

\section{$p$-adic height pairings I: Selmer complexes}
\label{section Selmer}

\subsection{Selmer complexes} 
\label{subsection Selmer complexes}
\subsubsection{}
In this section we construct  $p$-adic height pairings
using  Nekov\'a\v r's formalism of Selmer complexes. 
%We follow \cite{Ne06} replacing Greenberg's local conditions
%by local conditions given by $(\Ph,\Gamma)$-submodules.
Let $F$ be a number field. We denote by $S_f$ (resp. $S_{\infty}$) the set of all non-archimedean (resp. archimedean) absolute values on $F.$ Fix a prime number $p$ and
a finite subset $S\subset S_f$ containing the set $S_p$ of all 
$\fq\in S_f$ such that $\fq\mid p.$ 
We will write $\Sigma_p$ for the complement of $S_p$ in $S.$  
Let $G_{F,S}$ denote   the Galois
group of the maximal algebraic extension of $F$  
unramified outside $S\cup S_{\infty}.$   For each $\fq\in S,$ we fix
a decomposition group at $\fq$ which we identify with $G_{F_\fq},$ 
and set $\Gamma_\fq=\Gamma_{F_\fq}.$  Fix a generator $\g_\fq\in \Gamma_\fq^0.$

\subsubsection{} Let $V$ be a $p$-adic representation of $G_{F,S}$ with coefficients in a
$\Qp$-affinoid algebra $A.$ We will write $V_\fq$ for the restriction 
of $V$ on the decomposition group at $\fq.$
For each $\fq\in S_p,$ we fix 
a $(\Ph,\Gamma_\fq)$-submodule $\bD_\fq$ of $\Ddagrig (V_\fq)$
 that is a $\CR_{F_\fq,A}$-module   direct summand  of $\DdagrigA (V_\fq).$  Set $\bD =(\bD_\fq)_{\fq\in S_p}$ and  define 
\begin{equation}
\nonumber 
U_\fq^{\bullet}(V,\bD)= \begin{cases} 
C^{\bullet}_{\Ph,\g_\fq}(\bD_\fq), & \textrm{if $\fq\in S_p$,}\\
C^{\bullet}_{\ur}(V_\fq), &\textrm{if $\fq\in \Sigma_p,$}
\end{cases}
\end{equation}
where  
\begin{equation}
C_{\ur}^{\bullet}(V_\fq) \,\,:\,\, V_\fq^{I_\fq}\xrightarrow{\Fr_\fq-1}  V_\fq^{I_\fq}, 
\qquad  \textrm{ $\fq\in \Sigma_p$,}
\nonumber 
\end{equation}
and the terms are concentrated in degrees $0$ and $1$.
Note  that, if $\fq\in S_p,$ we have   $[U_\fq^{\bullet}(V,\bD)]=\RG (F_\fq,\bD_\fq)\in
\mathcal D^{[0,2]}_{\perf}(A)$ by Theorem~\ref{Theorem KPX}, i). On the other hand,
if $\fq\in \Sigma_p$, then in general the complex $U_\fq^{\bullet}(V,\bD)$ is not quasi-isomorphic to a perfect complex of $A$-modules.

First assume that $\fq\in  \Sigma_p.$ Then we have a canonical morphism
\begin{equation} 
\label{morphism g_v}
g_\fq\,:\,U_\fq^{\bullet}(V,\bD) \rightarrow C^{\bullet}(G_{F_\fq},V)
\end{equation} 
defined by 
\begin{align}
&g_\fq (x_0)=x_0, &\textrm{if}\quad x_0\in  U_\fq^{0}(V,\bD), \nonumber\\ 
&g_\fq (x_1) (\Fr_\fq)=x_1, &\textrm{if}\quad x_1\in  U_\fq^{1}(V,\bD) \nonumber 
\end{align}
and the restriction map 
\begin{equation}
\label{morphism f_v v in sigma}
f_\fq=\res_\fq\,:\, C^{\bullet}(G_{F,S},V)\rightarrow C^{\bullet}(G_{F_\fq},V).
\end{equation}
Now assume that $\fq\in S_p.$ The inclusion $\bD_\fq\subset \Ddagrig (V_\fq)$ induces a 
morphism $U_\fq^{\bullet}(V,\bD)=C^{\bullet}_{\Ph,\gamma}(\bD_\fq)\rightarrow C^{\bullet}_{\Ph,\g}(V_\fq).$ We denote by 
\begin{equation}
g_\fq\,:\,U_\fq^{\bullet}(V,\bD) \rightarrow K^{\bullet}(V_\fq),\qquad \fq\mid p
\end{equation} 
the composition of  this morphism with the quasi-isomorphism
\linebreak
 $\alpha_{V_\fq}\,:\,C^{\bullet}_{\Ph,\g}(V_\fq)\simeq K^{\bullet}(V_\fq)$ 
 constructed in Section~\ref{subsection K(V)}
and by 
\begin{equation}
\label{morphism f_v v mid p}
f_\fq\,:\, C^{\bullet}(G_{F,S},V)\rightarrow K^{\bullet}(V_\fq), \qquad \fq\mid p
\end{equation}
the composition of the restriction map   $\res_\fq\,:\,C^{\bullet}(G_{F,S},V)\rightarrow C^{\bullet}(G_{F_\fq},V)$
with the quasi-isomorphism $\xi_{V_\fq}\,:\,C^{\bullet}(G_{F_\fq},V)\rightarrow K^{\bullet}(V_\fq)$
(see Proposition~\ref{quasi-iso with K(V)}). 

 Set 
\[
K^{\bullet} (V)=\left (\underset{\fq\in \Sigma_p}\bigoplus C^{\bullet}(G_{F_\fq},V)\right )
\oplus \left (\underset{\fq\in S_p}\bigoplus K^{\bullet}(V_\fq)\right )
\]
and
$U^{\bullet}(V,\bD)=\underset{\fq\in S}\oplus U_\fq^{\bullet}(V,\bD).$

Recall that 
\begin{equation}
\nonumber 
C^{\bullet}(G_{F,S},V)\in \mathcal K^{[0,3]}_{\ft}(A)
\end{equation}
and  the associated object of the derived category 
\begin{equation}
\nonumber
\RG_S (V):= \left [C^{\bullet}(G_{F,S},V) \right ]\in \mathcal D^{[0,3]}_{\perf}(A)
\end{equation}
(see \cite{Ne06} and \cite{Po13}).
Therefore,  we have  a diagram in $\mathcal K^{[0,3]}_{\ft}(A)$
\begin{equation}
\label{diagram Selmer complex}
\nonumber 
\xymatrix{
C^{\bullet}(G_{F,S},V) \ar[r]^f &K^{\bullet}(V) \\
&U^{\bullet}(V,\bD), \ar[u]^g
}
\end{equation}
where $f=(f_\fq)_{\fq\in S}$ and $g=(g_\fq)_{\fq\in S},$
and the corresponding diagram in 
\linebreak
$\mathcal D^{[0,3]}_{\ft}(A)$
\begin{equation}
\nonumber
\xymatrix{
\RG_S(V) \ar[r]&\underset{\fq\in S}\bigoplus \RG (F_\fq,V)\\
&\underset{\fq\in S}\bigoplus \RG (F_\fq, V,\bD), \ar[u]
}
\end{equation}
where $\RG (F_\fq,V,\bD)= \left [U^{\bullet}_\fq(V,\bD) \right ]$.

The associated Selmer complex is defined as
\begin{equation*}
S^{\bullet}(V,\bD)=\mathrm{cone} \left [ C^{\bullet}(G_{F,S},V)\oplus U^{\bullet}(V,\bD)
\xrightarrow{f-g} K^{\bullet}(V) \right ] [-1].
\end{equation*} 
We set $\RG (V,\bD):=\left [S^{\bullet}(V,\bD) \right ]$ and write 
$H^{\bullet}(V,\bD)$ for the cohomology of 
$S^{\bullet}(V,\bD).$ From the definition, it follows directly
that $S^{\bullet}(V,\bD)\in \mathcal K^{[0,3]}_{\ft}(A)$ and 
if, in addition, $\left [U^{\bullet}_\fq(V,\bD) \right ]\in 
\mathcal D^{[0,1]}_{\perf}(A)$ for all $\fq\in \Sigma_p,$ then 
 $\RG (V,\bD) \in \mathcal D^{[0,3]}_{\perf}(A).$ 

\subsubsection{}
The previous construction can be slightly generalized. 
Fix a finite subset $\Sigma \subset \Sigma_p$ and, for each $\fq\in \Sigma ,$
a locally direct summand  $M_\fq$ of the $A$-module $V_\fq$ stable under the action of $G_{F_\fq}.$  Let $M=(M_\fq)_{\fq\in \Sigma}.$ Define 
\begin{equation}
\nonumber 
U_\fq^{\bullet}(V,\bD,M)= \begin{cases} 
C^{\bullet}_{\Ph,\g_\fq}(\bD_\fq), & \textrm{if $\fq\in S_p$,}\\
C^{\bullet}_{\ur}(V_\fq), &\textrm{if $\fq\in \Sigma_p\setminus\Sigma,$}\\
C^{\bullet}(G_{F_\fq}, M_\fq),  &\textrm{if $\fq\in \Sigma.$}
\end{cases}
\end{equation}
We denote by $S^{\bullet}(V,\bD,M)$ the associated Selmer complex 
and set 
\linebreak
 $\RG (V,\bD,M):=\left [S^{\bullet}(V,\bD,M) \right ].$
If $M_\fq=0$ for all $\fq\in \Sigma,$ we write $S^{\bullet}(V,\bD,\Sigma)$ and  
$\RG  (V,\bD,\Sigma)$ for $S^{\bullet}(V,\bD,M)$ and 
$\RG (V,\bD,M)$ respectively. Note that  $\RG (V,\bD,\Sigma_p) \in \mathcal D^{[0,3]}_{\perf}(A).$ 

\subsubsection{}
We construct  cup products for our  Selmer complexes  $\RG (V,\bD,M).$ Consider the 
dual representation $V^*(1)$ of $V.$ We equip $V^*(1)$ with 
the dual local conditions setting
\begin{align*}
&\bD_\fq^{\perp}=\mathrm{Hom}_{\CR_A}(\Ddagrig (V_\fq^*(1))/\bD_\fq,\CR_A),\qquad \forall \fq\in S_p,\\
&M_\fq^{\perp}=\mathrm{Hom}_{A}(V_\fq/M_\fq,A),\qquad \forall \fq\in \Sigma,
\end{align*}
and denote by $f_\fq^{\perp}$ and $g_\fq^{\perp}$ the morphisms (\ref{morphism g_v}-\ref{morphism f_v v mid p}) associated to
\linebreak
$(V^*(1),\bD^{\perp},M^{\perp}).$ 
Consider the following data 

1) The complexes $A_1^{\bullet}=C^{\bullet}(G_{F,S},V),$ $B_1^{\bullet}=U^{\bullet}(V,\bD,M),$
and $C_1^{\bullet}=K^{\bullet}(V) $ equipped with the morphisms
$f_1=(f_\fq)_{\fq\in S}\,:\,A_1^{\bullet}\rightarrow C_1^{\bullet}$ and 
$g_1=\underset{\fq\in S}\oplus g_\fq \,:\,B_1^{\bullet}\rightarrow C_1^{\bullet};$

2) The complexes $A_2^{\bullet}=C^{\bullet}(G_{F,S},V^*(1)),$ 
$B_2^{\bullet}=U^{\bullet}(V^*(1),\bD^{\perp},M^{\perp}),$ and
\linebreak
$C_2^{\bullet}=K^{\bullet}(V^*(1))$
equipped with the morphisms $f_2=(f^{\perp}_\fq)_{\fq\in S}\,:\,A_2^{\bullet}\rightarrow C_2^{\bullet}$ and $g_2=\underset{\fq\in S} \oplus g^{\perp}_\fq\,:\,B_2^{\bullet}\rightarrow C_2^{\bullet};$

3) The complexes $A_3^{\bullet}=\tau_{\geqslant 2}C^{\bullet}(G_{F,S}, A(1)),$
$B_3^{\bullet}=0$ and 
\linebreak
$C_3^{\bullet}=\tau_{\geqslant 2}K^{\bullet}(A(1))$
equipped with the map $f_3\,:\,A_3^{\bullet}\rightarrow C_3^{\bullet}$ given by
\[
\tau_{\geqslant 2}C^{\bullet}(G_{F,S}, A(1)) \xrightarrow{(\res_\fq)_\fq} 
\underset{\fq}\bigoplus \tau_{\geqslant 2} C^{\bullet}(G_{F_\fq},A (1))
\rightarrow \tau_{\geqslant 2}K^{\bullet}(A(1))
\]
and the zero map $g_3\,:\,B_3^{\bullet}\rightarrow C_3^{\bullet}.$

4) The cup product $\cup_A\,:\,A_1^{\bullet}\otimes A_2^{\bullet}\rightarrow A_3^{\bullet}$
defined as the composition
\begin{multline}
\nonumber 
\cup_A\,:\,C^{\bullet}(G_{F,S},V)\otimes  C^{\bullet}(G_{F,S},V^*(1))
\xrightarrow{\cup_c} C^{\bullet}(G_{F,S},V\otimes V^*(1))
\rightarrow \\
C^{\bullet}(G_{F,S},A^*(1)) \rightarrow \tau_{\geqslant 2}C^{\bullet}(G_{F,S},A^*(1)), 
\end{multline}

5) The zero  cup product $\cup_B\,:\,B_1^{\bullet}\otimes B_2^{\bullet}\rightarrow B_3^{\bullet}.$

6) The cup product $\cup_C\,:\,C_1^{\bullet}\otimes C_2^{\bullet}\rightarrow C_3^{\bullet}$
defined as the composition
\[
K^{\bullet}(V) \otimes K^{\bullet}(V^*(1)) \xrightarrow{\cup_K} 
K^{\bullet}(V\otimes V^*(1)) \rightarrow K^{\bullet}(A(1))
\rightarrow \tau_{\geqslant 2}K^{\bullet}(A(1)).
\]

7) The zero maps $h_f\,:\,A^{\bullet}_1\otimes A^{\bullet}_2 \rightarrow C_3^{\bullet}[-1]$
and $h_g\,:\,B^{\bullet}_1\otimes B^{\bullet}_2 \rightarrow C_3^{\bullet}[-1].$

%\begin{equation}
%\label{data for cup product}
%\begin{aligned}
%&A_1^{\bullet}=C^{\bullet}(G_{F,S},V), &&A_2^{\bullet}=C^{\bullet}(G_{F,S},V^*(1)),  \\
%&B_1^{\bullet}=U^{\bullet}(V,\bD), &&B_2^{\bullet}=U^{\bullet}(V^*(1),\bD^{\perp}),  \\
%&C_1^{\bullet}=K^{\bullet}(V), &&C_2^{\bullet}=K^{\bullet}(V^*(1)),
% \\
%&f_1=(f_v)_{v\in S}\,:\,A_1^{\bullet}\rightarrow C_1^{\bullet},
%&&f_2=(f^{\perp}_v)_{v\in S}\,:\,A_2^{\bullet}\rightarrow C_2^{\bullet},
% \\
%&g_1=\underset{v\in S}\oplus g_v \,:\,B_1^{\bullet}\rightarrow C_1^{\bullet},
%&&g_2=\underset{v\in S} \oplus g^{\perp}_v\,:\,B_2^{\bullet}\rightarrow C_2^{\bullet},\\
%&A_3^{\bullet}=\tau_{\geqslant 2}C^{\bullet}(G_{F,S},\Qp(1)), && 
%\\
%&B_3^{\bullet}=0,  &&     
%\\
%&C_3^{\bullet}=\tau_{\geqslant 2}K^{\bullet}(\Qp(1)), &&
%\\
%&f_3\,:\,A_3^{\bullet}\rightarrow C_3^{\bullet} \,\,
%\text{\rm is induced  by} &{\text{\rm }}&\text{\rm $C^{\bullet}(G_{F,S},\Qp(1))\rightarrow K^{\bullet}(\Qp(1)).$}  
%\\
%&g_3\,:\,B_3^{\bullet}\rightarrow C_3^{\bullet}\,\, \text{\rm 
%is the zero map}, &&
%\end{aligned}
%\end{equation}

\begin{mytheorem}
\label{theorem cup-product of Selmer complexes} 
i) There exists a canonical, up to homotopy, quasi-isomorphism
\[
r_S\,:\,E_3^{\bullet} \rightarrow A[-2].
\]

ii) The data 1-7) above satisfy the conditions {\bf P1-3)} of Section~\ref{products} and therefore define, for each $a\in A$ and each  quasi-isomorphism  $r_S,$  the cup product 
\[
\cup_{a,r_S}\,:\, S^{\bullet}(V,\bD,M) \otimes_A S^{\bullet}(V^*(1),\bD^{\perp},M^{\perp})
\rightarrow A[-3].
\]

iii) The homotopy class of   $\cup_{a,r_S}$ does not depend on the choice 
of $r\in A$ and, therefore,  defines   a pairing
\begin{equation}
\label{cup product for selmer}
\cup_{V,\bD,M} \,:\,\RG (V,\bD,M)\otimes _A^{\mathbf L}
\RG (V^*(1),\bD^{\perp},M^{\perp})
\rightarrow A[-3].
\end{equation}
\end{mytheorem}
\begin{proof} i) We repeat  {\it verbatim} the argument of \cite{Ne06}, Section 5.4.1. 
For each $\fq\in S$, let $i_\fq$ denote the composition of the canonical isomorphism $A\simeq H^2(F_\fq, A(1))$  of the local class field theory with the morphism 
\linebreak
$\tau_{\geqslant 2}C^{\bullet}(G_{F_\fq}, A(1))\rightarrow K^{\bullet}(A_\fq(1)).$ 
Then we have a commutative diagram
\begin{equation}
\nonumber 
\xymatrix{
\tau_{\geqslant 2} C^{\bullet}(G_{F,S}, A(1)) \ar[r]^{(\res_\fq)_\fq}
&\underset{\fq\in S}\bigoplus \tau_{\geqslant 2}K^{\bullet}(A_\fq(1))
\ar[r]^(.6){j} &E_3^{\bullet} \\
&\underset{\fq\in S}\bigoplus A[-2] \ar[u]^{(i_\fq)_\fq}
\ar[r]^{\Sigma} 
&A[-2] \ar[u]^{i_S},
}
\end{equation}
where $i_S=j\circ i_{\fq_0}$ for some fixed $\fq_0\in S$ and $\Sigma$ denotes 
the summation over $\fq\in S.$  By global class field
theory, $i_S$ is a quasi-isomorphism and, because $A[-2]$ is concentrated in degree 
$2,$ there exists a homotopy inverse $r_S$ of $i_S$ which is unique up to homotopy.

ii) We only need to show that the condition {\bf P3)} holds in our case.
Note that $\cup_A=\cup_c,$ $\cup_B=0$ and $\cup_C=\cup_K.$ 
From the definition of $\cup_K$  it follows immediately that
\begin{equation}
\label{proof theorem selmer eq 1}
\cup_K \circ (f_1\otimes f_2)=f_3\circ \cup_c.
\end{equation}
If $\fq\in S_p$ (resp. if $\fq\in \Sigma$), from the orthogonality of $\bD_\fq^{\perp}$ and $\bD_\fq$ (resp. from the orthogonality of $M_\fq$ and $M_\fq^{\perp}$)  it follows  that $\cup_K\circ (g_\fq\otimes g_\fq^{\perp})=0.$ If 
$\fq\in \Sigma_p\setminus \Sigma ,$  we have
$\cup_c \circ (g_\fq\otimes g_\fq^{\perp})=0$ by \cite{Ne06}, Lemma 7.6.4.
Since $g_3\circ \cup_B=0,$ this gives 
\begin{equation}
\label{proof theorem selmer eq 2}
\cup_C\circ (g_1\otimes g_2)= g_3\circ \cup_B =0.
\end{equation}
The equations (\ref{proof theorem selmer eq 1}) and (\ref{proof theorem selmer eq 2})
show that {\bf P3)} holds with $h_f=h_g=0.$ 
We define $\cup_{a,r_S}$ as the composition of the cup product constructed in  Proposition~\ref{construction of cup product for E_1, E_2} with $r_S.$
The rest of the theorem follows 
from Proposition~\ref{construction of cup product for E_1, E_2}.
\end{proof}

\subsubsection{\bf Remark}  Consider the case $\Sigma=\emptyset .$ The cup product $\cup_{V,\bD}$ induces a map 
\begin{equation}
\label{map from Selmer of dual to dual selmer}
\RG (V^*(1),\bD^{\perp}) \rightarrow \mathbf{R}\mathrm{Hom}_A(\RG (V,\bD),A)[-3],
\end{equation}  
which, in general, is not an isomorphism. For each $\fq\in S$ define
\[
\widetilde U_\fq^{\bullet}(V,\bD)=\mathrm{cone} \left ( U_\fq^{\bullet}(V,\bD) \xrightarrow{g_\fq} K^{\bullet}(V_\fq)
\right ) [-1]
\]
and $\widetilde{\RG} (F_\fq,V,\bD)=\left [\widetilde U_\fq^{\bullet}(V,\bD)\right ].$
Then the cup product 
\linebreak
$K^{\bullet}(V_\fq)\otimes K^{\bullet}(V_\fq^*(1))\rightarrow A[-2]$ 
induces a pairing
\[
\widetilde U_\fq^{\bullet}(V,\bD) \otimes U_\fq^{\bullet}(V^*(1),\bD^{\perp}) \rightarrow A[-2]
\]
which gives rise to a morphism 
\begin{equation}
\label{condition for duality of Selmer complexes}
\RG (F_\fq,V^*(1),\bD^{\perp}) \rightarrow 
\R\mathrm{Hom}_A (\widetilde{\RG} (F_\fq,V,\bD)     , A)[-2].  
\end{equation}
Repeating the arguments of \cite{Ne06} (see the proofs of Proposition 6.3.3 and 
Theorem 6.3.4 of {\it op. cit.}) it is easy to show that if   
$\RG (F_\fq,V,\bD)$ and $\RG (F_\fq,V^*(1),\bD^{\perp})$ are perfect 
and (\ref{condition for duality of Selmer complexes}) holds for all $\fq\in S,$ 
then  (\ref{map from Selmer of dual to dual selmer}) is an isomorphism. 
First assume that  $\fq\in S_p.$ The complexes $\RG (F_\fq,\bD)$ and 
$\RG (F_\fq,\bD^{\perp})$ are perfect by Theorem~\ref{Theorem KPX}. Consider the tautological exact sequence
\[
0\rightarrow \bD_\fq \rightarrow \Ddagrig (V_\fq) \rightarrow \widetilde \bD_\fq \rightarrow 0,
\]
where $\widetilde \bD_\fq=\Ddagrig (V_\fq)/\bD_\fq.$ Applying
the functor $\RG (F_\fq, -)$ to this sequence, we obtain an exact sequence 
\[
0\rightarrow \RG (F_\fq,\bD_\fq) \rightarrow \RG (F_\fq,\Ddagrig (V_\fq)) \rightarrow \RG (F_\fq, \widetilde \bD_\fq) \rightarrow 0,
\]
and therefore  in this case 
$\widetilde{\RG} (F_\fq,V,\bD)\simeq\RG (F_\fq,\widetilde \bD_\fq).$
From the definition of $\bD_\fq^{\perp}$
we have 
$\bD_\fq^{\perp}\simeq {\widetilde\bD_\fq}^*(\chi).$ 
Using Theorem~\ref{Theorem KPX}, we obtain 
\begin{multline}
\nonumber
\RG (F_\fq,\bD_v^{\perp}) \simeq \RG (F_\fq, {\widetilde\bD_\fq}^*(\chi)) \simeq \\
\simeq   \R\mathrm{Hom}_A (\RG (F_\fq,\widetilde\bD_\fq),A)[-2] 
\simeq
\R\mathrm{Hom}_A( \widetilde{\RG} (F_v,V,\bD), A)[-2],
\end{multline}
and therefore (\ref{condition for duality of Selmer complexes}) holds 
automatically for 
all $\fq\in S_p.$

If $\fq\in \Sigma_p,$ the complexes $\RG (F_\fq,V,\bD)$ and $\RG (F_\fq,V^*(1),\bD^{\perp})$
are not perfect in general and, in addition, (\ref{condition for duality of Selmer complexes}) does not always hold. However, (\ref{condition for duality of Selmer complexes}) holds in many important cases, in particular if $A$ is a finite extension
of $\Qp.$ 

\subsubsection{} Equip the complexes $A_i^{\bullet},$ $B_i^{\bullet}$ and
$C^{\bullet}_i$ with the transpositions given by 
\begin{equation}
\label{transpositions for A,B,C} 
\begin{aligned}
&\mathcal T_{A_1}=\mathcal T_{V,c},\\
&\mathcal T_{B_1}=\left (\underset{\fq\in S_p}\oplus\mathrm{id}_{C_{\Ph,\gamma}(\bD_\fq)}\right )
\oplus \left (\underset{\fq\in\Sigma_p\setminus \Sigma}\oplus \mathrm{id}_{C_{\ur}(V_\fq)}\right )
\oplus \left (\underset{\fq\in \Sigma}\oplus \mathcal T_{M_\fq,c}\right ),\\ 
&\mathcal T_{C_1}=\left (\underset{\fq\in S_p}\oplus \mathcal T_{K(V_\fq)}\right )\oplus 
\left (\underset{\fq\in \Sigma_p}\oplus \mathcal T_{V_\fq,c}\right ),\\
&\mathcal T_{A_2}=\mathcal T_{V^*(1),c},\\
&\mathcal T_{B_2}= \left (\underset{\fq\in S_p}\oplus\mathrm{id}_{C_{\Ph,\gamma}(\bD_\fq^{\perp})}\right )
\oplus \left (\underset{\fq\in\Sigma_p\setminus \Sigma}\oplus \mathrm{id}_{C_{\ur}(V_\fq^*(1))}\right )
\oplus \left (\underset{\fq\in \Sigma}\oplus \mathcal T_{M_\fq^{\perp},c}\right ),\\ 
&\mathcal T_{C_2}=\left (\underset{\fq\in S_p}\oplus \mathcal T_{K(V_\fq^*(1))}\right )\oplus 
\left (\underset{\fq\in \Sigma_p}\oplus \mathcal T_{V_\fq^*(1),c}\right ),\\
&\mathcal T_{A_3}=\mathcal T_{A(1),c},\\
&\mathcal T_{B_3}=\mathrm{id}, \\
&\mathcal T_{C_3}=\left (\underset{\fq\in S_p}\oplus \mathcal T_{K(A(1)_\fq)}\right )\oplus 
\left (\underset{\fq\in \Sigma_p}\oplus \mathcal T_{A(1)_\fq,c}\right ).
\end{aligned}
\end{equation}

\begin{mytheorem}
\label{theorem simmetricity of selmer cup product}
 i) The data (\ref{transpositions for A,B,C}) satisfy the conditions {\bf T1-7)} of Section~\ref{products}.
 
ii) We have a commutative diagram
\begin{equation*}
\xymatrix{
\RG (V,\bD,M) \otimes_A^{\mathbf L} \RG (V^*(1),\bD^{\perp},M^{\perp})
\ar[rr]^(.7){\cup_{V,\bD}} \ar[d]^{s_{12}} &&A[-3] \ar[d]^{=}\\
\RG (V^*(1),\bD^{\perp},M^{\perp}) \otimes_A^{\mathbf L} \RG (V,\bD,M)
\ar[rr]^(.7){\cup_{V^*(1),\bD^{\perp}}}  &&A[-3].
}
\end{equation*}
\end{mytheorem}
\begin{proof}
i) We check the conditions {\bf T3-7)} taking $\cup'_A=\cup_c,$ 
$\cup'_B=0$ and $\cup'_C=\cup_K.$  From (\ref{proof theorem selmer eq 1}) and  (\ref{proof theorem selmer eq 2}) it follows that 
{\bf T3)} holds if  we  take $h'_f=h'_g=0.$ To check the condition {\bf T4)}
we remark that, by Proposition~\ref{homotopy for Cg,Ph to K(V)},i) we have    $f_i\circ \mathcal T_A=\mathcal T_{C}\circ f_i$
and we can take $U_i=0.$ The existence of a homotopy $V_i$
follows from Proposition~\ref{homotopy for Cg,Ph to K(V)} ii)  and 
\cite{Ne06}, Proposition 7.7.3. Note that again we can set $V_i=0.$ 

We prove the  existence of  homotopies $t_A$, $t_B$ and $t_C$ 
satisfying  {\bf T5)}. From the commutativity of the diagram (\ref{commutativity of transpositions with cup-prod for cont. coh.}), it follows that 
$\cup_c \circ s_{12}\circ (\mathcal T_A \otimes
\mathcal T_A)=\mathcal T_A\circ \cup_c$ and we can take $t_A=0.$
Since $\cup_B'=\cup_B=0,$ we can take $t_B=0.$ We construct $t_C$
as a system of homotopies $(t_{C,\fq})_{\fq\in S}$ such that 
$t_{C,\fq}\,:\,\cup_c\circ s_{12}\circ (\mathcal{T}_{V_\fq,c}\otimes 
\mathcal{T}_{V(1)_\fq,c}) \rightsquigarrow \mathcal{T}_{A(1)_\fq,c}\circ \cup_c$
for $\fq\in \Sigma_p$ and 
$t_{C,\fq}\,:\,\cup_K\circ s_{12}\circ (\mathcal{T}_{K(V_\fq)}\otimes 
\mathcal{T}_{K(V(1)_\fq)}) \rightsquigarrow \mathcal{T}_{K(A(1)_\fq)}\circ \cup_K$
for $\fq\in S_p.$
As before, from  (\ref{commutativity of transpositions with cup-prod for cont. coh.}) it follows that for $\fq\in \Sigma_p$ one can take $t_{C,\fq}=0.$ If $\fq\in S_p,$ 
by Proposition~\ref{transpositions for T(A)}
we can set 
\begin{equation}
\label{homotopy t_C,v}
t_{C,\fq}((x_{n-1},x_n)\otimes (y_{m-1}\otimes y_m))=
(-1)^n(\mathcal{T}_{A(1)_\fq,c}(x_{n-1}\cup_c y_{m-1}),0)
\end{equation}
for $(x_{n-1},x_n)\in K^n(V_\fq)$ and $(y_{m-1},y_m)\in K^m(V^*(1)_\fq)$
(see (\ref{definition of the homotopy})). This proves {\bf T5)}.
From (\ref{homotopy t_C,v}) it follows that $t_C\circ(f_1\otimes f_2)=0$ 
and  it is easy to see that {\bf T6)} and {\bf T7)} hold 
if we take $H_f=H_g=0.$  

ii) For each Galois module $X,$ we denote by  $a_X\,:\,\id  \rightsquigarrow \mathcal T_{X,c}$  the homotopy 
(\ref{definition of homotopy a}). Recall that we can take $a_X$ such that $a_X^0=a_X^1=0.$ Consider the following homotopies 
\begin{equation}
\begin{aligned}
&k_{A_1}=a_{V}\,:\,\id  \rightsquigarrow \mathcal T_{A^{\bullet}_1}, &\text{ on $A^{\bullet}_1$,}\\
&k_{B_1}=\left (\underset{\fq\in S_p\cup \Sigma_p\setminus \Sigma}\oplus 
0_{U_\fq (V,\bD,M)}\right )
\oplus \left (\underset{\fq\in \Sigma}\oplus a_{M_\fq}\right ):\,\id \rightsquigarrow \mathcal F_{B_1} &\text{ on $B^{\bullet}_1$,}\\
&k_{C_1}=\left (\underset{\fq\in S_p}\oplus 
a_{K(V_\fq)}\right )
\oplus \left (\underset{\fq\in \Sigma_p}\oplus a_{V_\fq}\right )\,:\,
\id \rightsquigarrow \mathcal T_{C^{\bullet}_1}, &\text{ on $C^{\bullet}_1$.}
\end{aligned}
\end{equation}
We will denote by  $k_{A_2},$ $k_{B_2}$ $k_{C_2}$ the homotopies on $A_2^{\bullet},$ 
$B_2^{\bullet}$ and $C_2^{\bullet}$ defined by the analogous formulas. 
From Proposition~\ref{homotopy for Cg,Ph to K(V)}, ii) 
it follows that 
\begin{equation*}
\begin{aligned}
&f\circ k_{A_1}=k_{C_1}\circ f, &f^{\perp}\circ k_{A_2}=k_{C_2}\circ f^{\perp},\\
&g\circ k_{B_1}=k_{C_1}\circ g,    &g^{\perp}\circ k_{B_2}=k_{C_2}\circ g^{\perp}. 
\end{aligned}
\end{equation*}
By (\ref{induced morphism for cones}),  these data induce  transpositions
$\mathcal T_{V}^{\sel}$ and $\mathcal T_{V^*(1)}^{\sel}$ on 
$S^{\bullet}(V,\bD,M)$ and $S^{\bullet}(V^*(1),\bD^{\perp},M^{\perp}),$
and  the formula  (\ref{homotopy for cones}) of Subsection~\ref{subsubsection cones} defines  homotopies $k_{V}^{\sel}\,:\,\id  \rightsquigarrow \mathcal T_{V}^{\sel}$ and  $k_{V^*(1)}^{\sel}\,:\,\id  \rightsquigarrow \mathcal T_{V^*(1)}^{\sel}.$ By Proposition~\ref{proposition commutativity of products}, the following diagram commutes up to homotopy: 
\begin{equation*}
\xymatrix{
S^{\bullet}(V,\bD,M) \otimes_A S^{\bullet} (V^*(1),\bD^{\perp},M^{\perp})
\ar[rr]^(.7){\cup_{a,r_S}} \ar[d]^{s_{12}\circ (\mathcal T_{V}^{\sel}
\otimes \mathcal T_{V^*(1)}^{\sel})} &&A[-3] \ar[d]^{=}\\
S^{\bullet}(V^*(1),\bD^{\perp},M^{\perp}) \otimes_A S^{\bullet} (V,\bD,M)
\ar[rr]^(.7){\cup_{1-a,r_S}}  &&A[-3].
}
\end{equation*}
Now the theorem follows from the fact that the map 
$(k_V^{\sel}\otimes k_{V^*(1)}^{\sel})_1,$ given by (\ref{homotopy (h otimes k)_1)}), furnishes a homotopy between $\id$ and  $\mathcal T_{V}^{\sel}
\otimes \mathcal T_{V^*(1)}^{\sel}.$
\end{proof}

\subsection{$p$-adic height pairings}
\label{subsection construction of h^sel}
\subsubsection{}
We keep notation and conventions of the previous subsection. Let 
$F^{\cyc}=\underset{n=1}{\overset{\infty}\cup F}(\zeta_{p^n})$ denote the cyclotomic $p$-extension of $F.$
The Galois group $\Gamma_F=\Gal (F^{\cyc}/F)$ decomposes into 
the direct sum $\Gamma_F=\Delta_F\times \Gamma_F^0$ of the group
$\Delta_F=\Gal (F(\zeta_p)/F)$ and a $p$-procyclic group $\Gamma_F^0.$
We denote by $\chi \,:\,\Gamma_F\rightarrow \Zp^*$ the cyclotomic character
and by $\chi_\fq$ the restriction of $\chi$ on $\Gamma_\fq,$
$\fq\in S.$ 
As in Section~\ref{subsection Bockstein map}, we equip 
the group algebra $A[\Gamma_F^0]$ with the involution 
$\iota \,:\,A[\Gamma_F^0]\rightarrow A[\Gamma_F^0] $ such that $\iota (\g)=\g^{-1},$ $\g \in \Gamma_F^0.$ Fix a generator $\g_F$ of $\Gamma_F^0.$   

Set $\widetilde A_{F}^{\iota}=A[\Gamma_F^0]^{\iota}/(J_F^2),$ 
where $J_F$ is the augmentation ideal of $A[\Gamma_F^0].$ We have an exact  sequence 
\begin{equation}
\label{exact sequence with tilde A}
0\rightarrow A\xrightarrow{\theta_F} \widetilde A_F^{\iota} \rightarrow A\rightarrow 0,
\end{equation}
where $\theta_F (a)=a\widetilde X,$ and $\widetilde X=\log^{-1}(\chi (\g_F)) (\g-1)$ 
does not depend on the choice of $\g_F \in \Gamma_F^0.$  
For each $p$-adic representation $V$ with coefficients in $A$,
(\ref{exact sequence with tilde A}) induces an exact sequence 
\begin{equation}
\label{global sequence with widetilde V_F}
0\rightarrow V\rightarrow \widetilde V_F \rightarrow V\rightarrow 0,
\end{equation}
where  $\widetilde V_F=\widetilde A_F^{\iota}\otimes_A V.$
As in Section~\ref{subsection Bockstein map}, passing to continuous 
Galois cohomology, we obtain a distinguished triangle
\begin{equation}
\nonumber 
C^{\bullet}(G_{F,S},V) \rightarrow 
C^{\bullet}(G_{F,S},\widetilde V_F) \rightarrow 
C^{\bullet}(G_{F,S},V) \xrightarrow{\beta_{V,c}} 
C^{\bullet}(G_{F,S},V)[1]. 
\end{equation}
For each $\fq\in S,$ the inclusion $\Gamma_\fq^0\hookrightarrow \Gamma_F^0$
induces  a commutative diagram  
\begin{equation}
\nonumber 
\xymatrix{
0 \ar[r] &V_\fq \ar[d]^{=} \ar[r]^{\theta_\fq} &\widetilde V_{F_\fq}\ar[d] \ar[r]
&V_\fq \ar[d]^{=}\ar[r] &0\\
0 \ar[r]&\ar[r] V_{\fq} \ar[r]^{\theta_F} &(\widetilde V_F)_\fq \ar[r]
&V_\fq \ar[r] &0,
}
\end{equation}
where the vertical middle arrow is  an isomorphism by the five lemma. 
Taking into account Proposition~\ref{formula for beta_c}, we see that 
the exact sequence (\ref{global sequence with widetilde V_F})
induces  a distinguished triangle 
\begin{equation}
\nonumber 
C^{\bullet}(G_{F_\fq},V) \rightarrow 
C^{\bullet}(G_{F_\fq},\widetilde V_F) \rightarrow 
C^{\bullet}(G_{F_\fq},V) \xrightarrow{\beta_{V_\fq,c}} 
C^{\bullet}(G_{F_\fq},V)[1]. 
\end{equation}
where  
$
\beta_{V_\fq,c}(x)=-\log \chi_\fq \cup x.
$

Let  $\bD_\fq$ be a $(\Ph,\Gamma_\fq)$-submodule of $\Ddagrig (V_\fq)$ and let $\widetilde \bD_{F,\fq}=\widetilde A_F^{\iota}\otimes_A\bD_\fq.$
As in Section~\ref{subsection Bockstein map}, we have an exact sequence 
\begin{equation}
\label{exact sequence with widetilde D_{F,v}}
0\rightarrow \bD_\fq\rightarrow \widetilde\bD_{F,\fq} \rightarrow 
\bD_\fq\rightarrow 0
\end{equation}
which sits in the diagram
\begin{equation}
\nonumber 
\xymatrix{
0 \ar[r] &\bD_\fq \ar[d]^{=} \ar[r]^{\theta_\fq} &\widetilde \bD_{\fq}\ar[d] \ar[r]
&\bD_\fq \ar[d]^{=}\ar[r] &0\\
0 \ar[r]&\ar[r] \bD_{\fq} \ar[r]^{\theta_F} &\widetilde \bD_{F,\fq} \ar[r]
&\bD_\fq \ar[r] &0.
}
\end{equation}
Taking into account  Proposition~\ref{formula for beta_D}, we obtain
that  (\ref{exact sequence with widetilde D_{F,v}}) induces the distingushed 
triangle  
\[
C^{\bullet}_{\Ph,\g_\fq}(\bD_\fq) \rightarrow C^{\bullet}_{\Ph,\g_\fq}(\widetilde\bD_{F,\fq})
\rightarrow  C^{\bullet}_{\Ph,\g_\fq}(\bD_\fq) 
\xrightarrow{\beta_{\bD_\fq}} C^{\bullet}_{\Ph,\g_\fq}(\bD_\fq)[1], 
\]
where  $\beta_{\bD_\fq}(x)= -(0, \log \chi_\fq(\g_\fq))\cup x.$ 
%\begin{equation}
%\label{formula for beta^F_D_v} 
%\beta^{F}_{\bD_v}(x)= -(0,n_v)\cup x.
%\end{equation}
Finally, replacing in the exact sequence (\ref{exact sequence with K(widetilde V)})
$\widetilde V$ by $\widetilde V_F,$ and taking into account 
Proposition~\ref{formula for beta_K}  we obtain the distinguished triangle  
\[
K^{\bullet}(V_\fq)\rightarrow K^{\bullet}((\widetilde V_F)_\fq)\rightarrow K^{\bullet}(V_\fq)
\xrightarrow{\beta_{K(V_\fq)}}
K^{\bullet}(V_\fq)[1],
\]
where
$\beta_{K(V_\fq)} (x) = -(0,\log \chi_\fq)\cup x.$

%\begin{equation}
%\label{formula for beta_K^F}
%\beta^{F}_{K(V_v)} (x) = -\frac{1}{\log \chi (\g_F)} (0,\log \chi_v)\cup x.
%\end{equation}
If $\fq\in \Sigma_p,$ we construct the Bockstein map for 
$U_\fq^{\bullet}(V,\bD,M)$ 
following \cite{Ne06}, Section 11.2.4. Namely, if $\fq\in \Sigma,$
then $U_\fq^{\bullet}(V,\bD,M)=C^{\bullet}(G_{F_\fq},M_\fq)$ and
the exact sequence 
\begin{equation}
\label{exact sequence with Mfq}
0\rightarrow  M_\fq \rightarrow \widetilde M_{F,\fq} \rightarrow M_\fq \rightarrow 0
\end{equation} 
gives rise to a map 
$\beta_{M_\fq,c}\,:\,C^{\bullet}(G_{F_\fq},M_\fq) \rightarrow  
C^{\bullet}(G_{F_\fq},M_\fq).$ If $\fq\in \Sigma_p\setminus \Sigma,$ 
then 
$(\widetilde V_F)^{I_\fq}=V^{I_\fq}\otimes \widetilde A^{\iota}_{F}$ 
and we denote by  $s\,:\,V^{I_\fq}\rightarrow (\widetilde V_F)^{I_\fq}$
the section given by $s (x)=x\otimes 1.$ There exists  a distingushed triangle 
\[
C_{\ur}^{\bullet}(V_\fq)\rightarrow C_{\ur}^{\bullet}((\widetilde V_F)_\fq) \rightarrow 
C_{\ur}^{\bullet}(V_\fq)
\xrightarrow{\beta_{V_\fq,\ur}} C^{\bullet}_{\ur}(V_\fq)[1],
\]
where  
$
\beta_{V_\fq,\ur}\,:\, C^0_{\ur}(V_\fq) \rightarrow C^1_{\ur}(V_\fq)
$
is given by 
\[
\beta_{V_\fq,\ur}^F(x)=\frac{1}{\widetilde X}(ds-sd) (x)=-\log \chi_\fq (\Fr_\fq) x.
\]

\begin{myproposition}
\label{Selmer complexes satisfy product conditions} 
In addition to (\ref{transpositions for A,B,C}), 
equip  the complexes  $A_i^{\bullet}$, $B_i^{\bullet}$ and 
$C_i^{\bullet}$ ($1\leqslant i \leqslant 3$) with the Bockstein maps 
given by 
\begin{align*}
&\beta_{A_1}=\beta_{V,c},\\
&\beta_{B_1}=\,\left (\underset{\fq\in S_p}\oplus 
\beta_{\bD_\fq}\right )\oplus 
\left (\underset{\fq\in \Sigma_p}\oplus \beta_{V_\fq,\ur}\right ),\\ 
&\beta_{C_1}=\left (\underset{\fq\in S_p}\oplus \beta_{K(V_\fq)}\right )\oplus 
\left (\underset{\fq\in \Sigma_p}\oplus \beta_{V_\fq,c}\right ),\\
&\beta_{A_2}=\beta_{V^*(1),c},\\
&\beta_{B_2}=\,\left (\underset{\fq\in S_p}\oplus \beta_{\bD^{\perp}_\fq}\right )\oplus 
\left (\underset{\fq\in \Sigma_p}\oplus \beta_{V_\fq^*(1),\ur}\right ),\\ 
&\beta_{C_2}=\left (\underset{\fq\in S_p}\oplus \beta_{K(V_\fq^*(1))}\right )\oplus 
\left (\underset{\fq\in \Sigma_p}\oplus \beta_{V_\fq^*(1),c}\right ),\\
&\beta_{A_3}=\beta_{A(1),c},\\
&\beta_{B_3}=0,\\
&\beta_{C_3}=\left (\underset{\fq\in S_p}\oplus \beta_{K(A(1)_\fq)}\right )\oplus 
\left (\underset{\fq\in \Sigma_p}\oplus \beta_{A(1)_\fq,c}\right ).
\end{align*}
Then these data satisfy the conditions {\bf B1-5)} of Section~\ref{products}.
\end{myproposition}
\begin{proof}  We check  {\bf B2-5)} for our Bockstein maps. 
For each $\fq\in \Sigma_p,$ Nekov\'a\v r constructed  homotopies  
\begin{align}
&{v}_{V,\fq}\,:\, g_\fq \circ \beta_{V_\fq,\ur} \rightsquigarrow   
\beta_{V_\fq,c}\circ g_\fq,
\nonumber
\\
&v_{V^*(1),\fq}\,:\, g_\fq^{\perp} \circ \beta_{V_\fq^*(1),\ur}  
\rightsquigarrow
\beta_{V_\fq^*(1),c}\circ g_\fq^{\perp}. \nonumber 
\end{align}
From Proposition~\ref{formula for beta_K}, ii) it follows that for all $\fq\in S_p$
\begin{align}
&g_\fq \circ \beta_{\bD_\fq} =  \beta_{K(V_\fq)}\circ g_\fq, \nonumber \\
&g_\fq^{\perp} \circ \beta_{\bD_\fq}  
=
\beta_{K(V_\fq^*(1))}\circ g_\fq^{\perp}. \nonumber 
\end{align}
Set $v_{V,\fq}=v_{V^*(1),\fq}=0$ for all $\fq\in S_p.$ Then the condition  {\bf B2)}
holds for $u_i=0$ and $v_i=(v_{i,\fq})_{\fq\in S}.$   

In {\bf B3)}, we can set $h_B=0$ because $\cup_B=0.$
The existence of a homotopy  $h_A$ between $\cup_A[1]\circ (\id \otimes \beta_{A,2})$
and $\cup_A[1]\circ (\beta_{A,1}\otimes \id)$ is proved in \cite{Ne06}, Section 11.2.6 
and the same method allows to construct $h_C.$ 
Namely, we construct a system $h_C=(h_{C,\fq})_{ \fq \in S}$ of homotopies such that
$h_{C,\fq}\,:\, \cup_c[1]\circ (\id \otimes \beta_{V_\fq^*(1),c}) \rightsquigarrow
\cup_c[1]\circ (\beta_{V_\fq,c}\otimes \id)$ for $\fq\in \Sigma_p$ and 
$h_{C,\fq}\,:\, \cup_K[1]\circ (\id \otimes \beta_{K(V_\fq^*(1))}) \rightsquigarrow
\cup_K[1]\circ (\beta_{K(V_\fq)}\otimes \id)$ for $\fq\in S_p.$
For $\fq\in \Sigma_p,$ the construction of $h_{C,\fq}$ is the same as those of $h_A.$ 
Now, let $\fq\in S_p.$ By Proposition~\ref{formula for beta_K},
one has  
$\beta_{K(V_\fq)}(x)=-
(0,\log \chi_\fq)\cup_K x.$ 
Consider the following diagram, where 
$z_\fq=(0,\log \chi_\fq)$
\begin{equation}
\label{diagram of prop. 13}
\xymatrix{
K^{\bullet}(V_\fq)\otimes_A K^{\bullet}(V_\fq^*(1))\ar[d]^{=}\ar[r]^{\id} &K^{\bullet}(V_\fq)\otimes_A K^{\bullet}(V_\fq^*(1))\ar[d]^{=}\\
K^{\bullet}(V_\fq)\otimes_A A \otimes_A K^{\bullet}(V_\fq^*(1))\ar[d]^{\id \otimes (-z_\fq) \otimes \id} \ar[r]^{\id} &A\otimes_A K^{\bullet}(V_\fq)\otimes_A K^{\bullet}(V_\fq^*(1))
\ar[d]^{(-z_\fq)\otimes \id \otimes \id}\\
K^{\bullet}(V_\fq)\otimes_A K^{\bullet}(A)[1]\otimes_A K^{\bullet}(V_\fq^*(1))\ar[d]^{\cup_K\otimes \id}
\ar[r]^{s_{12}\otimes \id} & K^{\bullet}(A)[1]\otimes_A K^{\bullet}(V_\fq)\otimes_A K^{\bullet}(V_\fq^*(1))\ar[d]^{\cup_K\otimes \id}\\
K^{\bullet}(V_\fq)[1]\otimes_A K^{\bullet}(V_\fq^*(1))\ar[d]^{\cup_K}\ar[r]^{\id} &K^{\bullet}(V_\fq)[1]\otimes_A K^{\bullet}(V_\fq^*(1))\ar[d]^{\cup_K}\\
K^{\bullet}(V_\fq\otimes V_\fq^*(1))[1]\ar[r]^{\id} &K^{\bullet}(V_\fq\otimes V_\fq^*(1))[1].
}
\end{equation}
The first, second and fourth squares of this diagram are commutative.
From Proposition ~\ref{transpositions for T(A)} (see also (\ref{homotopy for K(V)})) it follows that  the diagram 
\[
\xymatrix{
K^{\bullet}(V_\fq)\otimes K^{\bullet}(A)[1] \ar[d]^{\cup_K}
\ar[rr]^{s_{12}\circ (\mathcal T_K\otimes \mathcal T_K)} & &K^{\bullet}(A)[1]\otimes K^{\bullet}(V_\fq)\ar[d]^{\cup_K}\\
K^{\bullet}(V_\fq)[1] \ar[rr]^{\mathcal T_K}&   &K^{\bullet}(V_\fq)[1] 
}
\]
is commutative up to  some homotopy $k_1\,:\,\mathcal T_K\circ \cup_K
\rightsquigarrow \cup_K\circ s_{12}\circ (\mathcal T_K\otimes \mathcal T_K).$
Since $\mathcal T_K^2=\id,$ we have a homotopy
\[
\mathcal T_K\circ k_1\,:\, \cup_K
\rightsquigarrow \mathcal T_K\circ \cup_K\circ s_{12}\circ (\mathcal T_K\otimes \mathcal T_K).
\] 
By \cite {Ne06}, Section 3.4.5.5 (see also Section~\ref{subsubsection homotopy a}),  for any topological $G_{F_\fq}$-module $M$ there exists a functorial homotopy $a\,:\,\id  \rightsquigarrow \mathcal T_c.$ 
By Proposition~\ref{homotopy for Cg,Ph to K(V)}, $a$ induces  a homotopy 
between $\id \,:\,K^{\bullet}(V_\fq)\rightarrow K^{\bullet}(V_\fq)$ and 
$\mathcal T_K\,:\,K^{\bullet}(V_\fq)\rightarrow K^{\bullet}(V_\fq)$
which we denote  by $a_K.$  Let 
$(a_K\otimes a_K)_1\,:\,\id   \rightsquigarrow \mathcal T_K\otimes \mathcal T_K$
denote the homotopy between the maps $\id$ and 
$\mathcal T_K\otimes \mathcal T_K\,:\,
K^{\bullet}(V_\fq)\otimes K^{\bullet}(\Qp)[1]\rightarrow K^{\bullet}(V_\fq)\otimes K^{\bullet}(\Qp)[1]$ given by (\ref{homotopy (h otimes k)_1)}). Then 
\begin{multline}
 d (a_K\circ \mathcal T_K\circ \cup_K\circ s_{12}\circ (\mathcal T_K
\otimes \mathcal T_K))+ (a_K\circ \mathcal T_K\circ \cup_K\circ s_{12}\circ (\mathcal T_K
\otimes \mathcal T_K))\,d=\\
=(\mathcal T_K-\id)\circ \cup_K\circ s_{12}\circ  (\mathcal T_K
\otimes \mathcal T_K),  \nonumber 
\end{multline}
and 
\begin{multline}
d (\cup_K\circ s_{12}\circ (a_K\otimes a_K)_1)+(\cup_K\circ s_{12}\circ 
(a_K\otimes a_K)_1)=\\
=\cup_K\circ s_{12}\circ (\mathcal T_K\otimes \mathcal T_K-\id). \nonumber
\end{multline}
Therefore the formula 
\begin{equation}
\label{definition of homotopy k_2}
k_2= a_K\circ \mathcal T_K\circ \cup_K\circ s_{12}\circ (\mathcal T_K
\otimes \mathcal T_K)+\cup_K\circ s_{12}\circ (a_K\otimes a_K)_1
\end{equation}
defines a homotopy
\[
k_2\,:\,\cup_K\circ s_{12}
\rightsquigarrow
\mathcal T_K \circ \cup_K\circ s_{12}\circ  (\mathcal T_K
\otimes \mathcal T_K).
\]
Then $k_{C,\fq}=\mathcal F_K\circ k_1-k_2$ defines a homotopy
\[
k_{C,\fq}\,:\,\cup_K \rightsquigarrow \cup_K\circ s_{12}
\]
and we proved that the third square of the diagram  (\ref{diagram of prop. 13})
commutes up to a homotopy. We define the homotopy
\[
h_{C,\fq}\,:\, \cup_K[1]\circ (\id \otimes \beta_{K(V_\fq^*(1)),c}) \rightsquigarrow
\cup_K[1]\circ (\beta_{K(V_\fq)}\otimes \id)
\]
by 
\begin{equation}
\label{definition of h_{C,v}}
h_{C,\fq}=\cup_K\circ (k_{C,\fq}\otimes \id)\circ (\id \otimes (-z_\fq)\otimes \id).
\end{equation}
This proves {\bf B3)}. 

Since $u_1=u_2=h_f=0,$ the condition  {\bf B4)}
reads
\begin{equation}
\label{condition with K_f}
dK_f-K_fd=-h_C\circ (f_1\otimes f_2)+f_3[1]\circ h_A
\end{equation}
for some second order homotopy $K_f.$ 
It is proved in \cite{Ne06}, Section 11.2.6, that if $\fq\in \Sigma_p,$ 
then 
\begin{equation}
\label{case v does not divide p}
h_{C,\fq} \circ (f_1\otimes f_2)= \text{\rm res}_\fq \circ h_A.
\end{equation}
Assume that $\fq\in S_p.$ 
Recall (see \cite{Ne06}, Section 11.2.6) that the homotopy $h_A$ is 
given by 
\begin{equation}
\label{definition of homotopy h_A}
h_A=\cup_c \circ (k_A\otimes \id) \circ (\id \otimes (-z) \otimes  \id),
\end{equation}
where $z=\log \chi$ and 
\begin{equation}
\label{definition of homotopy k_A}
k_A=-a\circ (\cup_c \circ s_{12} \circ (\mathcal T_c \otimes  \mathcal T_c))-
(\mathcal T_c\circ \cup_c \circ s_{12}) \circ (a\otimes a)_1.
\end{equation}
From (\ref{definition of the homotopy}), it follows that for all 
$x\in C^{n}(G_{F,S},V)$ and $y\in  C^{m}(G_{F,S},V^*(1))$ we have 
\begin{multline}
\label{computation of k_1}
(k_1\otimes \id)\circ (\id \otimes (-z_\fq) \otimes \id) \circ (f_1\otimes f_2)(x\otimes y)=
\\
=(k_1\otimes \id)( (0,-\log \chi_\fq)\otimes (0,x_\fq) \otimes (0,y_\fq))=\\
=k_1((0,-\log \chi_\fq  )\otimes (0,x_\fq))\otimes (0,y_\fq)=0, 
\end{multline}
where $x_\fq=\res_\fq (x),$ $y_\fq=\res_\fq (y).$
On the other hand, comparing (\ref{definition of homotopy k_2}) and 
(\ref{definition of homotopy k_A}) we see that 
\begin{multline}
\label{computation of k_2}
(k_2\otimes \id) \circ (\id \otimes (-z_\fq)\otimes \id) \circ (f_1\otimes f_2) (x\otimes y)=\\
=k_2((0,-\log \chi_\fq  )\otimes (0,x_\fq))\otimes (0,y_\fq)=\\
=-(0, \res_\fq ( k_A (-z \otimes x)))\otimes (0,y_\fq). 
\end{multline}
From (\ref{computation of k_1}), (\ref{computation of k_2}), 
(\ref{definition of h_{C,v}}) and (\ref{definition of homotopy k_A}) we obtain that
\begin{multline}
\label{computation of h_{C,v}}
h_{C,\fq} \circ (f_1\otimes f_2)(x\otimes y)= \\
=(0, \res_\fq ( k_A (-z \otimes x))) \cup_K 
(0,y_\fq)=\\
=(0, \res_\fq ( k_A (-z \otimes x))\cup_c y)=
(0, \res_\fq (h_A (x\otimes y))).
\end{multline}
From (\ref{computation of h_{C,v}}) and  (\ref{case v does not divide p}) it follows 
that $h_C\circ (f_1\otimes f_2)=f_3[1]\circ h_A$ and 
therefore we can set $K_f=0$ in (\ref{condition with K_f}).  Thus, {\bf B4)} is proved.

It remains to check {\bf B5)}. Since $v_1=v_2=h_g=0,$ this condition reads
\begin{equation}
\label{condition for K_g}
dK_g-K_gd=-h_C\circ (g_1\otimes g_2)+\cup_{C[1]}\circ (v_1\otimes g_2)-
\cup_{C[1]}\circ (g_1\otimes v_2)
\end{equation}
for some second order homotopy $K_g.$ Write $K_g=(K_{g,\fq})_{\fq\in S}.$ 
For $\fq\in \Sigma_p,$ Nekov\'a\v r proved that the $\fq$-component of the right hand side of 
(\ref{condition for K_g}) is equal to zero. 
For $\fq\in S_p,$ we have   $v_{1,\fq}=v_{2,\fq}=0$ 
and  $h_{C,v}\circ (g_1\otimes g_2)=0$ because of orthogonality of $\bD$ and 
$\bD^{\perp},$ and again we can set $K_{g,\fq}=0.$  To sum up, the condition  
(\ref{condition for K_g}) holds for $K_g=0.$ The proposition is proved.
\end{proof}

\subsubsection{}
The exact sequences (\ref{global sequence with widetilde V_F}),  
(\ref{exact sequence with widetilde D_{F,v}}) and
(\ref{exact sequence with Mfq}) give rise to a 
distinguished triangle 
\[
\RG (V,\bD,M) \rightarrow \RG (\widetilde V_{F},\widetilde \bD_{F},\widetilde M_F) \rightarrow \RG (V,\bD, M) \xrightarrow{\delta_{V,\bD,M}} \RG (V,\bD,M)[1]
\]

\begin{definition} The $p$-adic height pairing associated to the data $(V,\bD,M)$ 
is defined as the morphism 
\begin{multline*}
\label{definition of height via selmer}
h_{V,\bD,M}^{\sel}\,:\,\RG(V,\bD,M) \otimes_A^{\mathbf L} \RG(V^*(1),\bD^{\perp},M^{\perp})
\xrightarrow{\delta_{V,\bD,M}}\\
\rightarrow \RG(V,\bD,M) [1] \otimes_A^{\mathbf L} \RG(V^*(1),\bD^{\perp},M^{\perp}) \xrightarrow{\cup_{V,\bD,M}} A[-2],
\end{multline*}
where $\cup_{V,\bD,M}$ is the pairing (\ref{cup product for selmer}).
\end{definition}

The height pairing  $h_{V,\bD,M}^{\sel}$ induces a pairing
\begin{equation}
\label{pairing sel for H^1}
h_{V,\bD,M,1}^{\sel}\,\,:\,\,H^1(V,\bD,M)\otimes_A H^1(V^*(1),\bD^{\perp},M^{\perp}) \rightarrow A.
\end{equation}

\begin{mytheorem} 
\label{theorem symmetricity of h^sel}
The diagram
\begin{equation}
\nonumber
\xymatrix{
\RG(V,\bD,M) \otimes_A^{\mathbf L} \RG(V^*(1),\bD^{\perp},M^{\perp}) 
\ar[rr]^(.7){h_{V,\bD,M}^{\sel}} \ar[d]^{s_{12}} &&A[-2]\ar[d]^{=}\\
\RG(V^*(1),\bD^{\perp},M^{\perp}) \otimes_A^{\mathbf L} \RG(V,\bD,M)
\ar[rr]^(.7){h_{V^*(1),\bD^{\perp},M^{\perp}}^{\sel}}
&&A[-2]
}
\end{equation}
is commutative. In particular, the pairing $h_{V,\bD,1}^{\sel}$
is skew-symmetric.
\end{mytheorem}
\begin{proof}
From Propositions~\ref{properties of products implies symmeticity} and \ref{Selmer complexes satisfy product conditions} it follows, that 
the diagram 
\begin{equation}
\nonumber
\xymatrix{
S^\bullet (V,\bD,M) \otimes_A S^{\bullet} (V^*(1),\bD^{\perp},M^{\perp}) 
\ar[rr]^(.7){h_{V,\bD,M}^{\sel}} 
\ar[d]^{s_{12}\circ (\mathcal T_V^\sel\otimes \mathcal T_{V^*(1)}^\sel)} &&E_3\ar[d]^{=}\\
S^\bullet (V^*(1),\bD^{\perp},M^{\perp}) \otimes_A S^{\bullet}(V,\bD,M)
\ar[rr]^(.8){h_{V^*(1),\bD^{\perp},M^{\perp}}^{\sel}}
&&E_3
}
\end{equation}
is commutative up to homotopy. Now  the theorem follows from the fact, that
$(\mathcal T_V^\sel\otimes \mathcal T_{V^*(1)}^\sel)$ is homotopic to the identity map (see the proof of Theorem~\ref{theorem simmetricity of selmer cup product}).
\end{proof}

\section{Splitting submodules}
\label{section splitting submodules}

\subsection{Splitting submodules} 
\label{subsection splittings submodules}
\subsubsection{}
Let $K$ be a finite extension of $\Qp,$ and  let $V$ be a potentially semistable representation
of $G_K$ with coefficients in a finite extension $E$ of $\Qp.$ 
For each finite extension $L/K$ we set 
\linebreak
$\bD_{*/L}(V)=(\mathbf{B}_{*}\otimes V)^{G_L},$
where $*\in\{\cris,\st,\dR\}$ and write  $\bD_* (V)=\bD_{*/K}(V)$ 
if $L=K.$ We will use the same convention for the functors $\mathcal{D}_{*/L}.$

Fix a finite Galois extension $L/K$ such that the restriction of $V$ on $G_L$
is semistable.  Then $\DstL(V)$ is a filtered
$(\Ph,N,G_{L/K})$-module and $\DdrL(V)=\DstL(V)\otimes_{L_0}L.$
A $(\Ph,N,G_{L/K})$-submodule of $\DstL(V)$ is a $L_0$-subspace $D$ of 
$\DstL(V)$ stable under the action of $\Ph$, $N$ and $G_{L/K}$.

\begin{definition} We say that  a $(\Ph,N,G_{L/K})$-submodule $D$ of $\DstL (V)$ is a splitting
submodule if 
\[
\DdrL(V)=D_L\oplus \F^0\DdrL(V), \qquad D_L=D\otimes_{L_0}L
\]
as $L$-vector spaces.
\end{definition}

In Subsections~\ref{subsection splittings submodules}-\ref{subsection canonical splitting}   we will always assume that $V$ satisfies the following condition:
\newline
\newline
{\bf S)} $\Dc (V)^{\Ph=1}=\Dc (V^*(1))^{\Ph=1}=0.$
\newline
\newline
If $D$ is a splitting submodule, we denote by $\bD$ the $(\Ph,\Gamma_K)$-submodule of
$\Ddagrig (V)$ associated to $D$ by Theorem~\ref{berger theorem2}.  The natural embedding $\bD \rightarrow \Ddagrig (V)$ induces
a map $H^1(\bD) \rightarrow H^1(\Ddagrig (V))\iso H^1(K,V).$ Passing to duals,
we obtain a map $H^1(K,V^*(1)) \rightarrow H^1(\bD^*(\chi)).$

\begin{myproposition}
\label{proposition properties of splitting submodules}
Assume that $V$ satisfies the condition {\bf S)}.
Let $D$ be a splitting submodule. Then 

i) $H^1_f(K,V^*(1)) \rightarrow H^1_f(\bD^*(\chi))$ is the zero map.

ii) $\im (H^1(\bD)\rightarrow H^1(K,V)) = H^1_f(K,V)$ and 
the map $H^1_f(\bD) \rightarrow  H^1_f(K,V)$ is an isomorphism.

iii) If, in addition, $\F^0(\DstL (V)/D)^{\Ph=1,N=0,G_{L/K}}=0,$ then
\linebreak
$H^1(\bD)=H^1_f(K,V).$
\end{myproposition}

\begin{proof} i) By Proposition~\ref{proposition properties H^1_f} we have a commutative diagram
\begin{equation}
\label{diagram from proposition properties of splitting submodules}
\xymatrix{H^1_{\text{\rm cris}}(V^*(1))
 \ar[d]^= \ar[r]  &H^1_{\text{\rm cris}}(\bD^*(\chi))   \ar[d]^{=}\\
H^1_f(K,V^*(1)) \ar[r]  &H^1_f(\bD^*(\chi)),}
\end{equation}
where we set 
\[
H^1_{\text{\rm cris}}(V^*(1))=\text{\rm coker} \left (\Dc (V)\xrightarrow{(1-\Ph,\pr)}
\Dc (V)\oplus t_V(K) \right )
\]
and 
\[
H^1_{\text{\rm cris}}(\bD^*(\chi))=\text{\rm coker} \left (\CDcris (\bD^*(\chi))\xrightarrow{(1-\Ph,\pr)}
\CDcris (\bD^*(\chi)))\oplus t_{\bD^*(\chi)}(K) \right )
\]
to simplify notation. 

Since $\Dc(V^*(1))^{\Ph=1}=0,$ the map $1-\Ph\,:\,\Dc (V^*(1)) \rightarrow \Dc (V^*(1))$
is an isomorphism 
 and  
$H^1_{\text{\rm cris}}(V^*(1))=
t_{V^*(1)}(K).$ On the other hand, all Hodge--Tate weights of $\bD^*(\chi)$ are $\geqslant 0$
and $t_{\bD^*(\chi)}(K)=0.$ 
Hence
\linebreak
$H^1_{\text{\rm cris}}(\bD^*(\chi))=\text{\rm coker}( 1-\Ph \,\mid\,\CDcris (\bD^*(\chi)))$
and  the upper map in (\ref{diagram from proposition properties of splitting submodules}) is zero because it is induced by the canonical
projection of $t_{V^*(1)}(K)$ on  $t_{\bD^*(\chi)}(K).$
This proves i).

Now we prove ii). Using i) together with the  orthogonality property of $H^1_f$ we obtain
that the map
\begin{equation}
\text{\rm Hom}_{E}(H^1(K,V)/H^1_f(K,V),E) 
\rightarrow \text{\rm Hom}_{E}(H^1(\bD)/H^1_f(\bD),E),
\nonumber
\end{equation}
induced by  $H^1(\bD)\rightarrow H^1(K,V),$ is zero. This implies that
the image of $H^1(\bD)$ is $H^1(K,V)$ is contained in $H^1_f(K,V).$
Finally  one has a diagram
\begin{equation}
\xymatrix{H^1_{\text{\rm cris}}(\bD)
 \ar[d]^{\simeq} \ar[r]  &H^1_{\text{\rm cris}}(V)   \ar[d]^{\simeq}\\
H^1_f(\bD) \ar[r]  &H^1_f(K,V).}
\nonumber
\end{equation}
From {\bf S)}
it follows that the top arrow can be identified with the natural map
$ t_{\bD}(K)\rightarrow t_{V}(K)$ which is an isomorphism by the definition of a splitting submodule.

iii) Taking into account ii), we only need to prove that the natural 
map $H^1(\bD)\rightarrow H^1(K,V)$ is injective. This follows from the exact sequence
\[
0\rightarrow \bD \rightarrow \Ddagrig (V)\rightarrow \bD' \rightarrow 0,\qquad \bD'=\Ddagrig (V)/\bD
\]
and the fact that $H^0(\bD')= \F^0(\DstL (V)/D)^{\Ph=1,N=0,G_{L/K}}=0$ 
(see Proposition~\ref{proposition properties H^1_f}, i)).
\end{proof}

\subsection{The canonical splitting}
\label{subsection canonical splitting}
\subsubsection{}
 Let
\begin{equation}
y\,:\qquad 0 \rightarrow V^*(1) \rightarrow Y_y
\rightarrow E \rightarrow 0
\nonumber
\end{equation}
be an extension of $E$ by $V^*(1)$. 

Passing to $(\Ph,\Gamma_K)$-modules, we obtain an extension
\begin{equation}
0 \rightarrow \Ddagrig (V^*(1))
\rightarrow \Ddagrig (Y_y)
\rightarrow \CR_{K,E} \rightarrow 0.
\nonumber
\end{equation}
By duality, we have  exact sequences
\begin{equation}
0 \rightarrow E(1) \rightarrow Y_y^*(1)
\rightarrow V \rightarrow 0
\nonumber
\end{equation}
and 
\begin{equation}
0\rightarrow \CR_{K,E}(\chi) \rightarrow  \Ddagrig (Y_y^*(1))
\rightarrow \Ddagrig (V) \rightarrow 0.
\nonumber
\end{equation}
We  denote
by $[y]$ the class of $y$ in 
$\Ext^1_{E[G_K]}(E,V^*(1))\iso H^1(K,V^*(1)).$
Assume that $y$ is crystalline, i.e. that $[y]\in H^1_f(K,V^*(1)).$
Let $D$ be a splitting submodule of $\DstL(V).$ 
Consider the commutative  diagram
\begin{equation}
\xymatrix{
y\,:\qquad 0\ar[r] &\Ddagrig (V^*(1))\ar[r] \ar[d]^{\pr} &\Ddagrig (Y_y)\ar[r] \ar[d] &\CR_{K,E}\ar[r] \ar[d]^{=}&0\\
\pr (y):\qquad 0 \ar[r]&\bD^*(\chi)\ar[r] &\bD_y^*(\chi)\ar[r] &\CR_{K,E}\ar[r] &0}
\nonumber
\end{equation}
where $\bD_y$ is the inverse image of $\bD$ in $ \Ddagrig (Y_y^*(1)).$ 
The class of $\pr (y)$ in $H^1(\bD^*(\chi))$ is  the image
of $[y]$ under the map
\begin{equation}
\Ext^1 (\CR_{K,E},\Ddagrig (V^*(1))) \rightarrow  \Ext^1 (\CR_{K,E},\bD^*(\chi))
\nonumber
\end{equation}
which coincides with the map 
\[
H^1(K,V^*(1))=H^1(\Ddagrig (V^*(1))) \rightarrow H^1(\bD^*(\chi))
\]
after the identification of  $\Ext^1(\CR_{K,E},-)$ with the cohomology group $H^1(-)$. 
Since we are  assuming that $[y]\in H^1_f(K,V^*(1)),$ by Proposition~\ref{proposition properties of splitting submodules} i), we obtain that 
 $[\pr(y)]=0.$ Thus the sequence $\pr (y)$ splits.

\subsubsection{}We will construct a canonical splitting of $\pr (y)$ using the idea of Nekov\'a\v r \cite{Ne92}. 
Since  $\dim_{E}\Dc (Y_y)=\dim_{E}\Dc (V^*(1))+1,$ the 
sequence
\begin{equation*}
0 \rightarrow \Dc (V^*(1)) \rightarrow \Dc (Y_y)
\rightarrow \Dc (E) \rightarrow 0 
\end{equation*}
is exact by the dimension argument. From  $\Dc (V^*(1))^{\Ph=1}=0$ 
and the snake lemma it follows  that 
\linebreak
$\Dc (Y_y)^{\Ph=1}=\Dc (E)$ and we obtain
a canonical $\Ph$-equivariant morphism of $K_0$-vector spaces 
$\Dc (E) \rightarrow \Dc (Y_y).$
By linearity, this map extends to a  $(\Ph,N,G_{L/K})$-equivariant morphism of
$L_0$-vector spaces $\DstL (E) \rightarrow \DstL(Y_y).$  Therefore we have a canonical
splitting 
\begin{equation}
\DstL (Y_y)\iso \DstL (V^*(1)) \oplus \DstL (E)
\nonumber
\end{equation}
of the sequence 
\begin{equation}
0 \rightarrow \DstL (V^*(1)) \rightarrow \DstL (Y_y)
\rightarrow \DstL (E) \rightarrow 0 
\nonumber
\end{equation}
in the category of $(\Ph,N,G_{L/K})$-modules. This splitting induces 
a $(\Ph,N,G_{L/K})$-equivariant isomorphism 
\begin{equation}
\label{splitting isomorphism}
\mathcal{D}_{\st/L}(\bD_y^*(\chi)) \iso  \mathcal{D}_{\st/L}(\bD^*(\chi))\oplus 
\mathcal{D}_{\st/L} (E) . 
\end{equation}
Moreover, since all Hodge--Tate weights of $\bD^*(\chi)$ are positive,
we have $\F^i \mathcal{D}_{\dR/L}(\bD_y^*(\chi)) \iso \F^i \mathcal{D}_{\dR/L}(\bD^*(\chi))\oplus 
\F^i\mathcal{D}_{\dR/L} (E) 
$
and therefore the  isomorphism 
\begin{equation}
\mathcal{D}_{\dR/L}(\bD_y^*(\chi)) \iso  \mathcal{D}_{\dR/L}(\bD^*(\chi))\oplus 
\mathcal{D}_{\dR/L} (E) 
\nonumber
\end{equation}
is compatible with filtrations.  Thus, we obtain that (\ref{splitting isomorphism}) is an isomorphism  in the category of filtered
$(\Ph,N,G_{L/K})$-modules. This gives a canonical splitting
\begin{equation}
\xymatrix{\pr (y):\qquad 0 \ar[r]&\bD^*(\chi)\ar[r] &\bD_y^*(\chi)\ar[r] 
&\CR_{K,E}\ar[r] \ar@<-1ex>@{.>}[l] &0}
\nonumber
\end{equation}
of the extension $\pr (y).$ 
Passing to duals, we obtain a splitting 
\begin{equation}
\label{splitting extension of ph-Gamma-modules}
\xymatrix{
0 \ar[r] &\CR_{K,E} (\chi) \ar[r] &\bD_y \ar[r] &\bD \ar[r]	
\ar@<-1ex>@{.>}[l]_{s_{\bD,y}} &0.}
\end{equation}
Taking cohomology,  we get a splitting 
 \begin{equation}
\label{splitting extension of cohomology at v dividing p} 
\xymatrix{
0 \ar[r] &H^1_f(K, E(1)) \ar[r] &H^1_f(\bD_y) \ar[r] &H^1_f(\bD) \ar[r]
\ar@<-1ex>@{.>}[l]_{s_y} &0.}
\end{equation}
Our constructions can be summarized in the diagram
\begin{equation}
\xymatrix{
0 \ar[r] &H^1_f(K, E(1)) \ar[r] \ar[d]^{=} &H^1_f(\bD_y) \ar[r] \ar[d]^{\simeq} &H^1_f(\bD) \ar[r]
\ar@<-1ex>@{.>}[l]_{s_y}\ar[d]^{\simeq} &0\\
0 \ar[r] &H^1_f(K, E(1)) \ar[r] &H^1_f(K,Y_y^*(1)) \ar[r] &H^1_f(K,V) \ar[r]
&0.}
\nonumber
\end{equation}
Here the vertical maps are isomorphisms by Proposition~\ref{proposition properties of splitting submodules} and the five lemma. 
\newline
\,

\subsubsection{\bf Remark} Assume that $H^0(\bD^*(\chi))=0.$ Then each crystalline extension of
$\bD$ by $\CR_K(\chi)$ splits uniquely. This follows from Proposition~\ref{proposition properties H^1_f} i)
which implies that  $H^1_f(\bD^*(\chi))=0$ and from
the fact that various splittings are parametrized by $H^0(\bD^*(\chi))$.

\subsection{Filtration associated to a splitting submodule}
\label{subsection filtration}
\subsubsection{}
In this subsection we assume that $K=\Qp.$ We review the construction 
of the canonical filtration on $\DstL (V)$ associated to a splitting
submodule $D$.  Note that this construction
is the direct generalization of the filtration constructed by Greenberg \cite{Gr94} in the ordinary case.  See \cite{Ben11} for more detail.

Let $V$ be a potentially semistable representation of $G_{\Qp}$ with coefficients 
in a finite extension $E$ of $\Qp.$ As before, we fixe a finite 
Galois extension $L/\Qp$ such that $V$ is semistable over $L$ and 
denote by $\DstL (V)$ the semistable module of the restriction
of $V$ on $G_L.$  
Let $D\subset \DstL (V)$ be a splitting submodule. 
Set $D'=\DstL (V)/D.$ Then
$\F^0 D'= D'$ and we define $M_1=(D')^{\Ph=1,N=0,G_{L/\Qp}}.$
For the dual filtered $(\Ph,N,G_{L/\Qp})$-module $D^*=\mathrm{Hom}_{L_0\otimes E}(D,\DstL (E(1)))$ we have 
$\F^0D^*=D^*$ and we define  $M_0=\left ((D^*)^{\Ph=1,N=0,G_{L/\Qp}}\right )^*.$ 
We have canonical projections  $\pr_{D'}\,\,:\,\,\DstL (V)\rightarrow D'$ 
and $\pr_{M_0}\,\,:\,\,D\rightarrow M_0.$
Define a five-step filtration 
\begin{multline*}
\{0\}=F_{-2}\DstL (V) \subset F_{-1}\DstL (V)\subset F_0\DstL (V)\subset\\
 F_1\DstL (V)\subset F_2\DstL (V)=
\DstL (V)
\end{multline*}
by
\begin{equation*}
F_i\DstL (V)=\begin{cases} \ker (\pr_{M_0})&\text{if $i=-1$,}\\
D&\text{if $i=0$,}\\
\pr_{D'}^{-1}(M_1) &\text{if $i=1$.}
\end{cases}
\end{equation*}
Set $W=F_1\DstL (V)/F_{-1}\DstL (V).$ 
These data can be represented by the  diagram
\begin{equation*}
\xymatrix{
0 \ar[r] & D \ar[r] \ar@{->>}[d]^{\pr_{M_0}} &\Dst^L(V) \ar[r]^{\,\,\pr_{D'}} & {D'} \ar[r] &0\\
0 \ar[r]         &M_0  \ar[r]    &W\ar[r]                                       &M_1 \ar[r]\ar@{^{(}->}[u]
&0.}
\end{equation*}

By Theorem~\ref{berger theorem2}, the  filtration $\left (F_i\DstL (V)\right )_{i=-2}^2$    induces a filtration 
\linebreak
 $\left (F_i\Ddagrig (V)\right )_{i=-2}^2$ on the 
$(\Ph,\Gamma_{\Qp})$-module $\Ddagrig (V)$ such that 
\[\mathcal{D}_{\st/L} (F_i\Ddagrig (V))=F_i\DstL (V).
\] 
Note that 
 $\bD=F_0\Ddagrig (V).$ 
We also set 
$\bM_0 =F_0\Ddagrig (V)/F_{-1}\Ddagrig (V),$ 
\linebreak
$\bM_1 =F_1\Ddagrig (V)/F_{0}\Ddagrig (V)$ and 
$\W=F_1\Ddagrig (V)/F_{-1}\Ddagrig (V).$  We have a tautological exact sequence 
\begin{equation}
\label{exact sequence with W}
0\rightarrow \bM_0 \xrightarrow{\alpha} \W \xrightarrow{\beta} \bM_1\rightarrow 0.
\end{equation}
By construction, $\bM_0$ and $\bM_1$ are crystalline  $(\Ph,\Gamma_{\Qp})$-modules such that 
\[
\mathcal{D}_{\mathrm{cris}/\Qp}(\bM_0)=M_0, \qquad \mathcal{D}_{\mathrm{cris}/\Qp}(\bM_1)=M_1.
\]
Since 
\begin{align*}
&M_0^{\Ph=p^{-1}}=M_0, &&\F^0M_0=0,\\
&M_1^{\Ph=1}=M_1, &&\F^0M_1=M_1,
\end{align*}
the structure of modules $\bM_0$ and $\bM_1$  
is given by Proposition~\ref{proposition isoclinic modules}. In particular, 
we have canonical decompositions
\begin{equation*} 
H^1(\bM_0)\overset{(\pr_f,\pr_c)}\simeq H^1_f(\bM_0)\oplus H^1_c(\bM_0),\qquad
H^1(\bM_1)\overset{(\pr_f,\pr_c)}\simeq H^1_f(\bM_1)\oplus H^1_c(\bM_1).
\end{equation*}
The exact sequence (\ref{exact sequence with W}) induces the cobondary map
$\delta_0\,:\,H^0(\bM_1)\rightarrow H^1(\bM_0).$ 
Passing to cohomology in the dual exact sequence
\begin{equation}
\label{dual exact sequence with W}
0\rightarrow \bM_1^*(\chi) \rightarrow \W^*(\chi) \rightarrow 
\bM_0^*(\chi)\rightarrow 0,
\end{equation}
we obtain 
the coboundary map $\delta_0^*\,:\,H^0(\bM_0^*(\chi))\rightarrow 
H^1(\bM_1^*(\chi)).$ 

\subsubsection{}In the remainder of this subsection we assume that $(V,D)$ 
satisfies the following conditions:
\newline
\newline
{\bf{F1)}}  For all $i\in\Z$ 
\[
\CDpst (\Ddagrig (V)/F_1\Ddagrig (V))^{\Ph=p^i}= 
\CDpst (F_{-1}\Ddagrig (V))^{\Ph=p^i}=0.
\]
{\bf F2)}  The composed maps 
\begin{align*}
&H^0(\bM_1)\xrightarrow{\delta_0} H^1(\bM_0) \xrightarrow{\pr_c} H^1_c(\bM_0),
\\
&H^0(\bM_1)\xrightarrow{\delta_0} H^1(\bM_0) \xrightarrow{\pr_f} H^1_f(\bM_0),
%&H^0(\bM_0^*(\chi))\xrightarrow{\delta^*_0} H^1(\bM_1^*(\chi)) \xrightarrow{\pr_c} %H^1_c(\bM_1^*(\chi)),
\end{align*}
where the second arrows denote canonical projections, are isomorphisms. 

\subsubsection{\bf Remarks} 1) It is easy to see, that the filtration on 
$\Ddagrig (V^*(1))$ associated to the dual splitting submodule $D^{\perp}=
\Hom(\DstL (V)/D, \DstL (E(1))),$ is dual to the filtration $F_i\DstL (V).$
In particular, 
\begin{align*}
&F_{-1}\Ddagrig (V^*(1))^*(\chi)\simeq \Ddagrig (V)/F_1\Ddagrig (V),\\
&\Ddagrig (V^*(1))/F_1\Ddagrig (V^*(1)) \simeq (F_{-1}\Ddagrig (V))^*(\chi),
\end{align*}
and the sequence (\ref{exact sequence with W}) for $(V^*(1),D^{\perp})$
coincides with (\ref{dual exact sequence with W}).
In particular, the  conditions {\bf F1-2)} hold for $(V,D)$ if and only if they hold for $(V^*(1),D^{\perp}).$

2) The condition {\bf F1)} implies that 
\[
H^0(\Ddagrig (V)/F_1\Ddagrig (V))= H^0(F_{-1}\Ddagrig (V))=0.
\]

3) If $V$ is semistable over $\Qp$, and the linear map $\Ph\,:\,\Dst (V)\rightarrow \Dst (V)$ is semisimple at $1$ and $p^{-1}$,   it is easy to see that  
\begin{displaymath}
F_i\Dst  (V)\,=\,\begin{cases} 
(1-p^{-1}\Ph^{-1})\,D+N(D^{\Ph=1}) &\text{if $i=-1$,}\\
D &\text{if $i=0$,}\\
D+\Dst  (V)^{\Ph=1} \cap N^{-1}(D^{\Ph=p^{-1}})&\text{if $i=1$.}
\end{cases}
\end{displaymath}
In this form, the filtration $F_i\Dst (V)$ was constructed in \cite{Ben11}.
\newline
\,

We summarize some properties of the filtration $F_i\Ddagrig (V).$

\begin{myproposition} 
\label{properties of filtration D_i}
Let $D$ be a regular submodule of $\Dst (V)$ such that $(V,D)$ satisfies 
the conditions {\bf F1-2)}. Then

i) $\mathrm{rk} (\bM_0)=\mathrm{rk} (\bM_1).$

ii) $H^0(\W)=0$ 

iii) The representation $V$ satisfies {\bf S)}, namely 
\[\Dc (V)^{\Ph=1}=\Dc (V^*(1))^{\Ph=1}=0.\]

iv) One has $H^1_f(F_{-1}\Ddagrig (V))=H^1(F_{-1}\Ddagrig (V))$ and $H^1_f(F_1\Ddagrig (V))=H^1_f(\Qp,V).$

%iii) The canonical maps induce inclusions 
%\[H^1(\bD_{-1}) \subset H^1_f(\bD)
%\subset H^1(\bD_{1}) \subset H^1(\Qp,V)
%\]

%v) The  sequences 
%\begin{eqnarray}
%&&0\rightarrow H^1(\bD_{-1}) \rightarrow H^1(\bD_1) 
%\rightarrow H^1(\W)\rightarrow 0, \nonumber \\
%&&0\rightarrow H^1(\bD_{-1}) \rightarrow H^1_f(\bD_1) 
%\rightarrow H^1_f(\W)\rightarrow 0 \nonumber
%\end{eqnarray}
%are exact.

v) We have exact sequences
\begin{equation}
\label{exact sequence H^1_f(W)}
0\rightarrow H^0(\bM_1)\rightarrow H^1(\bM_0) \rightarrow H^1_f(\W)
\rightarrow 0
\end{equation}
and 
\begin{equation}
\label{exact sequence H^1_f(Q_p,V)}
0\rightarrow H^0(\bM_1)\rightarrow H^1(\bD) \rightarrow H^1_f(\Qp,V)
\rightarrow 0.
\end{equation}
\end{myproposition}
\begin{proof} i) From {\bf F2)} and the fact that $\dim_EH^0(\bM_1)=
\mathrm{rk}(\bM_1)$ and
\linebreak
$\dim_EH^1_c(\bM_0)=\mathrm{rk}(\bM_0)$ (see Proposition~\ref{proposition isoclinic modules})
we obtain that $\mathrm{rk} (\bM_0)=\mathrm{rk} (\bM_1).$

ii) By Proposition~\ref{proposition isoclinic modules}, iv),   $H^0(\bM_0)=0,$
and  we have an exact sequence
\[
0\rightarrow H^0(\W)\rightarrow H^0(\bM_1)\xrightarrow{\delta_0}H^1(\bM_0).
\]
By {\bf F2)}, the map $\delta_0$ is injective and therefore $H^0(\W)=0.$
Applying the same argument to the dual exact sequence, we obtain 
that $H^0(\W^*(\chi))=0.$

iii)  First prove that $\mathcal{D}_{\mathrm{cris}}(\W)=
\mathcal{D}_{\mathrm{cris}}(\bM_0).$
The exact sequence (\ref{exact sequence with W}) gives an exact sequence 
\[
0\rightarrow \mathcal{D}_{\mathrm{cris}}(\bM_0) \xrightarrow{\alpha} 
\mathcal{D}_{\mathrm{cris}}(\W)\xrightarrow{\beta} \mathcal{D}_{\mathrm{cris}}(\bM_1)
\]
and we have immediately the inclusion 
$\mathcal{D}_{\mathrm{cris}}(\bM_0)\subset\mathcal{D}_{\mathrm{cris}}(\W).$ 
Thus, it is enough to check that $\dim_E\mathcal{D}_{\mathrm{cris}}(\W)=
\dim_E\mathcal{D}_{\mathrm{cris}}(\bM_0).$ Assume that 
\linebreak
 $\dim_E\mathcal{D}_{\mathrm{cris}}(\W)>
\dim_E\mathcal{D}_{\mathrm{cris}}(\bM_0).$ Then there exists 
$x\in \mathcal{D}_{\mathrm{cris}}(\W)$ such that $m=\beta (x)\neq 0.$
Since $\Ph$ acts trivially on $\mathcal{D}_{\mathrm{cris}}(\bM_1)=\bM_1^{\Gamma_{\Qp}},$
$\CR_{\Qp,E}m$ is a $(\Ph,\Gamma_{\Qp})$-submodule of $\bM_1,$ and there exists 
a submodule $\X\subset \W$ which sits in the  following commutative diagram
with exact rows
\[
\xymatrix{
0\ar[r] &\bM_0 \ar[r] \ar[d]^{=} &\X \ar[r] \ar[d] &\CR_{\Qp,E}m
\ar[r] \ar[d] &0\\
0\ar[r] &\bM_0 \ar[r]  &\W \ar[r]  &\bM_1
\ar[r]  &0.
}
\]
Since $\CDcris (\W)=(\W [1/t])^{\Gamma_{\Qp}},$ 
there exists $n\geqslant 0$ such that $t^nx\in \X,$ and therefore
$x\in \CDcris (\X).$ This implies that $\X$ is crystalline, and by
Proposition~\ref{proposition properties H^1_f} iv) 
we have a commutative diagram
\[
\xymatrix{
Em \ar @{^{(}->}[d] \ar[r] &H^1_f(\bM_0) \ar @{^{(}->}[d]\\
H^0(\bM_1)\ar @{^{(}->}[r]^{\delta_0} &H^1(\bM_0).
}
\]
Thus, $\mathrm{Im} (\delta_0)\cap H^1_f(\bM_0)\neq \{0\}$ and 
the condition {\bf F2)} is violated. This proves that  $\mathcal{D}_{\mathrm{cris}}(\W)=
\mathcal{D}_{\mathrm{cris}}(\bM_0).$ 

Now we can finish the proof of iii). Taking invariants, we have 
\linebreak
$\CDcris (\W)^{\Ph=1}=\CDcris (\bM_0)^{\Ph=1}=0.$ 
By {\bf F1)}   
\[\mathcal{D}_{\mathrm{cris}}(F_{-1}\Ddagrig (V))^{\Ph=1}=
\mathcal{D}_{\mathrm{cris}}(\Ddagrig (V)/F_1\Ddagrig (V))^{\Ph=1}=0,
\]
and, applying the functor $\mathcal{D}_{\mathrm{cris}}(-)^{\Ph=1}$ to the exact sequences
\begin{align*}
&0\rightarrow F_1\Ddagrig (V) \rightarrow \Ddagrig (V)\rightarrow 
\Ddagrig (V)/F_1\Ddagrig (V),\\
&0\rightarrow F_{-1}\Ddagrig (V) \rightarrow F_1\Ddagrig (V)
\rightarrow 
\W \rightarrow 0,
\end{align*}
we obtain that  $\Dc (V)^{\Ph=1}\subset \mathcal{D}_{\mathrm{cris}}(\W)^{\Ph=1}=0.$ 
The same argument shows that $\Dc (V^*(1))^{\Ph=1}=0.$

iv) By {\bf F1)} together with Proposition~\ref{proposition properties H^1_f} and
the Euler--Poincar\'e characteristic formula, we have
\begin{multline*}
\dim_EH^1(F_{-1}\Ddagrig (V))-\dim_EH^1_f(F_{-1}\Ddagrig (V))=\\
=\dim_E H^0((F_{-1}\Ddagrig (V))^*(\chi))=0,
\end{multline*}
and therefore $H^1_f(F_{-1}\Ddagrig (V))=H^1(F_{-1}\Ddagrig (V)).$
Since 
\linebreak $H^0(\Ddagrig (V)/F_1\Ddagrig (V))=0,$ the exact sequence 
\[
0\rightarrow F_1\Ddagrig (V) \rightarrow \Ddagrig (V)
\rightarrow \Ddagrig (V)/F_1\Ddagrig (V) \rightarrow 0
\]
induces, by Proposition~\ref{proposition properties H^1_f} iv),  an exact sequence 
\[
0\rightarrow H^1_f(F_1\Ddagrig (V)) \rightarrow H^1_f(\Ddagrig (V)).
\rightarrow H^1_f(\Ddagrig (V)/F_1\Ddagrig (V))
\rightarrow 0
\]
On the other hand,  since 
\[
\CDdr (\Ddagrig (V)/F_1\Ddagrig (V))=\F^0\CDdr (\Ddagrig (V)/F_1\Ddagrig (V)),
\]
by Proposition~\ref{proposition properties H^1_f}, i) we have  
\begin{equation*}
\dim_E H^1_f(\Ddagrig (V)/F_1\Ddagrig (V))=\dim_E H^0(\Ddagrig (V)/F_1\Ddagrig (V))=0,
\end{equation*}
and therefore 
$
 H^1_f(F_1\Ddagrig (V)) = H^1_f(\Ddagrig (V))=H^1_f(\Qp,V).
$

v) To prove the exacteness of (\ref{exact sequence H^1_f(W)}), we only need to show that the image of the map 
$\alpha\,:\,H^1(\bM_0)\rightarrow H^1(\W),$ induced by the exact sequence 
(\ref{exact sequence with W}),
 coincides with  $H^1_f(\W).$
By {\bf F2)}, $\mathrm{Im} (\delta_0)\cap H^1_f(\bM_0)=\{0\},$ and therefore 
the map $H^1_f(\bM_0)\rightarrow H^1_f(\W)$ is injective. 
Set $e=\mathrm{rk}(\bM_0)=\mathrm{rk}(\bM_1).$ 
Since 
\begin{equation*}
\dim_E H^1_f(\W)= \dim_E t_{\W}(\Qp)-H^0(\W)= e=\dim_EH^1_f(\bM_0),
\end{equation*}
we obtain that $H^1_f(\bM_0)= H^1_f(\W).$
On the other hand, the exact sequence 
\[
0\rightarrow H^0(\bM_1)\xrightarrow{\delta_0} H^1(\bM_0)\xrightarrow{\alpha} H^1(\W)
\]
shows that $\dim_E\mathrm{Im}(\alpha)=\dim_EH^1(\bM_0)-
\dim_EH^0(\bM_1)= e=\dim_E H^1_f(\bM_0).$ Therefore $\mathrm{Im}(\alpha)=
H^1_f(\bM_0)=H^1_f(\W),$ and the exacteness of 
(\ref{exact sequence H^1_f(W)}) is proved. 

Since $H^0(\W)=0$ and $H^1_f(F_{-1}\Ddagrig (V))=H^1(F_{-1}\Ddagrig (V)),$ 
by Proposition~\ref{proposition properties H^1_f} iv)  we have an exact sequence 
\[
0\rightarrow H^1(F_{-1}\Ddagrig (V)) \rightarrow H^1_f(F_{1}\Ddagrig (V))
\rightarrow H^1_f(\W)\rightarrow 0,
\]
which shows that  $H^1_f(F_{1}\Ddagrig (V))$ is the inverse image of 
$H^1_f(\W)$ under the map  
$H^1(F_{1}\Ddagrig (V))\rightarrow H^1(\W).$  Therefore we have the following
commutative diagram with exact rows 
\begin{equation*}
\xymatrix{
&  &0\ar[d] &0\ar[d] &\\
&  &H^0(\bM_1) \ar[d] \ar[r]^= &H^0(\bM_1) \ar[d] &\\
0 \ar[r] &H^1(F_{-1}\Ddagrig (V)) \ar[d]^{=} \ar[r] &H^1(\bD) \ar[d]
\ar[r] &H^1(\bM_0) \ar[r] \ar[d] &0\\
0 \ar[r] &H^1(F_{-1}\Ddagrig (V)) \ar[r] &H^1_f(F_{1}\Ddagrig (V )) \ar[d]
\ar[r] &H^1_f(\W) \ar[r] \ar[d] &0\\
& &0 &0. &
}
\end{equation*}
Since the right column of this diagram is exact, the five lemma 
gives the exacteness of the middle column.  Now the exacteness of (\ref{exact sequence H^1_f(Q_p,V)}) follows from the fact that 
$H^1_f(F_1\Ddagrig (V))=H^1_f(\Qp,V)$ by iv).
\end{proof}

\subsubsection{} The tautological exact sequence
\[
0\rightarrow \bD \rightarrow \Ddagrig (V) \rightarrow \bD'\rightarrow 0
\]
induces the coboundary map
\[
\partial_0\,:\,H^0(\bD') \rightarrow H^1(\bD),
\] 
which is injective because $H^0(\Qp,V)=0$ by Proposition~\ref{properties of filtration D_i} ii).

\begin{myproposition}
\label{proposition about decomposition of H^1(D)}
Let $V$ be a $p$-adic representation of $G_{\Qp}$ which satisfies the conditions   {\bf F1-2)}. Then
\[
H^1(\bD)= H^1_{\Iw} (\bD)_{\Gamma_{\Qp}^0}\oplus \partial_0 \left (H^0(\bD')\right ).
\]
\end{myproposition}
\begin{proof}

 Since $\CDpst \left ( \left ( F_{-1}\Ddagrig (V)\right )^*(\chi)\right )^{\Ph=p^i}=0$  for all $i\in \Z,$ by Lemma~\ref{lemma about H^2_Iw} we have $H^2_{\Iw}(F_{-1}\Ddagrig (V))=0.$ 
Then the tautological exact sequence 
\[
0\rightarrow F_{-1}\Ddagrig (V) \rightarrow \bD \rightarrow \bM_0\rightarrow 0
\]
induces an exact sequence 
\[
0\rightarrow H^1_{\Iw}(F_{-1}\Ddagrig (V)) \rightarrow H^1_{\Iw}(\bD)
 \rightarrow H^1_{\Iw}(\bM_0)\rightarrow 0.
\]
Since $H^1_{\Iw}(\bM_0)^{\Gamma_{\Qp}^0}=H^0(\bM_0)=0$ by Proposition~\ref{properties of filtration D_i}, the snake lemma gives an exact sequence
\begin{equation}
\label{exact sequence from proposition about decomposition of H^1}
0\rightarrow H^1_{\Iw}(F_{-1}\Ddagrig (V))_{\Gamma_{\Qp}^0} \rightarrow H^1_{\Iw}(\bD)_{\Gamma_{\Qp}^0} 
 \rightarrow H^1_{\Iw}(\bM_0)_{\Gamma_{\Qp}^0} \rightarrow 0.
\end{equation}
The Hochschild--Serre exact sequence 
\[
0\rightarrow H^1_{\Iw}(F_{-1}\Ddagrig (V))_{\Gamma_{\Qp}^0} \rightarrow
 H^1(F_{-1}\Ddagrig (V)) \rightarrow
H^2_{\Iw}(F_{-1}\Ddagrig (V))^{\Gamma_{\Qp}^0} \rightarrow 0
\]
together with the fact that
\begin{multline*}
\dim_E H^2_{\Iw}(F_{-1}\Ddagrig (V))^{\Gamma_{\Qp}^0}=
\dim_E H^2_{\Iw}(F_{-1}\Ddagrig (V))_{\Gamma_{\Qp}^0}=\\
=\dim_E H^0 \left ((F_{-1}\Ddagrig (V))^*(\chi) \right )=0
\end{multline*}
implies that  $H^1_{\Iw}(F_{-1}\Ddagrig (V))_{\Gamma_{\Qp}^0} =
 H^1(F_{-1}\Ddagrig (V)).$
 On the other hand, $H^1_{\Iw}(\bM_0)_{\Gamma_{\Qp}^0}=H^1_c(\bM_0)$ by
 Proposition~\ref{coinvariants of H^1_Iw =H^1_c}. Therefore, the sequence 
 (\ref{exact sequence from proposition about decomposition of H^1}) reads
 \[
 0\rightarrow H^1(F_{-1}\Ddagrig (V)) \rightarrow H^1_{\Iw}(\bD)_{\Gamma_{\Qp}^0} 
 \rightarrow H^1_c(\bM_0) \rightarrow 0
 \]
 and we have a commutative diagram
\begin{equation}
\label{diagram 1 proof proposition about decomposition of H^1(D)}
\xymatrix{
0\ar[r] &H^1(F_{-1}\Ddagrig (V)) \ar[r] \ar[d]^{=} &H^1_{\Iw}(\bD)_{\Gamma_{\Qp}^0} 
\ar @{^{(}->}[d] \ar[r]& H^1_c(\bM_0) \ar[r] \ar @{^{(}->}[d] &0 \\
0\ar[r] &H^1(F_{-1}\Ddagrig (V)) \ar[r] &H^1(\bD)
\ar[r]& H^1(\bM_0) \ar[r]  &0.
}
\end{equation} 
Since $H^0(\Ddagrig (V)/F_1\Ddagrig (V))=0,$ the exact sequence 
\[
0\rightarrow \bM_1\rightarrow \bD' \rightarrow  \Ddagrig (V)/F_1\Ddagrig (V)\rightarrow 0
\]
gives $H^0(\bM_1)=H^0(\bD')$ and we have a   commutative diagram
\begin{equation}
\label{diagram 2 proof proposition about decomposition of H^1(D)}
\xymatrix{
H^0(\bD')\ar[r]^{\partial_0} &H^1(\bD) \ar[d]\\
H^0(\bM_1) \ar[u]^{=} \ar @{^{(}->}[r]^{\delta_0} &H^1(\bM_0).
}
\end{equation} 
Finally, from {\bf F2)} it follows that   $H^1_c(\bM_0)\cap \delta_0 \left (H^0(\bM_1)\right )=\{0\},$
and the dimension argument shows that 
\begin{equation}
\label{decomposition H^1(bM0) in proof proposition about decomposition of H^1(D)}
H^1(\bM_0)= H^1_c(\bM_0) \oplus \delta_0 \left (H^0(\bM_1)\right ).
\end{equation}
Now, the proposition follows from (\ref{decomposition H^1(bM0) in proof proposition about decomposition of H^1(D)}) and    the diagrams (\ref{diagram 1 proof proposition about decomposition of H^1(D)}) and (\ref{diagram 2 proof proposition about decomposition of H^1(D)}).
\end{proof}

\section{$p$-adic height pairings II: universal norms}
\label{section universal norms}

\subsection{The pairing $h_{V,D}^{\norm}$} 
\label{subsection pairing h-norm}
\subsubsection{} In this section, we construct  the pairing $h^{\norm}_{V,D},$
which is a direct generalization of the pairing constructed in \cite{Ne92},
Section 6 and \cite{PR92}. 
Let  $V$ is a $p$-adic representation
of $G_{F,S}$ with coefficients in a finite extension $E$ of $\Qp.$
We fix  a system $\bD=(\bD_\fq)_{\fq\in S_p}$ of submodules $\bD_\fq\subset \Ddagrig (V_\fq)$
and denote by $\bD^{\perp}=(\bD_\fq^{\perp})_{\fq\in S_p}$ the orthogonal complement of 
$\bD.$ We have  tautological exact  sequences 
\begin{equation}
\nonumber
0\rightarrow \bD_\fq \rightarrow \Ddagrig (V_\fq) \rightarrow \bD'_\fq
\rightarrow 0, \qquad \fq\in S_p,
\end{equation}
where $\bD'_\fq=\Ddagrig (V_\fq)/\bD_\fq.$ Passing to duals, we have exact sequences
\begin{equation}
\label{equivalent condition N}
\nonumber
0\rightarrow (\bD'_\fq)^*(\chi_\fq) \rightarrow \Ddagrig (V_\fq^*(1)) \rightarrow \bD_\fq^*(\chi_\fq)
\rightarrow 0,
\end{equation}
where
$
(\bD'_\fq)^*(\chi_\fq)=\bD_\fq^{\perp}.
$
If the contrary is not explicitly stated, in this section we will assume that the following conditions hold
\newline

{\bf N1)} $H^0(F_\fq,V)=H^0(F_\fq,V^*(1))=0$ for all $\fq\in S_p;$ 

{\bf N2)} $H^0(\bD'_\fq)=H^0(\bD_\fq^*(\chi_\fq))=0$ for all $\fq\in S_p.$ 
\newline
\newline
From {\bf N2)}, it follows immediately that $H^1(\bD_\fq)$ injects into 
$H^1(F_\fq,V).$ By our definition of Selmer complexes we have 
\begin{equation}
H^1(V,\bD)\simeq \ker 
 \left ( H^1_S(V)\rightarrow \underset{\fq\in \Sigma_p}\bigoplus \frac{H^1(F_\fq,V)}{H^1_f(F_\fq,V)}\right ) \bigoplus 
\left (\underset{v\in S_p}\bigoplus \frac{H^1(F_\fq,V)}{H^1(\bD_\fq)}\right ),
\nonumber 
\end{equation}
and the same formula holds for $V^*(1)$ if we replace $\bD_\fq$ by  $\bD_\fq^{\perp}.$ 
Using this isomorphism, we identify each cohomology class $[(x, (x_\fq^+), (\lambda_\fq)]\in H^1(V,\bD)$ with the corresponding class $[x]\in H^1_S(V).$

\subsubsection{} Let $[y]\in H^1(V^*(1), \bD^{\perp})$ and let $Y_y$ be  the associated extention 
\[
0\rightarrow V^*(1) \rightarrow Y_y \rightarrow E \rightarrow 0.
\]
Passing to duals, we have an exact sequence 
\[
0\rightarrow E(1) \rightarrow Y_y^*(1) \rightarrow V \rightarrow 0.
\] 
For each $\fq\in S_p,$ this sequence induces an exact sequence of $(\Ph,\Gamma_\fq)$-modules
\[
0\rightarrow \CR_{F_\fq,E}(\chi_\fq) \rightarrow 
\Ddagrig(Y_y^*(1)_\fq) \rightarrow \Ddagrig (V_\fq) \rightarrow 0.
\] 
Consider the commutative  diagram
\begin{equation*}
\xymatrix{
0\ar[r] &H^1(F_\fq, E(1)) \ar[d]^{=} \ar[r]&H^1(\bD_{\fq,y})
\ar[r]^{\pi_{\bD,\fq}} \ar@{^{(}->}[d]_{g_{\fq,y}} &H^1(\bD_\fq) \ar[r]^{\delta_{\bD,\fq}^1} \ar@{^{(}->}[d]_{g_\fq} &
H^2(F_\fq,E(1)) \ar[d]^{=}\\
0 \ar[r]& H^1(F_\fq,E(1)) \ar[r] &H^1 (F_\fq,Y_y^*(1)) \ar[r]^{\pi_\fq} &H^1(F_\fq,V) \ar[r]^{\delta_{V,\fq}^1}
&H^2(F_\fq, E(1))\\
0 \ar[r]& H^1_S(E(1)) \ar[r]^{} \ar[u] &H^1_S(Y_y^*(1)) \ar[r]^{\pi} \ar[u]^{\res_\fq} &H^1_S(V)
\ar[r]^{\delta_V^1} \ar[u]^{\res_\fq}
&H^2_S( E(1)), \ar[u]^{\res_\fq}
}
\end{equation*}
where $\bD_{\fq,y}$ denotes the inverse image of $\bD_\fq$ in $\Ddagrig (V_\fq).$

In the following lemma we do not assume that the condition {\bf N2)} holds. 
\begin{mylemma}
\label{lemma construction h^norm}
Assume that  $V$ is a $p$-adic representation satisfying the condition {\bf N1)}. 
Let $[x]=[(x, (x_\fq^+), (\lambda_\fq))]\in H^1(V,\bD)$ and let 
$x_\fq=\res_\fq (x).$ Then 

i) If $\fq\nmid p,$ then  $H^1_f(F_\fq,E(1))=0$ and 
\begin{equation}
\nonumber
H^1_f(F_\fq,Y_y^*(1))\simeq H^1_f(F_\fq,V).
\end{equation}

ii) For each $\fq\in S_p$ one has  $\delta_{V,\fq}^1([x_\fq])=\delta_{\bD,\fq}^1([x_\fq^+])=0.$

iii) $\delta_V^1([x])=0.$

iv)  The sequence
\[
0\rightarrow H^1(E(1),\CR (\chi)) \rightarrow H^1(Y_y^*(1),\bD_y)
\rightarrow H^1(V,\bD) \rightarrow 0,
\]
where $\CR (\chi)=(\CR_{F_\fq,E}(\chi_\fq))_{\fq \in S_p},$ is exact.
\end{mylemma}
\begin{proof}
i) If $\fq\nmid p$, then $E(1)$ is unramified at $\fq$, 
$H^0(F_\fq^{\ur}/F_\fq, E(1))=0$ and
\begin{equation}
\nonumber
H^1_f( F_\fq,E(1)) =H^1(F_\fq^{\ur}/F_\fq, E(1))= E(1)/({\Fr}_\fq-1)E(1)=0.
\end{equation}
Since $[y]$ is unramified at $\fq$,  the sequence
\begin{equation}
\nonumber
0 \rightarrow E(1) \rightarrow Y_y^*(1)^{I_\fq}\rightarrow V^{I_\fq}
\rightarrow 0
\end{equation}
is exact. Passing to the associated long exact cohomology sequence of 
$\Gal (F_\fq^{\ur}/F_\fq)$
and taking into account that 
\[
H^1(F_\fq^{\ur}/F_\fq, E(1))=H^2(F_\fq^{\ur}/F_\fq,E(1))=0
\]
we obtain that $H^1(F_\fq^{\ur}/F_\fq,Y_y^*(1)^{I_\fq})\iso H^1(F_\fq^{\ur}/F_\fq,V^{I_\fq}).$
This proves i).

ii) For each $\fq\in S_p$ we have   $g_\fq([x_\fq^+])=[x_\fq].$ 
From the orthogonality of $\bD_\fq$ and $\bD_\fq^{\perp}$ it follows that
\begin{equation*}
\label{local nullity of delta (x)}
\delta_{\bD}^1(x_\fq^+)=-x_\fq^+\cup y_\fq^+=0.
\end{equation*}
Therefore, $\delta_{V,\fq}^1([x_\fq])=\delta_{\bD,\fq}^1([x_\fq^+])=à0$ for each $\fq\in S_p.$ 

iii) Let $\fq\in \Sigma_p.$ 
Since  $[x_\fq]\in H^1_f(F_\fq,V),$ from  i) it follows that again 
$\delta_{V,\fq}([x_\fq])=0.$   As the localization map
\[
H^2_S(E(1))\rightarrow \underset{v \in S}\bigoplus H^2(F_\fq, E(1))
\]
is injective, we obtain that  $\delta_{V}^1(x)=0.$

iv) First prove the surjectivity of $\pi\,:\,H^1(Y_y(1),\bD_y) \rightarrow H^1(V,\bD).$ 
For each $\fq\in \Sigma_p$ we denote by 
\[
s_{y,\fq}\,:\,
H^1_f(F_\fq,Y_y^*(1))\simeq H^1_f(F_\fq,V)
\]
the inverse of the isomorphism i). 
Let $[x]\in H^1(V,\bD).$ By ii), $\delta_V^1([x])=0,$ and there exists 
$[a]\in H^1_S(Y_y^*(1))$ such that $\pi ([a])=[x].$ Let $\fq\in \Sigma_p.$
 Since $[x_\fq^+]\in H^1_f(F_\fq,V),$ there exists $[b^+_\fq]\in H^1(F_\fq,E(1))$
 such that 
 \[
[a^+_\fq]= s_{y,\fq}([x_\fq^+])+[b_\fq^+].
\]
The localization map $H^1_S(E(1))\rightarrow \underset{\fq\in\Sigma_p}\bigoplus 
H^1(F_\fq, E(1))$ is surjective, and there exists $[b]\in H^1_S(E(1))$ such that
$[b_\fq]=\res_{\fq}([b_\fq])=[b_\fq^+]$ for each $\fq \in \Sigma_p.$ 
Then $\widehat{[x]}=[a]-[b]\in H^1(Y_y^*(1),\bD_y)$ and satisfies 
$\pi (\widehat{[x]})=[x].$ Thus, the map $\pi$ is surjective. 

Finally, from i) we have
\[
H^1(E(1),\CR (\chi)) =\ker \left ( H^1_S(E(1))
\rightarrow \underset{\fq\in \Sigma_p}\bigoplus H^1(F_\fq,E(1))\right ),
\]
and it is easy to see that $H^1(E(1),\CR (\chi))$ coincides with the kernel of $\pi.$ The lemma is proved.
\end{proof}

\subsubsection{}
Let $\log_p\,:\,\Qp^*\rightarrow \Qp$ denote the $p$-adic logarithm
normalized by $\log_p(p)=0.$
 For each finite place $\fq$ we define an homomorphism  
$\ell_\fq\,:\,F_\fq^*\rightarrow \Qp$ by 
\begin{equation}
\nonumber
\ell_\fq(x)=\begin{cases} \log_p(N_{F_{\fq}/\Qp}(x))  , &\text{\rm if $\fq\mid p$,}\\
\log_p\vert x\vert_\fq, &\text{\rm if  $\fq\nmid p,$}
\end{cases}
\end{equation}
where $N_{F_\fq/\Qp}$ denotes the norm map.  By linearity, $\ell_\fq$ can be extended to 
a map $\ell_\fq\,:\,F_\fq^*\widehat \otimes_{\Zp} E\rightarrow E,$  and the isomorphism 
$F_\fq^*\widehat \otimes_{\Zp} E\iso H^1(F_\fq,E(1))$ allows to consider $\ell_\fq$
as a map $H^1(F_\fq,E(1))\rightarrow E$ which we denote again by $\ell_\fq.$

From the product formula 
\[
\vert N_{F/\Q}(x)\vert_{\infty}\underset{\fq\in S_f}\prod \vert x\vert_\fq=1
\]
and the fact that  $N_{F/\Q}(x)=\underset{\fq\mid p}\prod N_{F_\fq/\Qp}(x)$ it follows that
\begin{equation}
\label{product formula for ell}
\underset{\fq\in S_f}\sum \ell_\fq(x)=1, \qquad \forall x\in F^*.
\end{equation}
\noindent
%\subsubsection{} 
We set $\La_{F_\fq}=\mathcal O_E[[\Gamma_v^0]]$ and  $\La_{F_\fq,E}=\La_{F_\fq}[1/p].$
\begin{mylemma}
\label{proposition universal norms}
Let  $V$ be  a $p$-adic representation of $G_{F,S}$ that satisfies {\bf N1-2)} and 
let  $[y]\in H^1(V^*(1),\bD^{\perp}).$ For each $\fq\in S_p,$ 
the following  diagram is commutative with exact rows and columns
\begin{equation}
\label{diagram universal norms}
\xymatrix{
0\ar[d] & 0 \ar[d] & \\
\mathcal{H} (\Gamma_\fq^0) \otimes_{\Lambda_{F_\fq,E}}\Hi^1(F_\fq,E(1)) \ar[r] \ar[d] & H^1(F_\fq, E(1)) \ar[d] 
\ar[r]^(.7){\ell_\fq}& E\\
\Hi^1(\bD_{\fq,y})\ar[d]^{\pi_{\bD,\fq}^{\Iw}}\ar[r]^{\pr_{\fq,y}} &H^1(\bD_{\fq,y})\ar[d]^{\pi_{\bD,\fq}}& 
\\
\Hi^1(\bD_\fq)\ar[d]\ar[r]^{\pr_\fq} &H^1(\bD_\fq) \ar[d] \ar[r]& 0 
\\
H^2(F_\fq, E(1)) \ar[r]^{=} &H^2(F_\fq, E(1))  .& 
}
\end{equation}

\end{mylemma}
\begin{proof} The exacteness of the left column is clear. The exactness 
of the right column follows from the fact that the diagram
\[
\xymatrix{
H^2(F_{\fq,n},E(1))\ar[r]^(.7){\inv_{v,n}}\ar[d]_{\mathrm{cores}}
&E\ar[d]^{\id}\\
H^2(F_{\fq,n-1},E(1))\ar[r]^(.7){\inv_{v,n-1}}\
&E
}
\]
is commutative, and therefore 
\[
\Hi^2(F_\fq,E(1))\simeq H^2(F_\fq, E(1))\simeq E.
\]
The diagram (\ref{diagram universal norms}) is clearly
commutative.   Now, we prove that the projection map $H^1_{\Iw}(\bD_\fq)\rightarrow 
H^1(\bD_\fq)$ is surjective. We have an exact sequence
\[
0\rightarrow H^1_{\Iw}(\bD_\fq)_{\Gamma_\fq^0}\rightarrow H^1(\bD_\fq)
\rightarrow H_{\Iw}^2(\bD_\fq)^{\Gamma_\fq^0}\rightarrow 0,
\]
and therefore it is enough to show that $H_{\Iw}^2(\bD_\fq)^{\Gamma_\fq^0}=0.$
Consider the exact sequence
\[
0\rightarrow H_{\Iw}^2(\bD_\fq)^{\Gamma_\fq^0}\rightarrow H_{\Iw}^2(\bD_\fq)
\xrightarrow{\g_v-1} H_{\Iw}^2(\bD_\fq)\rightarrow H_{\Iw}^2(\bD_\fq)_{\Gamma_\fq^0}
\rightarrow 0.
\]
Since  $H_{\Iw}^2(\bD_\fq)$ is a finite-dimensional $E$-vector space, we have 
\[
\dim_E H_{\Iw}^2(\bD_\fq)^{\Gamma_\fq^0}=\dim_E H_{\Iw}^2(\bD_\fq)_{\Gamma_\fq^0}=
\dim_E H^2(\bD_\fq)=\dim_EH^0(\bD_\fq^*(\chi))=0.
\]
Thus, the  map  $H^1_{\Iw}(\bD_\fq)\rightarrow 
H^1(\bD_\fq)$ is surjective.
To prove the exactness of the first row, we remark that 
the sequence 
\[
H^1_{\Iw}(F_\fq,E(1))\rightarrow H^1(F_\fq,E(1))\xrightarrow{\ell_\fq} E
\] 
is known to be exact (see, for example, \cite{Ne06}, Section 11.3.5), and that
the image of the projection 
$\CH (\Gamma_\fq^0)\otimes_{\La_{F_\fq,E}}H^1_{\Iw}(F_\fq,E(1))\rightarrow H^1(F_\fq,E(1))$
coincides with the image of the projection  $H^1_{\Iw}(F_\fq,E(1))\rightarrow H^1(F_\fq,E(1)).$
\end{proof}

\subsubsection{} By Lemma~\ref{proposition universal norms}, for each $\fq \in S_p$ we have the following commutative diagram with
exact rows, where the map $\pr_{\fq}$ is surjective 
\begin{equation}
\label{diagram construction of h^norm}
\xymatrix{
& &H^1_{\Iw}(\bD_{\fq,y}) \ar[r]^{\pi_{\bD,\fq}^\Iw} \ar[d]^{\pr_{\fq,y}} 
&H^1_{\Iw}(\bD_\fq) \ar@{->>}[d]^{\pr_{\fq}} \ar[r]^{} & H^2(F_\fq, E(1)) \ar[d]^{=}\\
0\ar[r] &H^1(F_\fq, E(1)) \ar[d]^{=} \ar[r]&H^1(\bD_{\fq,y})
\ar[r]^{\pi_{\bD,\fq}} \ar@{^{(}->}[d]^{g_{\fq,y}} &H^1(\bD_\fq) \ar[r]^{\delta_{\bD,v}^1} \ar@{^{(}->}[d]^{g_\fq} &
H^2(F_\fq,E(1))\\
0 \ar[r]& H^1(F_\fq,E(1)) \ar[r] &H^1 (F_\fq,Y_y^*(1)) \ar[r]^{\pi_\fq} 
&H^1(F_\fq,V).
%\ar[r] 
&
}
\end{equation}
Let $[x]\in H^1(V,\bD).$ By Lemma~\ref{lemma construction h^norm} ii),  
for each $\fq\in S_p$ we have $\delta_{\bD,\fq}([x_\fq^+])=0$,
and therefore there exists $[x_{\fq,y}^{\Iw}]\in H^1_{\Iw}(\bD_{\fq,y})$ such that 
$\pr_{\fq}\circ \pi_{\bD,\fq}^\Iw \left ({[x_{\fq,y}^{\Iw}]}\right )=[x_\fq^+].$
By Lemma~\ref{lemma construction h^norm} iv), there exists a  lift $\widehat{[x]}\in H^1(Y_y^*(1),\bD_y)$ of $[x].$ For each $v\in S_p$ we set
\begin{equation}
\label{definition of u_v}
[u_\fq]=g_{\fq,y}\circ \pr_{\fq,y}([x_{\fq,y}^{\Iw}])-\widehat{[x_{\fq}]},
\end{equation}
where $\widehat{x_\fq}=\res_\fq(\widehat{x}).$
 Then $\pi_\fq([u_\fq])=0,$
and therefore $[u_\fq]\in H^1(F_\fq,E(1)).$

\begin{definition} Let $V$ be a $p$-adic representation of $G_{F,S}$ 
equipped with a family $\bD=(\bD_\fq)_{\fq\in S_p}$ of  $(\Ph,\Gamma_\fq)$-modules satisfying the conditions {\bf N1-2)}. 
The $p$-adic height  pairing $h^{\norm}_{V,D}$ associated to these data is defined to be the map
\begin{equation}
\nonumber
\begin{aligned}
&h^{\norm}_{V,D}\,\,:\,\,H^1(V,\bD)\times H^1(V^*(1),\bD^{\perp})
\rightarrow E,\\
&h^{\norm}_{V,D}([x],[y])=\underset{\fq\in S_p}\sum 
\ell_\fq \left ([u_\fq]\right ).
\end{aligned}
\end{equation}
\end{definition}

\subsubsection{\bf Remarks} 1) If $\widetilde{[x]}\in H^1(Y_y^*(1),\bD_y)$ is another lift of $[x],$ then from (\ref{product formula for ell}) and the fact that 
$\widehat{[x_\fq]}=\widetilde{[x_\fq]}=s_{y,\fq}([x_\fq])$ for all $\fq\in S_p,$  it follows that the definition of $h^{\norm}_{V,D}([x],[y])$ does not depend on the choice of the lift $\widehat{[x_\fq]}.$ 

2) It is not indispensable to take $\widehat{[x]}$ in $H^1(Y_y^*(1),\bD_y).$
If $\widehat{[x]}\in H^1_S(Y_y^*(1))$ is such that $\pi (\widehat{[x]})=[x],$
we can again define $[u_\fq]$ by (\ref{definition of u_v}). 
For $\fq\in \Sigma_p$ we set
\[
[u_\fq]=g_{\fq,y}\circ s_{y,\fq}([x_\fq^+])-\widehat{[x_{\fq}]},
\]
where $s_{y,\fq}\,:\,H^1_f(F_\fq,V^*(1))\iso H^1_f(F_\fq,Y_y^*(1))$ denotes the isomorphism  from Lemma~\ref{lemma splitting for v prime to p} i).  
Note that again $[u_\fq]\in H^1(F_v,E(1)).$
Then 
%\begin{equation}
%\nonumber
%[\widehat{x_v}]=\res_v \left (\widehat{[x]}\right ), \quad 
%[u_v]= [\widehat{z_{v}^+}]-[\widehat{x_v}].
%\end{equation} 
\begin{equation}
\nonumber
h^{\norm}_{V,D}([x],[y])=\underset{\fq\in S}\sum 
\ell_\fq \left ([u_\fq]\right ).
\end{equation}

3) The map $h^{\norm}_{V,D}$ is bilinear. This can be shown directly, but follows 
from Theorem~\ref{theorem comparision sel and norm heights} below.

\subsection{Comparision with $h^{\sel}_{V,\bD}$}

\subsubsection{} In this subsection we compare $h^{\norm}_{V,D}$ with
the $p$-adic height pairing constructed in Subsection~\ref{subsection construction of h^sel}. We take $\Sigma=\emptyset$ and denote by
\[
h^{\sel}_{V,\bD,1}\,:\,H^1(V,\bD)\times H^1(V^*(1),\bD^{\perp}) \rightarrow E
\]
the associated height pairing (\ref{pairing sel for H^1}).

\begin{mytheorem}
\label{theorem comparision sel and norm heights}
Let  $V$ be a $p$-adic representation 
of $G_{F,S}$ with coefficients in a finite extension $E$ of $\Qp.$  Assume that  the family $\bD=(\bD_\fq)_{\fq\in S_p}$ satisfies the conditions {\bf N1-2)}.  Then $h^{\norm}_{V,D}$ is a bilinear map 
and 
\[
h^{\norm}_{V,D}=h^{\sel}_{V,\bD,1}.
\]
\end{mytheorem}
\begin{proof}

The proof repeats the arguments of \cite{Ne06}, Sections 11.3.9-11.3.12, where this statement is proved in the case of $p$-adic height
pairings arising from Greenberg's local conditions. We remark that 
in this case our definition of $h^{\norm}_{V,D}$ differs 
from Nekov\'a\v r's $h_{\pi}^{\norm}$ by a sign. 

Let  $[x]\in H^1(V,\bD)$ and $[y]\in H^1(V^*(1),\bD^{\perp}).$ 
We use the notation of Section~\ref{subsection Selmer complexes} and denote 
by $f_\fq$ and $g_\fq$ the morphisms defined by 
(\ref{morphism g_v}--\ref{morphism f_v v mid p}).
As before, to simplify notation we set $x_\fq=f_\fq (x)$ and $y_\fq=f_\fq^{\perp}(y).$
We represent $[x]$ and $[y]$ by cocycles $x^\sel=(x, (x_\fq^+), (\lambda_\fq))
\in S^1(V,\bD)$
and $y^\sel=(y, (y_\fq^+), (\mu_\fq))\in  S^1(V^*(1),\bD^{\perp}),$ where 
\begin{align}
\nonumber 
&x\in C^1(G_{F,S},V), 
&&x_\fq^+\in U_\fq^1(V,\bD),
&&\lambda_\fq\in K^0(V_\fq),\\
\nonumber
&y\in C^1(G_{F,S},V^*(1)), 
&&y_\fq^+\in U_\fq^1(V^*(1),\bD^{\perp}),
&&\mu_\fq\in K^0(V_\fq^*(1))
\end{align}
and 
\begin{align}
\nonumber 
&dx=0, &&dy=0,\\
\nonumber
&dx_\fq^+=0 &&dy_\fq^+=0,\\
\nonumber
&g_\fq(x_\fq^+)=f_\fq(x)+d\lambda_\fq,
&&g_\fq^{\perp}(y_\fq^+)=f_\fq^{\perp}(x)+d\mu_\fq,\qquad \forall \fq\in S.
\end{align}

By Propositions~\ref{formula for beta_c}, \ref{formula for beta_D} and \ref{formula for beta_K}
we have  
\begin{equation}
\label{formula for beta_V,D}
\beta_{V,\bD}(x^\sel)= (-z\cup x, (-w_\fq\cup x_\fq^+), (z_\fq\cup \lambda_\fq))\in S^2(V,\bD),
\end{equation}
where  
\begin{align}
&z= \log \chi \in C^1(G_{F,S},E(0)), \nonumber\\
&w_\fq=\begin{cases} 0, &\text{\rm if $\fq\in \Sigma_p$},\\
(0, \log \chi_\fq(\g_\fq))\in C^1_{\Ph,\g_\fq}(E(0)), &\text{\rm if $\fq\in S_p$},
\end{cases}
\nonumber
\\
&z_\fq=\begin{cases}
\log \chi_\fq  \in C^1(G_{F_\fq},E(0)),
&\text{\rm if $\fq\in \Sigma_p$},\\
(0,\log \chi_\fq)\in K^1(E(0)_\fq),
&\text{\rm if $\fq\in S_p$}.
\end{cases}
\nonumber
\end{align}

\begin{mylemma} 
\label{first lemma comparision sel and norm heights}
Let $\widehat{[x]}\in H^1(Y_y^*(1),\bD_y)$ be a lift 
of an element $[x]\in H^1(V,\bD).$ Represent $\widehat{[x]}$ and 
$[\widehat x_\fq^+]\in H^1(\bD_{\fq,y})$ for $\fq\in S_p$
(respectively $[\widehat x_\fq^+]\in H^1_f(F_\fq,Y_y^*(1))$ for $\fq \in \Sigma_p$)  by cocycles 
$\widehat x\in C^1(G_{F,S},Y_y^*(1))$ and 
$\widehat x_\fq^+\in C^1_{\Ph,\g_\fq}(\bD_{\fq,y})$
(respectively by $\widehat x_\fq^+\in C^1_{\ur}(Y_y^*(1)_\fq)$).
 Since 
$g_\fq(\widehat x_\fq^+)=f_\fq(\widehat x)+d\widehat\lambda_\fq$ for some 
$\widehat\lambda_\fq\in K^0(Y_y^*(1)_\fq),$ we obtain a cocycle 
$\widehat x^{\,\sel}=(\widehat x,(\widehat x_\fq^+),(\widehat \lambda_\fq))\in S^1(Y_y^*(1),\bD_y).$
Then
\[
\beta_{Y_y^*(1),\bD_y}(\widehat x^{\,\sel})=(\widehat a,(\widehat b_\fq),(\widehat c_\fq))
\in S^2(Y_y^*(1),\bD_y),
\]
where 
\begin{equation}
\nonumber
\widehat b_\fq= 
\begin{cases} 0, &\text{ if $\fq\in \Sigma_p$},\\
w_\fq\cup u_\fq \in  C^2_{\Ph,\g_\fq}(E(1)_\fq),
&\text{ if $\fq\in S_p$} 
\end{cases}
\end{equation}
and $u_\fq$ is a representative of the cohomology class 
(\ref{definition of u_v}).
\end{mylemma}
\begin{proof} The proof is analogous to the proof of Lemma 11.3.10 of \cite{Ne06}. Since the complex $C^{\bullet}_{\ur}(Y_y^*(1)_v)$ is concentrated in degrees $0$ and $1,$ it is 
clear that   $\widehat b_\fq=0$ for $\fq\in \Sigma_p.$ 
For each   $\fq\in S_p,$
we consider  the algebra $\CH_E(\Gamma_F^0)=\{f(\g_F-1)\mid f(X)\in \CH_E\}$ and define $\Ind_{F_\infty/F}(\bD_{\fq})=\bD_\fq\widehat\otimes_E\CH_E(\Gamma_F^0)^{\iota}.$  The natural inclusion 
$\CH_E(\Gamma_\fq^0)\subset \CH_E(\Gamma_F^0)$ allows us to consider 
$\Ind_{F_\infty/F}(\bD_{\fq})$ as a $(\Ph,\Gamma_\fq)$-module over the ring 
$\CR_{F_\fq,\CH_E}=\underset{r}\varinjlim \mathbf{B}_{\mathrm{rig},F_\fq}^{\dagger ,r}\widehat{\otimes} \CH_E(\Gamma_\fq^0).$
Let $J_{\CH}$ denote the augmentation ideal of $\CH_E(\Gamma_F^0).$ 
The obvious commutative diagram
\begin{equation*}
\xymatrix{
\CH_E(\Gamma_F^0) \ar[rr]^{\mod{\Gamma_\fq^0}} \ar[d]^{\mod{J_{\CH}^2}}
& &E[\Gamma_F^0/\Gamma_\fq^0] \ar[d]\\
\widetilde A_{F} \ar[rr] & &E,
}
\end{equation*}
where the right vertical morphism is the augmentation map, induces 
a commutative diagram
\begin{equation}
\label{diagram of first  lemma comparision} 
\xymatrix{
H^1(\Ind_{F_\infty/F}(\bD_{\fq})) \ar[d] \ar[r] & H^1\left (\bD_\fq\otimes_E E[\Gamma_F^0/\Gamma_\fq^0]\right ) \ar[d]\\
H^1(\widetilde\bD_{F,\fq}) \ar[r] &H^1(\bD_\fq).
}
\end{equation}
Since $\Gamma_\fq^0$ acts trivially of $E[\Gamma_F^0/\Gamma_\fq^0],$ 
the group $H^1\left (\bD_\fq\otimes_E E[\Gamma_F^0/\Gamma_\fq^0]\right )$
is isomorphic to the direct sum of $(\Gamma_F^0:\Gamma_\fq^0)$ copies 
of $H^1 (\bD_\fq)$ and the right vertical map of 
(\ref{diagram of first  lemma comparision}) is surjective. 
Since $\CH_E(\Gamma_F^0)$ is a free $\CH_E(\Gamma_\fq^0)$-module of rank 
$[\Gamma_F^0:\Gamma_\fq^0],$ the group $H^1(\Ind_{F_\infty/F}(\bD_{\fq}))$
is isomorphic to the direct sum of $[\Gamma_F^0:\Gamma_\fq^0]$-copies of $H^1_{\Iw}(\bD_\fq).$ Since the projection $\pr_\fq\,:\,H^1_{\Iw}(\bD_\fq)
\rightarrow H^1(\bD_\fq)$ is surjective by Proposition~\ref{proposition universal norms}, the upper horizontal map of (\ref{diagram of first  lemma comparision}) is also surjective. Therefore the projection 
$H^1(\Ind_{F_\infty/F}(\bD_\fq))\rightarrow H^1(\bD_\fq)$ is surjective 
and we have the following analog of the diagram (\ref{diagram construction of h^norm}) 
\begin{equation}
\xymatrix{
& &H^1 (\Ind_{F_\infty/F}(\bD_{\fq,y})) \ar[r]^{\pi_{\bD,\fq}^{\Iw,F}} \ar[d]^{\pr_{\fq,y}^F} 
&H^1(\Ind_{F_\infty/F}(\bD_\fq)) \ar@{->>}[d]^{\pr_{\fq}^F} \ar[r] & 0\\
0\ar[r] &H^1(F_\fq, E(1)) \ar[d]^{=} \ar[r]&H^1(\bD_{\fq,y})
\ar[r]^{\pi_{\bD,\fq}} \ar@{^{(}->}[d]^{g_{\fq,y}} &H^1(\bD_\fq) \ar[r] \ar@{^{(}->}[d]^{g_\fq} &0\\
0 \ar[r]& H^1(F_\fq,E(1)) \ar[r] &H^1 (F_\fq,Y_y^*(1)) \ar[r]^{\pi_\fq} 
&H^1(F_\fq,V)
\ar[r] &0.
}
\end{equation}
From the above discussion it follows that 
$C^1_{\Ph,\g_\fq}(\Ind_{F_\infty/F}(\bD_{\fq,y}))$ is isomorphic to the direct sum of $[\Gamma_F^0:\Gamma_\fq^0]$ copies of $C^1_{\Ph,\g_\fq}(\Ind_{F_{\fq,\infty}/F_\fq}(\bD_{\fq,y})).$ Therefore there exists a cocycle 
$\widetilde x_{\fq,y}\in C^1_{\Ph,\g_\fq}(\Ind_{F_\infty/F}(\bD_{\fq,y}))$  such that 
\linebreak 
$\pr_{\fq,y}^F(\widetilde x_{\fq,y})=\pr_{\fq,y}(x_{\fq,y}^{\Iw}).$
Since the map 
\[
C^{\bullet}_{\Ph,\g_\fq}(\Ind_{F_\infty/F}(\bD_{\fq,y})) \rightarrow 
C^{\bullet}_{\Ph,\g_\fq}(\bD_{\fq,y})
\]
factors through $C^{\bullet}_{\Ph,\g_\fq}(\widetilde \bD_{F,\fq,y}),$
where $\widetilde \bD_{F,\fq,y}=\bD_{\fq,y}\otimes \widetilde A_F^{\iota},$ 
from the distinguished triangle 
\[
C^{\bullet}_{\Ph,\g_\fq}(\bD_{\fq,y}) \rightarrow 
 C^{\bullet}_{\Ph,\g_\fq}(\widetilde \bD_{F,\fq,y}) 
\rightarrow
 C^{\bullet}_{\Ph,\g_\fq}(\bD_{\fq,y}) 
\xrightarrow{\beta_{\bD_{\fq,y}}} C_{\Ph,\g_\fq}(\bD_{\fq,y}) [1]
\]
it follows that $\beta_{\bD_{\fq,y}}(\pr_{\fq,y}^F(\widetilde x_{\fq,y}))=0.$
Thus,
\begin{multline*}
\widehat b_\fq=\beta_{\bD_{\fq,y}}(\widehat x_\fq^+)=
\beta_{\bD_{\fq,y}}(\widehat x_\fq^+ -\pr_{\fq,y}(x_{\fq,y}^{\Iw})+
\pr_{\fq,y}^F(\widetilde x_{\fq,y}) )=\\
=-\beta_{\bD_{\fq,y}}(u_\fq)+ \beta_{\bD_{\fq,y}}(\pr_{\fq,y}^F(\widetilde x_{\fq,y}))=
-\beta_{\bD_{\fq},y}(u_\fq)=w_\fq\cup u_\fq.
\end{multline*}
The lemma is proved.
\end{proof}

For each $\fq\in S_p,$ we have the canonical isomorphism of local class  field
theory
\[
\inv_\fq \,:\, H^2(F_\fq, E(1)) \iso E.
\]
Let 
$\kappa_\fq\,:\, F_\fq^*\widehat \otimes E \rightarrow H^1(F_\fq,E(1))$
denote the Kummer map. Then 
\begin{equation}
\inv_\fq(\log\chi_\fq\cup \kappa_\fq (x))=\log_p (N_{F_\fq/\Qp}(x))=\ell_\fq(x)
\end{equation}
(\cite{CL}, Chapitre 14, see also \cite{Ben97}, Corollaire 1.1.3).
\begin{mylemma}
\label{second lemma comparision sel and norm heights}
Assume that $\beta_{V,\bD}([x^\sel])\in H^2(V,\bD)$ 
is represented by a $2$-cocycle $e=(a, (b_\fq), (c_\fq))$ of the form
$e=\pi (\widehat e),$ where  
\[\widehat e=(\widehat a,(\widehat b_\fq),
(\widehat c_\fq))\in S^2(Y_y^*(1),\bD_y)
\] 
is also a $2$-cocycle and $\pi \,:\, S^2(Y_y^*(1),\bD_y) \rightarrow 
 S^2(V,\bD)$ denotes the canonical projection.  Then
\[
[\beta_{V,\bD}(x^\sel)\cup y^\sel]=
\underset{\fq\in S_p}\sum \inv_\fq (g_\fq(\widehat b_\fq)\cup  f_\fq^{\perp}(\alpha_y)
+g_\fq (b_\fq)\cup \mu_\fq),
\]
where $\alpha_y\in C^0(G_{F,S}, Y_y)$ is an element that maps to $1\in C^0(G_{F,S},E)=E$ and satisfies
$d\alpha_y=y.$
If, in addition,
\[
\widehat b_\fq\in C^2_{\Ph, \g_\fq}(E(1)_\fq),\qquad  \forall \fq\in S_p,
\]
then 
\[
[\beta_{V,\bD}(x^\sel)\cup y^\sel]=
\underset{\fq\in S_p}\sum \inv_\fq (g_\fq(\widehat b_\fq)).
\]
\end{mylemma}
\begin{proof} The proof of this lemma is purely formal and follows 
{\it verbatim} the proof of \cite{Ne06}, Lemma 11.3.11.
\end{proof}

Now we can proof Theorem~\ref{theorem comparision sel and norm heights}.
Combining Lemma~\ref{first lemma comparision sel and norm heights} and 
Lemma~\ref{second lemma comparision sel and norm heights} we have 
\begin{multline}
\nonumber 
h^{\sel}_{V,\bD,1}([x],[y])= [\beta_{V,\bD}(x^\sel)\cup y^\sel]=
\underset{\fq\in S_p}\sum \inv_\fq (g_\fq(\widehat b_\fq))=\\
=\underset{\fq\in S_p}\sum \inv_\fq (g_\fq (w_\fq\cup u_\fq))=
\underset{\fq\in S_p}\sum \ell_\fq (u_\fq)=h^{\norm}_{V,D}([x],[y]). 
\end{multline} 
\end{proof}

\section{$p$-adic height pairings III:   splitting of local extensions}
\label{section splitting}

\subsection{The pairing $h^{\spl}_{V,D}$} 
\label{subsection pairing h^{spl}}
\subsubsection{}
Let $F$ be a finite extension of $\Q.$ We keep notation
of Sections~\ref{section Selmer}-\ref{section universal norms}. In particular, we  fix a finite set $S$ of places of $F$ such that $S_p\subset S$ and denote by $G_{F,S}$ the Galois group
of the maximal algebraic extension of $F$ which is unramified outside
$S\cup S_\infty.$ 
For each topological $G_{F,S}$-module $M,$ we write $H^*_S(M)$ for the continuous cohomology of $G_{F,S}$ with coefficients in $M.$

 Let $V$ be a $p$-adic 
representation of $G_{F,S}$ with coefficients in a finite extension $E/\Qp$ which is potentially semistable at all $\fq\mid p.$ 
Following Bloch and Kato, for each $\fq\in S$ we define the 
subgroup  $H^1_f(F_\fq,V)$ of $H^1(F_\fq,V)$ by
\begin{equation}
H^1_f(F_\fq,V)=\begin{cases} \ker (H^1(F_\fq,V)\rightarrow H^1(F_\fq,V\otimes_{\Bbb Q_p}\Bc)) &\text{\rm if $\fq\mid p,$}\\
\ker (H^1(F_\fq,V)\rightarrow H^1(F_\fq^{\ur},V)) &\text{\rm if $\fq\nmid p$.}
\end{cases}
\nonumber 
\end{equation}
The Bloch--Kato Selmer group of $V$ is defined as 
\begin{equation}
H^1_f(V)= \ker \left ( H^1_S(V)\rightarrow \underset{\fq\in S}\bigoplus \frac{H^1(F_\fq,V)}{H^1_f(F_\fq,V)}\right ).
\nonumber 
\end{equation}
In this section, we assume that, for all $\fq\in S_p,$ 
the representation  $V_\fq$ satisfies the condition {\bf S)}
of Section~\ref{section splitting submodules}, namely that 
\newline
\newline
{\bf S)}  $\Dc (V_\fq)^{\Ph=1}=\Dc (V_\fq^*(1))^{\Ph=1}=0$ for all 
$\fq\in S_p.$
\newline
\newline
For each $\fq\mid p$, we fix a splitting $(\Ph, N, G_{F_\fq})$-submodule $D_\fq$ of $\Dpst(V_\fq)$ (see Section~\ref{subsection splittings submodules}).  
We will associate to these data a pairing
$$
h^{\spl}_{V,D}\,: H^1_f(V)\times H^1_f(V^*(1)) 
\rightarrow  E
$$ 
and compare it  with the height pairing constructed in Section 4 of \cite{Ne92} using
the exponential map and splitting of the Hodge filtration.

Let $[y]\in H^1_f(V^*(1)).$ Fix a representative $y\in C^1(G_{F,S},V^*(1))$ 
of $y$ and consider the corresponding extension of Galois representations
\begin{equation}
\label{extension Y_y}
0\rightarrow V^*(1)\rightarrow Y_y\rightarrow E \rightarrow 0.
\end{equation}
Passing to duals, we obtain an extension
\begin{equation}
\nonumber
0\rightarrow E(1) \rightarrow Y_y^*(1)\rightarrow V \rightarrow 0.
\end{equation}
From {\bf S)} it follows that $H^0_S(V)=0,$ and the associated long exact sequence of 
global Galois cohomology reads
\begin{equation}
\nonumber
0\rightarrow H^1_S(E(1))\rightarrow H^1_S(Y_y^*(1))
\rightarrow H^1_S(V)\xrightarrow{\delta_V^1} H^2_S(E(1))
\rightarrow \ldots .
\end{equation}
Also, for each place $\fq\in S$ we have the long exact sequence of local
Galois cohomology
\begin{multline}
\nonumber
H^0(F_\fq,V)\rightarrow H^1 ( F_\fq, E(1))
\rightarrow  H^1( F_\fq,Y_y^*(1))\rightarrow \\
\rightarrow H^1( F_\fq,V)
\xrightarrow{\delta_{V,\fq}^1} 
H^2( F_\fq, E(1)) \rightarrow \ldots .
\end{multline}
The following results, which can be seen as an analog of Lemma~\ref{lemma construction h^norm}, are well known  but we recall them for the reader's  
convenience. 

\begin{mylemma} 
\label{lemma splitting for v prime to p} Let $V$ be a $p$-adic representation 
of $G_{F,S}$ that is potentially semistable at all $\fq \in S_p$ and 
satisfies the condition  {\bf S)}.  
Assume that $[y]\in H^1_f(V^*(1)).$ Then

i)  $\delta_V^1 ([x])=0$ for all $x\in H^1_f(V);$

ii) There exists an exact sequence
\begin{equation}
\nonumber
0\rightarrow H^1_f (E(1))\rightarrow  H^1_f(Y_y^*(1))
\rightarrow H^1_f(V)\rightarrow 0.
\end{equation}

\end{mylemma}
\begin{proof} i) For any $x\in C^1(G_{F,S},V),$
let $x_\fq=\res_\fq (x)\in C^1(G_{F_\fq},V)$ denote the localization of $x$ at $\fq$. 
If  $[x]\in H^1_f(V),$ then for each $\fq$ one has
$\delta_{V,\fq}^1 ([x_\fq])=-[x_\fq]\cup [y_\fq]=0$ because $H^1_f(F_\fq,V)$ and $H^1_f(F_\fq,V^*(1))$
are orthogonal to each other under the cup product. Since the  map
$H^2_S(E(1)) \rightarrow \underset{\fq\in S}\bigoplus H^2( F_\fq,E(1))$
is injective and the localization commutes with cup products, this shows that $\delta_V^1([x])=0.$

ii) This is a particular case of \cite{FP94}, Proposition II, 2.2.3.
\end{proof}

\subsubsection{} 
\label{subsubsection definition of h^spl}
Let $[x]\in H^1_f(V)$ and $[y]\in H^1_f(V^*(1)).$ 
In Section~\ref{subsection canonical splitting},  for each $\fq\in S_p$
we constructed the canonical splitting (\ref{splitting extension of cohomology at v dividing p}) which sits in the diagram 
\begin{equation*}
\nonumber
\xymatrix{
0 \ar[r] &H^1_f(F_\fq, E(1)) \ar[r] \ar[d]^{=} &H^1_f(\bD_{\fq,y})\ar[r] \ar[d]^{\simeq}_{g_{\fq,y}} &H^1_f(\bD_\fq)\ar@<-1ex>[l]_{s_{y,\fq}} \ar[r] \ar[d]^{\simeq}_{g_\fq} &0\\
0 \ar[r] &H^1_f(F_\fq, E(1)) \ar[r] &H^1_f(F_\fq,Y_y^*(1))\ar[r] &H^1_f(F_\fq,V) \ar[r] &0.
}
\end{equation*}
By Lemma~\ref{lemma splitting for v prime to p} ii), we can lift $[x]\in H^1_f(V)$ to an element   $\widehat{[x]}\in H^1_f(Y_y^*(1)).$ 
Let $\widehat{[x_\fq]}=\res_\fq\left (\widehat{[x]}\right )\in H^1_f(F_\fq,Y_y^*(1)).$ 
If $\fq\in S_p,$ we denote by $[\widetilde x_\fq^+]$ the unique element of $H^1_f(\bD_\fq)$ such that
$g_\fq([\widetilde x_\fq^+])=[x_\fq].$ 

\begin{definition} The $p$-adic height  pairing associated to 
splitting submodules $D=(D_\fq)_{\fq\in S}$ is defined to be the map
\[
h_{V,D}^{\spl}\,\,:\,\, H^1_f(V)\times H^1_f(V^*(1)) 
\rightarrow E
\]
given by
\[
h_{V,D}^{\spl}([x],[y])=\underset{\fq\in S_p}\sum \ell_\fq 
\left (g_{\fq,y}\circ s_{y,\fq}([\widetilde x_\fq^+])-\widehat{[x_\fq]}\right  ).
\]
\end{definition}

\subsubsection{\bf Remarks} 1) For each $\fq\in \Sigma_p,$ denote by 
$s_{y,\fq}\,:\,H^1_f(F_\fq,V)\iso H^1_f(F_\fq,Y^*_y(1))$ the isomorphism constructed 
in Lemma~\ref{lemma construction h^norm}, i)
and by $g_\fq\,:\,H^1_f(F_\fq,V) \hookrightarrow H^1(F_\fq,V)$ and 
$g_{\fq,y}\,:\,H^1_f(F_\fq,Y^*_y(1)) \hookrightarrow 
H^1(F_\fq,Y^*_y(1))$ the canonical embeddings. Let $[\widetilde x_\fq^+]\in H^1_f(F_\fq,V)$
be the unique element such that $g_\fq ([x_\fq^+])=[x_\fq].$ 
From the product formula (\ref{product formula for ell}) it follows, that  $h_{V,D}^{\spl}$ can be defined  by 
\[
h_{V,D}^{\spl}([x],[y])=\underset{\fq\in S}\sum \ell_\fq \left (g_{\fq,y}\circ s_{y,\fq}([\widetilde x_\fq^+]) -\widehat{[x_\fq]}\right  ),
\] 
where $\widehat{[x]}\in H^1_S(V)$ is an arbitrary lift of $[x].$ 

2) The pairing $h_{V,D}^{\spl}$ is a bilinear skew-symmetric map.
This can be shown directly, but follows from the interpretation
of $h_{V,D}^{\spl}$ in terms of Nekov\'a\v r's height pairing 
(see Proposition~\ref{proposition comparision heights spl and Hodge}
below). 

\subsection{Comparision with Nekov\' a\v r's height pairing} 
\subsubsection{}
We relate the pairing $ h_{V,D}^{\spl}$ to the $p$-adic height pairing 
constructed by Nekov\'a\v r in \cite{Ne92}, Section 4. 
First recall Nekov\'a\v r's construction. 
If $[y]\in H^1_f(V^*(1)),$ the  extension (\ref{extension Y_y}) 
is crystalline at all $\fq\in S_p,$ and therefore the sequence 
\begin{equation}
\nonumber
0\rightarrow \Dc (V_\fq^*(1))\rightarrow \Dc (Y_{y,\fq})
\rightarrow \Dc (E(0)_\fq)\rightarrow 0
\end{equation}
is exact.  Since $\Dc (V_\fq^*(1))^{\Ph=1}=0,$ we have an isomorphism
of vector spaces
\[
\Dc (E(0)_\fq) \iso \Dc (Y_{y,\fq})^{\Ph=1},
\]
which can be extended by linearity to a map
$\bD_{\dR}(E(0)_\fq) \rightarrow \bD_{\dR}(Y_{y,\fq}).$
Passing to duals, we obtain a $F_\fq$-linear map 
$\bD_{\dR}(Y_{y,\fq}^*(1))\rightarrow \bD_{\dR}(E(1)_\fq)$
which defines a splitting $s_{\dR,\fq}$ of the exact sequence
\begin{equation}
\nonumber
\xymatrix{
0\ar[r] &\bD_{\dR}(E(1)_\fq)\ar[r] 
&\bD_{\dR}(Y_{y,\fq}^*(1)) \ar[r]
&\bD_{\dR}(V_\fq)
\ar@<-1ex>[l]_{s_{\dR,\fq}} \ar[r] &0.
}
\end{equation}
Fix a splitting  $w_\fq\,:\,\bD_{\dR}(V_\fq)/\F^0\bD_{\dR}(V_\fq)\rightarrow \bD_{\dR}(V_\fq)$ 
of the canonical projection  
\begin{equation}
\label{projection on the tangent space}
\pr_{\dR,V_\fq}\,:\,\bD_{\dR}(V_\fq)
\rightarrow \bD_{\dR}(V_\fq)/\F^0\bD_{\dR}(V_\fq).
\end{equation} 
We have a commutative diagram
\begin{equation}
\nonumber
\xymatrix{
0\ar[r] &H^1(F_\fq,E(1)) \ar[r] &H^1_f(F_\fq,Y^*_y(1)) \ar[r] &H^1_f(F_\fq,V)\ar[r] \ar@<-1ex>[l]_{s_{y,\fq}^w}&0 \\
& &\displaystyle\frac{\bD_{\dR}(Y_{y,\fq}^*(1))}{\F^0\bD_{\dR}(Y_{y,\fq}^*(1))}
\ar[u]^{\exp_{Y_{y,\fq}^*(1)}}_{\simeq}\ar[r]
&\displaystyle\frac{\bD_{\dR}(V_\fq)}{\F^0\bD_{\dR}(V_\fq)} \ar[u]^{\exp_{V_\fq}}_{\simeq}
\ar[d]^(.6){w_\fq} &\\
& &\bD_{\dR}(Y_{y,\fq}^*(1)) \ar[u]^(.4){\pr_{\dR,V^*_{y,\fq}(1)}}
&\bD_{\dR}(V_\fq). \ar[l]_{s_{\dR,\fq}}
&  
}
\end{equation}
Then the map 
$
s_{y,\fq}^w\,:\, H^1_f(F_\fq,V)\rightarrow H^1_f(F_\fq,Y_y^*(1))
$
defined  by 
\[
s_{y,\fq}^w=\exp_{Y_{y,\fq}^*(1)}\circ \pr_{\dR,Y^*_{y,\fq}(1)} \circ s_{\dR,\fq} \circ
w_\fq \circ \exp_{V_\fq}^{-1}
\] 
gives a splitting of the top row of the diagram, which depends only 
on the choice of $w_\fq$ and $[y].$ 

\begin{definition}[\sc Nekov\' a\v r] The $p$-adic height pairing associated
to  a family $w=(w_\fq)_{\fq\in S_p}$ of  splitting $w_\fq$ of the projections (\ref{projection on the tangent space})
is defined to be the map
\[
h_{V,w}^{\Hodge}\,\,:\,\,H^1_f(V)\times H^1_f(V^*(1)) \rightarrow E
\]
given by
\[
h_{V,w}^{\Hodge}([x],[y])=\underset{\fq\mid p}\sum \ell_\fq \left (s_{y,\fq}^w([x_\fq])-\widehat{[x_\fq]}\right  ),
\]
where $\widehat{[x]}\in H^1_f(Y_y^*(1))$ is a lift of $[x]\in H^1_f(V)$
and  $\widehat{[x_\fq]}$ denotes its localization at $\fq.$
\end{definition}

In \cite{Ne92}, it is proved that $h_{V,w}^{\Hodge}$ is a $E$-bilinear 
map.  

\subsubsection{} Now, let $D_\fq$ be a  splitting submodule of $\DstL (V_\fq).$ We have  
\begin{equation}
\label{definition of splitting associated to D_v}
\DdrL (V_\fq)=D_{\fq,L} \oplus \F^0\DdrL (V_\fq),
\qquad D_{\fq,L}=D_\fq\otimes_{L_0}L.
\end{equation}  
Set $D_{\fq,F_\fq}= (D_{\fq,L})^{G_{F_\fq}}.$ Since the decomposition 
(\ref{definition of splitting associated to D_v}) is compatible 
with the  Galois action, taking galois  invariants 
we have 
\[
\bD_{\dR}(V_\fq)=D_{\fq,F_\fq} \oplus \F^0\bD_{\dR}(V_\fq).
\]
This decomposition defines a splitting of the projection (\ref{projection on the tangent space})
which we will denote by $w_{D,\fq}.$ 
 
\begin{myproposition} 
\label{proposition comparision heights spl and Hodge}
Let $V$ be a $p$-adic representation of $G_{F,S}$ such that
for each $\fq\in S_p$ the restriction of $V$ on the decomposition group at $\fq$  is potentially semistable and satisfies the condition {\bf S)}. 
Let $(D_\fq)_{\fq\in S_p}$ be a family of splitting submodules and let $w_D=(w_{D,\fq})_{\fq\in S_p}$ be the associated system
of splittings. Then 
\[
h_{V,D}^{\spl}=h_{V,w_D}^{\Hodge}.
\] 
\end{myproposition}

We need the following auxiliary result. As before, we denote by $\bD_\fq$
the $(\Ph,\Gamma_\fq)$-module associated to $D_\fq.$

\begin{mylemma}
\label{lemma for comparision height spl and Hodge}
The following diagram
\begin{equation}
\nonumber
\xymatrix{
\CDdr (\bD_\fq) \ar[r]^{s_{\bD_\fq,y}} \ar[d] &\CDdr (\bD_{\fq,y})\ar[d] \\
\bD_{\dR}(V_\fq) \ar[r]^(.4){s_{\dR,\fq}} &\bD_{\dR}(Y_{y,\fq}^*(1)),
}
\end{equation}
where the vertical maps are induced by the canonical inclusions
of corresponding $(\Ph,\Gamma_\fq)$-modules and $s_{\bD_\fq,y}$
is the map induced by the  splitting (\ref{splitting extension of ph-Gamma-modules}), is commutative.
\end{mylemma}
\begin{proof}[Proof of the lemma] The proof is an easy exercice and is omitted here.
\end{proof}

\begin{proof}[Proof  of Proposition~\ref{proposition comparision heights spl and Hodge}]
From the functoriality of the exponential map and Proposition~\ref{proposition properties of splitting submodules} it follows that the diagram
\begin{equation}
\label{diagram 1 proof  comparision spl and Hodge}
\xymatrix{
\CDdr (\bD_\fq) \ar[r]^{\exp_{\bD_\fq}} \ar[d]^{=} &H^1_f(\bD_\fq) \ar[d]^{=}\\
\bD_{\dR}(V_\fq)/\F^0\bD_{\dR}(V_\fq) \ar[r]^(.6){\exp_{V_\fq}} & H^1_f(F_\fq,V)
}
\end{equation}
is commutative. The same holds if we replace $V_\fq$ and $\bD_\fq$ by
$Y_{y,\fq}^*(1)$ and  $\bD_{\fq,y}$ respectively. Consider the 
 diagram
\begin{equation}
\label{diagram 2 proof  comparision spl and Hodge}
\xymatrix{
\CDdr (\bD_\fq) \ar[r]^{s_{\bD_\fq,y}} \ar[d] &\CDdr (\bD_{\fq,y}) \\
\displaystyle\frac{\bD_{\dR}(V_\fq)}{\F^0\bD_{\dR}(V_\fq)} \ar[d]^(.6){w_{D,\fq}}  
&\displaystyle\frac{\bD_{\dR}(Y_{y,\fq}^*(1))}{\F^0\bD_{\dR}(Y_{y,\fq}^*(1))}, \ar[u]^{=}
\\
\bD_{\dR}(V_\fq) \ar[r]^{s_{\dR,\fq}} &\bD_{\dR}(Y_{y,\fq}^*(1)). \ar[u]_(.4){\pr_{\dR,Y^*_{y,\fq}(1)}}
}
\end{equation}
From the definition of $w_{D,\fq},$ 
it follows that the composition of vertical maps in the left (resp. right) colomn  is induced by the inclusion $\bD_\fq\subset \Ddagrig (V_\fq)$
(resp. by  $\bD_{\fq,y}\subset \Ddagrig (Y_{y,\fq}^*(1))$) and therefore 
the diagram 
(\ref{diagram 2 proof  comparision spl and Hodge}) 
is commutative by Lemma~\ref{lemma for comparision height spl and Hodge}. From the commutativity 
of (\ref{diagram 1 proof  comparision spl and Hodge}) and 
(\ref{diagram 2 proof  comparision spl and Hodge}) and the definition
of $s_{y,\fq}$ and $s_{y,\fq}^{w},$ it follows now that 
$s_{y,\fq}=s_{y,\fq}^{w}$ for all $\fq\in S_p,$ and the proposition
is proved.
\end{proof}

  \subsection{Comparision with $h^{\norm}_{V,D}$}
\subsubsection{} In this section, we compare the pairing $h^{\spl}_{V,D}$ with the pairing 
$h^{\norm}_{V,D}$ constructed in Section~\ref{section universal norms}.
Let $V$ be a $p$-adic representation of $G_{F,S}$ that is potentially 
semistable at all $\fq \in S_p.$ Fix a system $(D_\fq)_{\fq \in S_p}$ 
of splitting submodules and denote by  $(\bD_\fq)_{\fq\in S_p}$ 
the system of $(\Ph,\Gamma_\fq)$-submodules of $\Ddagrig (V_\fq)$
associated to $(D_\fq)_{\fq \in S_p}$  by Theorem~\ref{berger theorem2}. 
We will assume, that $(V,D)$ satisfies the condition 
{\bf S)} of Section~\ref{subsection pairing h^{spl}} and the condition 
 {\bf N2)} of Section~\ref{subsection pairing h-norm}. 
 Note, that {\bf S)} implies {\bf N1)}.
We also remark, that from Proposition~\ref{proposition properties H^1_f} i) and the fact that
the Hodge--Tate weights of $\DstL (V_\fq)/D_\fq$ and 
$\DstL (V_\fq^*(1))/D_\fq^{\perp}$ are positive, 
it follows that,  under our assumptions, {\bf N2)} is equivalent to the following condition
\newline
\,

{\bf N2')} For each $\fq\in S_p,$
\[
(\DstL (V_\fq)/D_\fq)^{\Ph=1, N=0,G_{L/F_\fq}}=
(\DstL (V_\fq^*(1))/D_\fq^{\perp})^{\Ph=1, N=0,G_{L/F_\fq}}=0,
\] 
where $L$ is a finite extension of $F_\fq$ such that  $V_\fq$
(respectively $V_\fq^*(1)$) is semistable over $L.$ 

\subsubsection{} The following statement is known (see \cite{Po13} and 
\cite{Ben14}), but we prove it here for completeness.  

\begin{myproposition} Assume that $V$ is a $p$-adic representation satisfying the conditions {\bf S)} and {\bf N2)}. Then 

i) $H^1_f(F_\fq,V)=H^1_f(\bD_\fq)=H^1(\bD_{\fq})$ and 
$H^1_f(F_\fq,V^*(1))=H^1_f(\bD_{\fq}^{\perp})=\break =H^1(\bD_{\fq}^{\perp})$ 
for all $\fq \in S_p.$ 

ii) $H^1_f(V)\simeq H^1(V,\bD)$ and $H^1_f(V^*(1))\simeq H^1(V^*(1),\bD^{\perp}).$
\end{myproposition}
\begin{proof} i) The first statement follows from {\bf N2)} and 
Proposition~\ref{proposition properties of splitting submodules} iii).  

ii) Note that by i) 
\[
\R^1\Gamma (F_\fq,V,\bD)=\begin{cases} H^1_f(F_\fq,V), &\text{\rm if $\fq\in \Sigma_p$,}
\\
H^1(\bD_\fq), &\text{\rm if $\fq\in S_p$.}
\end{cases}
\]
By definition, the group $H^1(V,\bD)$ is the kernel 
of the morphism
\[
H^1_S(V)\bigoplus \left ( \underset{\fq\in \Sigma_p}\bigoplus H^1_f(F_\fq,V)\right )
\bigoplus \left (\underset{\fq\in S_p}\bigoplus H^1(\bD_\fq)\right )
\rightarrow \underset{\fq\in S}\bigoplus H^1(F_\fq,V)
\]
given by
\[
([x],[y_\fq]_{\fq\in S})\mapsto ([x_\fq]-g_\fq (y_\fq))_{\fq\in S}, 
\qquad [x_\fq]=\res_\fq ([x]),
\]
where $g_\fq$ denotes the canonical inclusion 
$H^1_f(F_\fq,V)\rightarrow H^1(F_\fq,V)$ if $\fq\in \Sigma_p$ and the map $H^1(\bD_\fq)\rightarrow  H^1(F_\fq,V)$ if $\fq\in S_p.$  In the both cases, $g_\fq$ is injective and, in addition, for each $\fq\in S_p$ we have  
$H^1(\bD_\fq)= H^1_f(F_\fq,V)$ 
by i). 
This implies that $H^1(V,\bD)=H^1_f(V).$ The same argument 
shows that $H^1(V^*(1),\bD^{\perp})=H^1_f(V^*(1)).$
 \end{proof}

\begin{mytheorem} 
\label{theorem comparision h^spl and h^norm}
Let $V$ be a  $p$-adic representation 
such that $V_\fq$ is potentially semistable for each  $\fq\in S_p,$ and let  
$(D_\fq)_{\fq \in S_p}$ be a family of splitting submodules. 
Assume that $(V,D)$ satisfies the conditions 
{\bf S)} and {\bf N2)}. Then 
\[
h^{\norm}_{V,D}=h^{\spl}_{V,D}.
\]
\end{mytheorem}
\begin{proof}
First note, that in our case the element $[\widetilde x_\fq^+],$ defined 
in Section~\ref{subsubsection definition of h^spl}, coincides with $[x_\fq^+].$
 Comparing the definitions of $h^{\norm}_{V,D}$ and 
$h^{\spl}_{V,D}$ we see that it is enough to show that 
$\ell_\fq \left (\pr_{\fq,y}([x_{\fq,y}^{\Iw}])-s_{\fq,y}([x_\fq^+])\right )=0$ for all $\fq\in S_p.$
The splitting $s_{y,\fq}$ of the exact sequence 
\[
0\rightarrow H^1_f(F_\fq,E(1)) \rightarrow H^1_f(\bD_{\fq,y})\rightarrow 
H^1(\bD_\fq)\rightarrow 0
\] (see (\ref{splitting extension of cohomology at v dividing p})) gives an isomorphism 
\[
H^1_{\Iw}(\bD_{\fq,y})_{\Gamma_\fq^0}\simeq H^1_{\Iw}(\bD_\fq)_{\Gamma_\fq^0}
\oplus H^1_{\Iw}(\CR_{F_\fq,E}(\chi))_{\Gamma_\fq^0}\simeq 
H^1(\bD_\fq)\oplus H^1_{\Iw}(F_\fq, E(1))_{\Gamma_\fq^0}.
\]
Since 
$
\pi_{\bD,\fq} \left (\pr_{\fq,y}([x_{\fq,y}^{\Iw}])-s_{\fq,y}([x_\fq^+]) \right )=0,
$
from this decomposition it follows that 
\[
\pr_{\fq,y}([w_\fq])-s_{\fq,y}([x_\fq^+])\in H^1_{\Iw}(F_\fq, E(1))_{\Gamma_\fq^0}=\ker (\ell_\fq),
\]
and the theorem is proved. 
\end{proof}

\begin{mycorollary} 
\label{corollary comparision pairings}
If $(V,D)$ satisfies the conditions {\bf S)} and {\bf N2)}, 
then the height pairings $h_{V,D,1}^\sel,$ $h^{\norm}_{V,D}$ and $h^{\spl}_{V,D}$ 
coincide.
\end{mycorollary}
\begin{proof}
This follows from Theorems~\ref{theorem comparision sel and norm heights} and 
\ref{theorem comparision h^spl and h^norm}. 
\end{proof}

\section{$p$-adic height pairings IV: extended Selmer groups}
\label{section extended selmer}
\subsection{Extended Selmer groups}
\label{subsection extended Selmer groups}
\subsubsection{} Let $F=\Q.$  Let $V$ be a $p$-adic representation of $G_{\Q,S}$ that is 
potentially semistable at $p.$ We fix a splitting submodule $D$ of $V.$ 
In Section~\ref{subsection filtration}, we associated to $D$ a canonical 
filtration $\left (F_i\Ddagrig (V)\right )_{-2\leqslant i \leqslant 2}.$
Recall that $F_0\Ddagrig (V)=\bD.$ We maintain the notation of  Section~\ref{subsection filtration} and set $\bM_0=\bD/F_{-1}\Ddagrig (V),$ 
$\bM_1=F_1\Ddagrig (V)/\bD$ and  $\W=F_1\Ddagrig (V)/F_{-1}\Ddagrig (V).$ 
The exact sequence 
\[
0\rightarrow \bM_0 \rightarrow \W \rightarrow \bM_1 \rightarrow 0
\]
induces the coboundary map $\delta_0\,:\,H^0(\bM_1)\rightarrow H^1(\bM_0).$
In this section, we generalize the construction of the height pairing 
$h^{\norm}_{V,D}$
to the case when  $V$ satisfies the conditions {\bf F1-2)} of 
Subsection ~\ref{subsection filtration}, namely 
 
{\bf{F1)}}  For all $i\in\Z$ 
\[
\CDpst (\Ddagrig (V)/F_1\Ddagrig (V))^{\Ph=p^i}= 
\CDpst (F_{-1}\Ddagrig (V))^{\Ph=p^i}=0.
\]

{\bf F2)}  The composed maps 
\begin{align*}
&\delta_{0,c}\,\,:\,\,H^0(\bM_1)\xrightarrow{\delta_0} H^1(\bM_0) \xrightarrow{\pr_c} H^1_c(\bM_0),
\\
&\delta_{0,f}\,\,:\,\,H^0(\bM_1)\xrightarrow{\delta_0} H^1(\bM_0) \xrightarrow{\pr_f} H^1_f(\bM_0),
%&H^0(\bM_0^*(\chi))\xrightarrow{\delta^*_0} H^1(\bM_1^*(\chi)) \xrightarrow{\pr_c} %H^1_c(\bM_1^*(\chi)),
\end{align*}
where the second arrows denote the canonical projections, are isomorphisms. 
Note, that if $V$ satisfies the conditions {\bf N1-2)} of Section~\ref{section universal norms},
we have $\bM_0=\bM_1=0.$

\subsubsection{} Define a bilinear map
\[
\left < \,\,,\,\,\right >_{\bD,f}\,:\,H^1_f(\bM_0)\times H^1_f(\bM_1^*(\chi)) \rightarrow E
\]
as the composition
\begin{multline*} 
H^1_f(\bM_0)\times H^1_f(\bM_1^*(\chi)) \xrightarrow{(\delta_{0,f}^{-1},\id)} 
H^0(\bM_1)\times H^1_f(\bM_1^*(\chi)) \xrightarrow{\cup}\\
H^1(\CR_{\Qp,E}(\chi)) \xrightarrow{\ell_{\Qp}} E.
\end{multline*}

\begin{mylemma}
\label{lemma computation of <,>f}
For all $x\in H^1_f(\bM_0)$ and $y\in H^1_f(\bM_1^*(\chi))$ we have
\[
\left <x,y\right >_{\bD,f}=-[i_{\bM_1^*(\chi),f}^{-1}(y), \delta_{0,f}^{-1}(x)]_{\bM_1},
\]
where $[\,\,,\,\,]_{\bM_0}\,\,:\,\,\CDcris (\bM_1^*(\chi))\times \CDcris (\bM_1)\rightarrow E$ 
denotes the canonical duality and $i_{\bM_1^*(\chi),f}\,:\,\CDcris (\bM_1^*(\chi))\rightarrow 
H^1_f(\bM_1^*(\chi))$ is the isomorphism constructed in Proposition~\ref{proposition isoclinic modules}.
\end{mylemma}
\begin{proof} Recall that for each 
$z\in H^1(\CR_{\Qp,E}(\chi))$ we have  
$
\inv_p (w_p\cup z)= \ell_p(z),
$
where $w_p=(0,\log \chi (\g_{\Qp})).$
 Therefore, using Proposition~\ref{proposition isoclinic modules}, we obtain
\begin{align*}
&\left <x,y\right >_{\bD,f}= \ell_{\Qp}(\delta_{0,f}^{-1}(x)\cup y)=\inv_p(w_p\cup \delta_{0,f}^{-1}(x)\cup y)=\\
&=-\inv_p(i_{\bM_1,c}(\delta_{0,f}^{-1}(x))\cup y))=\\
&=-\inv_p(i_{\bM_1,c}(\delta_{0,f}^{-1}(x))\cup i_{\bM_1^*(\chi),f}\circ i_{\bM_1^*(\chi),f}^{-1}(y))=\\
&=-[i_{\bM_1^*(\chi),f}^{-1}(y), \delta_{0,f}^{-1}(x)]_{\bM_1}.
\end{align*}
\end{proof}

\subsubsection{} Let $\rho_{\bD,f}$ and$\rho_{\bD,c}$ denote the  composed maps 
\begin{align*}
&\rho_{\bD,f}\,:\,H^1(\bD) \rightarrow H^1(\bM_0)  \xrightarrow{\pr_{f}} H^1_f(\bM_0) 
,\\
&\rho_{\bD,c}\,:\,H^1(\bD) \rightarrow H^1(\bM_0)  \xrightarrow{\pr_{c}} H^1_c(\bM_0).
\end{align*}
Note that $H^0(\bM_1)=H^0(\bD').$

\begin{myproposition} 
\label{proposition computation of extended Selmer group}
Let $V$ be a $p$-adic representation of $G_{\Q,S}$ which is potentially
semistable at $p.$ Assume that the restriction of $V$ on the decomposition group at $p$ 
satisfies the conditions {\bf F1-2)} of Section~\ref{subsection filtration}. Then 

i) There exists an exact sequence 
\begin{equation}
\label{exact sequence for extended Selmer group}
0\rightarrow H^0(\bD') \rightarrow H^1(V,\bD) \rightarrow H^1_f(V) \rightarrow 0.
\end{equation}

ii) The map 
\begin{align*}
&\spl_{V,\bD} \,:\,H^1(V,\bD) \rightarrow H^0(\bD'),\\
&\left [(x, (x_\fq^+),(\lambda_\fq))\right ]\mapsto \delta_{0,c}^{-1}\circ \rho_{\bD,c}
\left ( [x_p^+]\right )
\end{align*}
defines a canonical splitting of (\ref{exact sequence for extended Selmer group}).
\end{myproposition}
\begin{proof} The first statement is proved in \cite{Ben14}, Proposition 11.
It follows directly from the definition of Selmer complexes and the 
exact sequence (\ref{exact sequence H^1_f(Q_p,V)}). 
The second statement follows immediately from the definition of $\spl_{V,\bD}.$
\end{proof}
\noindent
From Proposition~\ref{proposition computation of extended Selmer group} it follows that
we have a canonical decomposition
\begin{equation*}
H^1(V,\bD) \simeq H^1_f(V)\oplus H^0(\bD')
\end{equation*}
and we denote by $j_{V,\bD} \,:\,H^1_f(V)\rightarrow H^1(V,\bD)$ the resulting injection.

\subsection{The pairing $h^{\norm}_{V,D}$ for extended Selmer groups}
\subsubsection{}
We keep previous notation and conventions. Let $[y]\in H^1_f(V^*(1))$ and let $Y_y$ denote 
the associated extention (\ref{extension Y_y}). As before, we denote by $\bD_y$ the inverse image of $\bD$ in $\Ddagrig (Y_y^*(1)_p).$ 
By Proposition~\ref{properties of filtration D_i} iii), the representation
$V_p$ satisfies the condition {\bf S)}, and therefore 
the exact sequence  (\ref{splitting extension of ph-Gamma-modules})
have a canonical splitting $s_{\bD,y}.$ Thus, we have the following diagram 
which can be seen as an analog of  the diagram (\ref{diagram universal norms})
in our situation
$$
\xymatrix{
0\ar[d] & 0 \ar[d] &\\
\mathcal{H}(\Gamma_{\Qp}^0) \otimes_{\Lambda_{E,\Qp}}\Hi^1(\Qp,E(1)) \ar[r] \ar[d] & H^1(\Qp, E(1)) \ar[d] \ar[r]^(.7){\ell_{\Qp}} &E \\
\Hi^1(\bD_{y})\ar[d]^{\pi_{\bD}^{\Iw}}\ar[r]^{\pr_{\bD,y}} &H^1(\bD_{y})\ar[d]^{\pi_{\bD}} &H^0(\bD') \ar[d]^{=}
\ar@{_{(}->}[l]_{\partial_0}
\\
\Hi^1(\bD)\ar[d]\ar[r]^{\pr_{\bD}} &H^1(\bD) \ar[d]&H^0(\bD') 
\ar@{_{(}->}[l]_{\partial_0}
\\
0 & 0. &
}
$$
In the diagram (\ref{diagram construction of h^norm}), the maps $g_v$ and $g_{v,y}$
are no more injective and we consider the following diagram 
with exact rows and columns
\begin{equation*}
\xymatrix{
& &0 \ar[d] &0 \ar[d] &\\
& &H^0(\bD') \ar[r]^{=} \ar[d]^{\partial_0} 
&H^0(\bD') \ar[d]^{\partial_0}  & \\
0\ar[r] &H^1(\Qp, E(1)) \ar[d]^{=} \ar[r]&H^1(\bD_{y})
\ar[r]^{\pi_{\bD}} \ar[d]^{g_{p,y}} &H^1(\bD) \ar[r] \ar[d]^{g_p} &0\\
0 \ar[r]& H^1(\Qp,E(1)) \ar[r] &H^1 (\Qp,Y_y^*(1)) \ar[r]^{\pi_p} &H^1(\Qp,V).
&
}
\end{equation*}
By Proposition~\ref{properties of filtration D_i} v), 
$\mathrm{Im} (g_p) =H^1_f(\Qp,V).$ 
%and the dimension argument shows that
%$H^1_f(\Qp,Y_y^*(1))\subset \mathrm{Im}(g_{p,y}).$
Let $[x]\in H^1_f(V)$ and let $\left [(x, (\widetilde x_\fq^+),(\widetilde \lambda_\fq))\right ]=j_{V,\bD} ([x]).$ 
Then $[\widetilde x_p^+]$ is the unique element of $H^1_f(\bD)$ such that 
$g_p([\widetilde x_p^+])=[x_p].$ 
Let 
\[
\widehat{[x]}\in \ker \left ( H^1_S(Y_y^*(1)) \rightarrow 
\frac{H^1 (\Q_\fq,Y_y^*(1))}{H^1_f(\Q_\fq, Y_y^*(1))} \right )
\]
 be an arbitrary lift of $[x].$
(Note, that by  Lemma~\ref{lemma splitting for v prime to p},
we can even take $\widehat{[x]}\in H^1_f(Y_y^*(1)).$)
By the five lemma  there exists a unique 
$[\widehat{x_p^+}]\in H^1(\bD_y)$ such that 
$g_{p,y}([\widehat{x_p^+}])= f_p(\widehat{[x]})$ and
$\pi_{\bD}([\widehat{x_p^+}])=[\widetilde x_p^+].$ On the other hand, from Proposition~\ref{proposition about decomposition of H^1(D)} it follows that there exist $\left [x_{p,y}^{\Iw}\right ]\in H^1_{\Iw}(\bD_y)$ and $[t_p]\in H^0(\bD')$ such that 
\begin{equation}
\label{choose of lifts equation}
[\widetilde x_p^+]+\partial_0([t_p])=\pr_{\bD}\circ \pi_{\bD}^{\Iw}\left (\left [x_{p,y}^{\Iw}\right ]\right ).
\end{equation}
Set 
\begin{equation}
\label{definition of u in extended selmer}
[u_p]=\pr_{\bD,y}\left (\left [x_{p,y}^{\Iw}\right ]\right )-\partial_0 (t_p)-
 [\widehat{x_p^+} ].
\end{equation}
Then $[u_p]\in H^1(\Qp,E(1)).$ 

\begin{definition} Let $V$ be a $p$-adic representation satisfying the conditions {\bf F1-2)}.  We define the height pairing 
\[
h^{\norm}_{V,D}\,:\,H^1_f(V)\times H^1_f(V^*(1))\rightarrow E
\]
by
\begin{equation*}
h^{\norm}_{V,D}([x],[y])=\ell_{\Q_p} ([u_p]).
\end{equation*} 
\end{definition}
\noindent
It is easy to see that $h^{\norm}_{V,D}([x],[y])$ does not depend on the choice of 
the lift $\left [x_{p,y}^{\Iw}\right ].$
The following result generalizes Theorem 11.4.6 of \cite{Ne06}.

\begin{mytheorem} 
\label{theorem comparision of heights for extended selmer} 
Let $V$ be a $p$-adic representation of $G_{\Q,S}$ that is 
potentially semistable at $p$ and satisfies the conditions {\bf F1-2)}. Then 

i) $h^{\norm}_{V,D}=h_{V,D}^{\spl};$

ii) For all 
$[x^\sel]=[(x,(x_\fq^+), (\lambda_\fq))]\in H^1(V,\bD)$ and 
$[y^\sel]=[(y,(y_\fq^+), (\mu_\fq))]\in H^1(V^*(1),\bD^{\perp})$ we have 
\begin{equation*}
h_{V,D}^{\sel}([x^\sel],[y^\sel])=h_{V,D}^{\norm}([x],[y])+
\left <\rho_{\bD,f}([x_p^+]), \rho_{\bD^{\perp},f}(y_p^+)\right >_{\bD,f}.
\end{equation*}
\end{mytheorem}
\begin{proof} 
i) Recall that in the definition of $h_{V,D}^{\norm}$ we can take 
$\widehat{[x]}\in H^1_f(Y_y^*(1)).$ Comparing the definitions of 
$h_{V,D}^{\norm}$ and $h_{V,D}^{\spl},$ we see that it is enough to
prove that
\begin{equation*}
[u_p]- \left (s_{y,p}([\widetilde x_p^+])- [\widehat x_p]\right )\in \ker (\ell_{\Qp}),
\end{equation*}
where $[u_p]$ is defined by (\ref{definition of u in extended selmer}) and
$s_{y,p}$ denotes the splitting (\ref{splitting extension of cohomology at v dividing p}).
Since the restriction of $g_{p,y}$ on $H^1(\Qp, E(1))$ is the identity map, 
we have
\[
[u_p]=g_{p,y}([u_p])=g_{p,y}([x_{p,y}^\Iw])-[\widehat x_p],
\]
and it is enough to check that
\begin{equation}
\label{formula comparision norms and spl}
g_{p,y}([x_{p,y}^\Iw])-g_{p,y}\circ s_{y,p}([\widetilde x_p^+])
\in \ker (\ell_{\Qp}).
\end{equation}
First remark that the canonical splitting (\ref{splitting extension of ph-Gamma-modules}) induces  splittings $s_{p,y}^{\Iw}$ and $s_{p,y}$ in the diagram 
\begin{equation}
\xymatrix{
0 \ar[r] &H^1_\Iw(\CR_{\Qp,E}(\chi)) \ar[r] \ar[d] &H^1_\Iw(\bD_y) \ar[r] \ar[d]^{\pr_{\bD,y}} &H^1_\Iw(\bD) \ar[r]
\ar@<-1ex>@{.>}[l]_{s_{p,y}^{\Iw}}\ar[d]^{\pr_{\bD}} &0\\
0 \ar[r] &H^1(\Qp, E(1)) \ar[r] &H^1(\bD_y) \ar[r] 
&H^1(\bD) \ar[r] \ar@<-1ex>@{.>}[l]_{s_{p,y}}
&0.}
\nonumber
\end{equation}
Write $[x_{p,y}^\Iw]$ in the form 
\[
[x_{p,y}^\Iw]=s_{p,y}^{\Iw}(a^{\Iw})+b^{\Iw}, \qquad 
a^{\Iw}\in H^1_\Iw(\bD), \quad b^{\Iw}\in H^1_\Iw(\CR_{\Qp,E}(\chi)). 
\]
By the definition of $[x_{p,y}^\Iw],$ we have
\[
\pr_{\bD,y}([x_{p,y}^\Iw])=s_{y,p}(a)+b,
\]
where $b\in \ker (\ell_{\Qp})=H^1(\Qp,E(1))_{\Gamma_{\Qp}^0}$ and 
\[
a=\partial ([t_p])+ s_{p,y}([\widetilde x_p^+])\in H^1(\bD).
\]
Since $g_{p,y}(s_{y,p}(\partial_0([t_p]))=0,$ we have 
\begin{equation*}
g_{p,y}(\pr_{\bD,y}([x_{p,y}^\Iw]))=b+g_{p,y}(s_{y,p}(a))=
b+g_{p,y}(s_{y,p}([\widetilde x_p^+])),
\end{equation*}
and (\ref{formula comparision norms and spl}) is checked.

ii) The proof is the same as that of Theorem 11.4.6 of \cite{Ne06} with 
obvious modifications.  First note that $j_{V,\bD}([x^\sel])=(x,(\widetilde x_\fq^+), (\widetilde \lambda_\fq)),$ where  
\[
\widetilde x_p^+=x_p^+-\partial_0\circ \left (\delta_{0,c}^{-1}\circ \rho_{\bD,c}\left ([x_p^+]\right )\right ).
\]
Set 
\begin{equation}
\label{definition of t_p}
[t_p]= -\delta_{0,f}^{-1}\circ\rho_{\bD,f}([x_p^+]).
\end{equation}
 Then 
\[
\rho_{\bD,f}([\widetilde x_p^+])+\rho_{\bD,f} (\partial_0([t_p]))=
\rho_{\bD,f}([\widetilde x_p^+])+\delta_{0,f}([t_p])=0.
\]
Thus, the image of $[\widetilde x_p^+]+\partial_0([t_p])$ under 
the projection $H^1(\bD)\rightarrow H^1(\bM_0)$ lies in $H^1_c(\bM_0).$
Now, from the diagram (\ref{diagram 1 proof proposition about decomposition of H^1(D)})
we obtain that $[\widetilde x_p^+]+\partial_0([t_p])\in H^1_{\Iw}(\bD)_{\Gamma_{\Qp}^0},$
and there exists  $\left [x_{p,y}^{\Iw}\right ]\in H^1_{\Iw}(\bD_y)$ which  satisfies (\ref{choose of lifts equation}) with $[t_p]$ given by (\ref{definition of t_p}). 

It is not difficult to check that the statement of Lemma~\ref{first lemma comparision sel and norm heights}
holds if we replace $\widehat b_v$ by
\[
\widehat b_p=w_p\cup (u_p+\partial_0(t_p)).
\]
Since $g_{p,y}(\partial_0([t_p]))=0,$  there exists $\widetilde t_p\in \Ddagrig (Y_y^*(1)_p)$ such that
$\widetilde t_p\mapsto t_p$ under the projection $\Ddagrig (Y_y^*(1)_p)\rightarrow \bD_y'$ 
and  
\[g_{p,y}(\partial_0(t_0))=d_0(\widetilde t_p) =((\Ph-1)(\widetilde t_p), (\g_{\Qp}-1) (\widetilde t_p)).
\]
Therefore 
\begin{align*}
&g_{p,y}(\widehat b_p)=w_p\cup u_p +w_p\cup g_{p,y}(d_0\widetilde t_p),\\
&g_p(b_p)\cup \mu_p=w_p \cup g_{p}(d\widetilde t_p)\cup \mu_p.
\end{align*}
Since the projection of $\widehat y$ on $E$ is $1,$ we have $w_p\cup u_p\cup f^{\perp}_p(\widehat y)=
w_p\cup u_p.$ Thus,
\begin{multline}
\label{computation proof of extended height}
\inv_p(g_{p,y}(\widehat b_p)\cup f^{\perp}_p(\widehat y)+g_p(b_p)\cup\mu_p)=\\
=\ell_{\Qp}(u_p)+
\inv_p (w_p\cup g_{p,y}(d\widetilde t_p)\cup f^{\perp}_p(\widehat y) +w_p \cup 
g_p(\pi_p(d\widetilde t_p))\cup \mu_p)=\\
=\ell_{\Qp}(u_p)-\inv_p (w_p\cup g_{p,y}(\widetilde t_p)\cup df^{\perp}_p(\widehat y) +w_p \cup 
g_p(\widetilde t_p)\cup d\mu_p)=\\
=\ell_{\Qp}(u_p)-\inv_p (w_p\cup \widetilde t_p\cup (f_p^{\perp}(y)+d\mu_p))=\\
=\ell_{\Qp}(u_p)-\inv_p (w_p\cup \widetilde t_p\cup g_p(y_p^+))=
\ell_{\Qp}(u_p)-\inv_p (w_p\cup  t_p\cup g_p(y_p^+))=\\
\ell_{\Qp}(u_p)-\ell_{\Qp}( t_p\cup g_p(y_p^+)).
\end{multline}
Now we remark that 
$
\ell_{\Qp}( t_p\cup g_p(y_p^+)) =\ell_{\Qp} \bigl (t_p\cup \rho_{\bD^{\perp},f}(y_p^+)\bigr )
$
and, taking into account (\ref{definition of t_p}), we have 
\begin{multline}
\label{computation of the second term proof of extended height}
\ell_{\Qp}( t_p\cup g_p(y_p^+))=-\ell_{\Qp} \bigl (\delta_{0,f}^{-1}\circ\rho_{\bD,f}([x_p^+])\cup 
\rho_{\bD,f}([y_p^+]) \bigr )=\\
=-\left <\rho_{\bD,f}([x_p^+]), \rho_{\bD^{\perp},f}([y_p^+])\right >_{\bD,f}.
\end{multline}

The first formula  of Lemma~\ref{second lemma comparision sel and norm heights} still
holds in our case.
From (\ref{computation proof of extended height}), (\ref{computation of the second term proof of extended height}) and the definition of $h^{\norm}_{V,D}$ it follows that
\begin{multline*}
h^{\sel}_{V,D}([x_f],[y_f])= \inv_p(g_{p,y}(\widehat b_p)\cup f^{\perp}_p(\widehat y)+g_p(b_p)\cup\mu_p)=\\
\ell_{\Qp}(u_p)+\left <\rho_{\bD,f}([x_p^+]), \rho_{\bD^{\perp},f}([y_p^+])\right >_{\bD,f}=\\
=h^{\norm}_{V,D}([x],[y])+\left <\rho_{\bD,f}([x_p^+]), \rho_{\bD^{\perp},f}([y_p^+])\right >_{\bD,f}
\end{multline*}
and the theorem is proved.
\end{proof}
\bibliographystyle{style}

\end{document}